\documentclass[reqno, 12pt]{amsart}

\usepackage{amsfonts, amsthm, amsmath, amssymb, enumitem}

\usepackage{hyperref}
\usepackage{a4wide}

\RequirePackage{mathrsfs} \let\mathcal\mathscr

\DeclareRobustCommand{\SkipTocEntry}[5]{} 

\numberwithin{equation}{section}

\newtheorem{theorem}{Theorem}[section]
\newtheorem{lemma}[theorem]{Lemma}
\newtheorem{proposition}[theorem]{Proposition}
\newtheorem{corollary}[theorem]{Corollary}

\theoremstyle{definition}
\newtheorem*{ack}{Acknowledgements}
\newtheorem{rem}[theorem]{Remark}

\newtheorem*{rem*}{Remark}

\newtheorem*{examples}{Examples}
\newtheorem{definition}[theorem]{Definition}

\newcommand{\dsum}{\sideset{}{^{\prime}}{\sum}}
\newcommand{\dagsum}{\sideset{}{^{\dagger}}{\sum}}

\newcommand{\starsum}{\sideset{}{^{*}}{\sum}}

\renewcommand{\d}{\,\mathrm{d}}
\newcommand{\bc}{\boldsymbol{c}}
\newcommand{\bd}{\boldsymbol{d}}
\newcommand{\m}{\boldsymbol{m}}
\newcommand{\n}{\boldsymbol{n}}
\renewcommand{\u}{\boldsymbol{u}}
\renewcommand{\v}{\boldsymbol{v}}

\newcommand{\Q}{\boldsymbol{Q}}
\newcommand{\vphi}{\varphi}
\renewcommand{\rho}{\varrho}
\newcommand{\1}{\mathbf{1}}

\newcommand{\beps}{\boldsymbol{\varepsilon}}
\newcommand{\bphi}{\boldsymbol{\varphi}}

\newcommand{\fK}{\mathfrak{K}}

\newcommand{\mQ}{\mathcal{P}_{\flat}}

\newcommand{\0}{\mathbf{0}}

\newcommand{\ZZ}{\mathbb{Z}}
\newcommand{\NN}{\mathbb{N}}
\newcommand{\QQ}{\mathbb{Q}}
\newcommand{\RR}{\mathbb{R}}
\newcommand{\CC}{\mathbb{C}}

\DeclareMathOperator{\EE}{{\mathbb{E}}}

\renewcommand{\leq}{\leqslant}
\renewcommand{\geq}{\geqslant}
\renewcommand{\bar}{\overline}

\newcommand{\x}{\boldsymbol{x}}

\newcommand{\Mod}[1]{\;(\operatorname{mod}\,#1)}

\newcommand{\<}{\langle}
\renewcommand{\>}{\rangle}

\DeclareMathOperator{\vol}{vol}
\DeclareMathOperator{\lcm}{lcm}

\newcommand{\eps}{\varepsilon}

\begin{document}

\title{Linear correlations of multiplicative functions}
\author{Lilian Matthiesen}
\address{KTH\\
Department of Mathematics\\
10044 Stockholm\\
Sweden}
\email{lilian.matthiesen@math.kth.se}


\begin{abstract}
We prove a Green--Tao type theorem for multiplicative functions.
\end{abstract}

\maketitle

\tableofcontents

\section{Introduction}\label{s:introduction}
The purpose of this paper is to establish an asymptotic result for 
correlations of complex-valued multiplicative functions.
More precisely, if $h_1, \dots h_r : \NN \to \CC$ are multiplicative functions that belong to a class 
$\mathcal{F}$ that will be introduced soon (see Section \ref{ss:the-class-F}), then we wish to 
asymptotically evaluate expressions of the form
\begin{equation} \label{eq:aim}
 \sum_{\n \in \ZZ^s \cap N \fK} h_1( \vphi_1(\n)) \dots h_r(\vphi_r(\n)), \qquad (N \to \infty)
\end{equation}
where $s \geq 2$, where $\fK \subset \RR^s$ is a fixed bounded convex subset that defines the family
$N\fK = \{N\x \in \RR^s : \x \in \fK\}$ of homogeneously expanding regions,
and where $\vphi_1, \dots, \vphi_r \in \ZZ[u_1, \dots, u_s]$
are fixed linear polynomials in $s \geq 2$ variables with the property that the non-constant parts 
of any two of these polynomials are pairwise linearly independent.

There has been extensive work on establishing upper bounds on expressions of the form \eqref{eq:aim}, 
both for linear polynomials and for polynomials of higher degree.
In \cite{erdos}, Erd{\H{o}}s obtained a correct order upper bound on the sum $\sum_{n\leq N} d(P(n))$, where $d$ denotes
the divisor function and $P$ an irreducible polynomial. 
Wolke \cite{wolke} extended Erd{\H{o}}s's approach from \cite{erdos} to all non-negative multiplicative functions $h$ that satisfy
$h(p^k) \leq C_1 k^{C_2}$ at all prime powers $p^k$ and with fixed constants $C_1$ and $C_2$.
In the case of linear polynomials, Shiu \cite{shiu} extended this work to all non-negative multiplicative functions 
satisfying the bound $h(p^k) \leq A^k$ at all prime powers as well as $h(n) \ll_{\eps} n^{\eps}$ for all $n \in \NN$ and $\eps > 0$.
Nair \cite{nair} extended Shiu's result in two directions, allowing now polynomials of higher degree with integer coefficients and 
non-zero discriminant as well as sub-multiplicative functions, i.e.\ to functions $h$ satisfying $h(mn) \leq h(m)h(n)$ 
whenever $\gcd(m,n)=1$.
Nair's work was further generalised by Nair and Tenenbaum \cite{nair-tenen}, who replace the sub-multiplicative function $h$ 
of one variable by any non-negative functions $F$ of $r \geq 1$ variables which satisfies the following sub-multiplicativity-type 
condition:
$$
F(m_1n_1, \dots, m_rn_r) \leq \min(A^{\Omega(m)},Bm^{\eps})F(n_1, \dots,n_r), \quad (m=m_1 \dots m_r),
$$
for all $r$-tuples $\m, \n \in \NN^r$ such that $\gcd(m_i,n_i)=1$ for $1 \leq i \leq r$.
The problem that Nair and Tenenbaum \cite{nair-tenen} study is that of establishing upper bounds for
sums of the form $\sum_{x < n \leq x+y} F(|P_1(n)|, \dots, |P_r(n)|)$ for polynomials $P_i \in \ZZ[X]$ such that 
$P = \prod_{1 \leq i \leq r} P_i$ has no fixed prime factor.
Nair and Tenenbaum's result covers the particular case where $F$ is given by
$F(m_1, \dots, m_r) = h_1(m_1) \dots h_r(m_r)$ for multiplicative functions $h_i$.
This case corresponds to the setting of \eqref{eq:aim}.
In unpublished work, Daniel \cite{daniel} established a Nair--Tenenbaum result with bounds that are uniform in the discriminant of $P$.
Finally, Daniel's work has been improved on and extended by Henriot in \cite{henriot}.

Asymptotic results on expressions of the form \eqref{eq:aim} are known in many special cases:
Green and Tao \cite{GT-mobius} establish such results for the M\"obius function $\mu$.
Using the machinery from \cite{GT-linearprimes,GT-mobius}, the author proved asymptotic results 
for the divisor function $d$ in \cite{lmd} and, in \cite{lmr, lmr2}, for the function $r$ counting representations 
by sums of two squares (and generalisations thereof to representations by binary quadratic forms). 
Lachand \cite{lachand} considers the characteristic function of $x^{1/u}$-smooth numbers.

The problem of finding upper bounds or even asymptotics for \eqref{eq:aim} in the case where $s=1$ is significantly harder 
and includes the famous open problem of finding the correct asymptotic behaviour of
$D_h(N)=\sum_{n\leq N} d(n)d(n+h)d(n-h)$ as a function of $h$.
When averaging over $h \in \{1, \dots, H\}$ with $H=N$, this question takes the form \eqref{eq:aim} with $s=2$.
Browning~\cite{TDB} showed that the average $\sum_{h\leq H} D_h(N)$ can be asymptotically evaluated for smaller values of $H$, 
more precisely for $H \geq N^{3/4 + \eps}$.
Using spectral methods, Blomer \cite{blomer} proves an asymptotic formula for (a smooth version of)
$\sum_{h \leq H}\sum_{n\leq N} a(n)d(n+h)d(n-h)$, where $H \geq N^{1/3 + \eps}$ and where $a$ is an arbitrary complex-valued
arithmetic function.
Finally, in work improving on many previous results (see the references in \cite{MRT}), 
Matom\"aki, Radziwi{\l}{\l} and Tao~\cite{MRT} recently proved asymptotic formulae for expressions of the form
$\sum_{X < n  \leq 2X} f(n) g(n+h)$ valid for almost all $|h| \leq H$ with $H \leq X^{1 - \eps}$ and  
where each of $f$ and $g$ can be a higher divisor function $d_k$, $k \geq 2$, or the von Mangoldt function\footnote{As 
the authors point out in \cite{MRT}, their methods do in fact apply to functions $f$ and $g$ stemming from a larger class of multiplicative functions.}.

To describe the known asymptotic results for \eqref{eq:aim} that apply to larger classes of multiplicative functions, we recall that 
a bounded multiplicative function $h: \NN \to \CC$ is said to be \emph{pretentious}, if 
there exists $t_h \in \RR$ and a Dirichlet character $\chi_h$ such that 
\begin{equation}\label{eq:pret}
\sum_{p \text{ prime}} \frac{1 - \Re (h(p)\chi_h(p)p^{it_h})}{p}< \infty.
\end{equation}
Frantzikinakis and Host \cite{FH} obtained asymptotic results for \eqref{eq:aim} 
with $\mathfrak{K}=(0,1]^s$ that apply to \emph{all} bounded complex-valued multiplicative functions. 
In the case where their asymptotic formula for \eqref{eq:aim} carries a genuine main term,
all multiplicative functions $h_i$ have the property that $|N^{-1} \sum_{n \leq N} h_i(n)| \asymp 1$. 
By Hal\'asz's theorem \cite{Halasz}, the latter condition implies that all $h_i$ are pretentious.
Klurman \cite{Kl} recently succeeded in asymptotically evaluating correlations of the form
$\sum_{n \leq N} h_1(P_1(n)) \dots h_r(P_r(n))$
for bounded pretentious multiplicative functions $h_1, \dots, h_r$ and 
for arbitrary polynomials $P_1, \dots, P_r \in \ZZ[X]$.
In \cite{KM}, Klurman and Mangerel obtain explicit main and error terms for the asymptotic result from \cite{FH}.

Our aim here is to leave the pretentious setting and prove a general asymptotic result for \eqref{eq:aim} that applies 
to unbounded multiplicative functions such as $d$ or $r$, as well as to functions with small mean values. 
To illustrate the latter goal, let us mention two examples of functions that the result applies to.
Firstly, the result applies, yielding a genuine main term, when we let each $h_i$ in \eqref{eq:aim} be a function of the 
form $h_i: n \mapsto \delta^{\omega(n)}$ where $\delta \in (0,1)$ and $\omega(n) = \#\{p:p|n\}$.
Such a function $h_i$ has a small mean value in the sense that 
$N^{-1}\sum_{n \leq N} h_i(n) \asymp (\log N)^{-1 + \delta} = o(1)$ and
the main term in our result will carry the correct logarithmic factor to capture this behaviour. 
A second example of an admissible function of small mean is $h_i:n\mapsto b(n)$, 
where $b$ denotes the characteristic function of sums of two squares.
In this case it follows from Landau's work \cite{landau} that
$N^{-1}\sum_{n \leq N} h_i(n) \sim (\log N)^{-1/2}$.
To handle functions with small mean values, we prove an asymptotic formula with an error term that, instead of being $o(1)$,
reflects the order of the mean values $N^{-1}\sum_{n \leq N} |h_i(n)|$ for $1 \leq i \leq r$.
This allows us to obtain asymptotic results with genuine main term in, amongst others, the above two examples.
We proceed by introducing the precise classes of multiplicative functions our main result will apply to.

\subsection{The classes $\mathcal{F}$ and $\mathcal{F}^*$ of multiplicative functions} \label{ss:the-class-F}
Throughout this paper we write
$$
S_h(x)= 
\frac{1}{x} 
\sum_{\substack{1\leq n \leq x}} 
h(n), \qquad
S_h(x;q,a)= 
\frac{q}{x} 
\sum_{\substack{1\leq n \leq x \\ n\equiv a \Mod{q} }} 
h(n)$$
and
\begin{align} \label{eq:def-E}
 E_{h}(x;q) &= \frac{1}{\log x} \frac{q}{\phi( q )}
\prod_{p \leq x, p \nmid q} \left(1 + \frac{|h(p)|}{p} \right)
\end{align}
for any $q,a \in \ZZ$, $q \not=0$ and $x\geq 1$. 
The first two conditions in the following definition should be compared to the conditions appearing in the 
results by Shiu \cite{shiu}, Nair \cite{nair} and Nair--Tenenbaum \cite{nair-tenen} mentioned in the 
preceding discussion on upper bounds.
\begin{definition} \label{d:M}
Let $\mathcal{F}$ denote the class of multiplicative functions
$h: \NN \to \CC$ with the following properties: 
\begin{itemize}
 \item[(i)] there exists a constant $H \geq 1$, depending on $h$, such that 
$|h(p^k)| \leq H^k$ for all primes $p$ and all integers $k \geq 1$;
 \item[(ii)] $|h(n)| \ll_{\eps} n^{\eps}$ for all $n \in \NN$ and $\eps > 0$; 
 \item[(iii)] there exists a positive constant $\alpha_h$ such that 
$$
\frac{1}{x}
\sum_{p \leq x} |h(p)| \log p \geq \alpha_h 
$$
for all sufficiently large $x$; and
\item[(iv)] $h$ \emph{has a stable mean value in arithmetic progressions, i.e.:}
For every constant $C>0$, there exists a function 
$\varphi = \varphi_C$ with $\varphi(x) \to 0$ as $x\to \infty$ such the estimate
$$
S_{h}(x';q,A)
= S_{h}(x;q,A)
 + O\Big(\varphi(x) E_{h}(x;q)\Big)
$$
holds for all $x \geq 2$ and $x' \in (x(\log x)^{-C},x)$,
and for all progressions $A \Mod{q}$ with $\gcd(q,A)=1$ and with a modulus
$q \in [1, (\log x)^C )$ that is divisible by the primorial $\prod_{p<\log \log x}p$.
\end{itemize}
\end{definition}

Together with $\mathcal{F}$, we consider the following slightly larger class $\mathcal{F}^* \supset \mathcal{F}$ 
that contains $n^{it}$-twists of elements of $\mathcal{F}$. 
In other words, condition (iv) only needs to hold for an $n^{it}$-twist of $f$, 
but not necessarily for $f$ itself.

\begin{definition} \label{d:F}
 Let $\mathcal{F}^*$ denote the class of multiplicative functions
$h: \NN \to \CC$ that satisfy conditions (i), (ii) and (iii) of Definition \ref{d:M},
as well as the following variant of condition~(iv):
\begin{itemize}
\item[(iv)']
For every constant $C>0$ there exists a function $\varphi = \varphi_C: \RR_{>0} \to \RR_{\geq 0}$ 
with $\varphi(x) \to 0$ as $x\to \infty$ and, given $C>0$ and $x > 1$, there exists 
$t_x \in \RR$ with $|t_x| \leq 2 \log x$ such that the function $h^* = h^*_x : n \mapsto h(n) n^{-it_x}$
satisfies the estimate:
$$
S_{h^*}(x';q,A)
= S_{h^*}(x;q,A)
 + O\Big(\varphi(x) E_{h^*}(x;q)\Big)
$$
for all $x' \in (x(\log x)^{-C},x)$ and all progressions $A \Mod{q}$ with $\gcd(q,A)=1$ and with a modulus
$q \in [1, (\log x)^C )$ that is divisible by the primorial $\prod_{p<\log \log x}p$.
\end{itemize}
\end{definition}

The classes $\mathcal{F}$ and $\mathcal{F}^*$ of functions are closely related to the classes 
$\mathcal{F}_H$ and $\mathcal{F}_{H,n^{it}}$ studied in \cite{lmm}.
Here, we impose the additional assumption (ii) which states that the growth of the function $h$ is bounded like 
the growth of the divisor function.
Out of the four conditions above, the last one is perhaps the least intuitive one and the one that is most difficult to 
check in an application.
Conditions (iv) and (iv)' have been studied in detail in \cite[\S4]{lmm}, where we prove several sufficient 
conditions for them. These ought to be 
significantly easier to check in many applications as they either only involve the values of $h$ at primes 
or only require to bound the correlation of $h$ with certain Dirichlet characters.
To conclude our discussion of the functions classes relevant to this paper let us record some explicit examples of functions 
satisfying the abstract set of rules defining $\mathcal{F}$.

\begin{examples}
Elements of the class $\mathcal{F}$ include the following functions (cf.\ \cite[\S \S 4.3-4.2]{lmm}): 
\begin{enumerate}
 \item the general divisor functions
$d_k(n) = \1* \dots * \1(n) = \1^{*k}(n)$ for $k\geq 2$, 
 \item the function
$\frac{1}{4}r(n) = \frac{1}{4}\#\{(x,y) \in \ZZ^2: x^2+y^2 = n\}$ which, up to the factor $\frac{1}{4}$, 
counts representations as a sum of two squares,
 \item the characteristic function $b(n)$ of the set of sums of two squares,
 \item the function $n \mapsto \delta^{\omega(n)}$ belongs to $\mathcal{F}$ if
$\delta \in (0,2)$ and $\omega(n) = \#\{ p \text{ prime}: p|n\}$,
 \item the function $n \mapsto |\lambda_f (n)|$, where $\lambda_f(n)$ describes the normalised Fourier 
 coefficients of a primitive holomorphic cusp form.
\end{enumerate}
\end{examples}
  
\subsection{Main result and underlying method of proof}
The main result of this paper (Theorem \ref{t:main} below) establishes an asymptotic formula for \eqref{eq:aim} under the 
assumption that $h_1, \dots, h_r$ all belong to $\mathcal{F}^*$.
The proof of this result proceeds via Green and Tao's nilpotent Hardy--Littlewood method (see \cite{GT-linearprimes}), 
which is a method consisting of two main parts. 
One of them requires us to establish that a `$W$-tricked' version of any $h \in \mathcal{F}^*$ is orthogonal 
to arbitrary nilsequences.
The other part amounts to showing that this $W$-tricked version of $h$ has a majorant function 
which is pseudo-random in the sense of \cite{GT-linearprimes} and of the 
`correct' average order in a sense we will specify later.
The first task, namely that of finding a suitable $W$-trick and obtaining non-trivial estimates 
for the correlation of the $W$-tricked version of $h$ and nilsequences, 
has been established in \cite{lmm} for all $h \in \mathcal{F}^*$.
The second task, namely the construction of correct-order pseudo-random majorants, 
will be the main focus of this work.
The fact that we are seeking a majorant function for $h$ means that our work is closely related to the sequel
of papers studying upper bounds on expressions of the form \eqref{eq:main} 
that were referred to at the very start of this introduction.

\subsection{Overview}
This paper is organised as follows.
In Section \ref{s:main}, we give the precise statements of our main result and an easier special case of it.
Section \ref{s:corollaries} describes a reformulation of the main result in terms of short character sums, 
allowing one in special cases to deduce local-global principles.
As an application, we recover a result of Frantzikinakis--Host~\cite{FH} concerning the case where in 
\eqref{eq:aim} all the arithmetic functions $h_1, \dots, h_r$ are `pretentious'.
Sections \ref{s:minor-arc} and \ref{s:W-reduction} are concerned with the more technical parts of the proof
of the main result: 
Section \ref{s:minor-arc} (conditionally) proves a `$W$-tricked' version of the main result, in which 
each of the multiplicative functions $h_j$ from \eqref{eq:main} is replaced by a function of the form 
$n \mapsto h_j(W_j n + A_j)$ for a suitable integer $W_j$ and a reduced residue $A_j \Mod{W_j}$.
To prove this result, we preliminarily assume the existence of families of pseudorandom majorants for the $W$-tricked functions
$n \mapsto h_j(W_j n + A_j)$.
In Section \ref{s:W-reduction}, we then deduce the main theorem from its just established $W$-tricked version.

Sections \ref{s:majorants-intro}--\ref{s:correlation-condition} are independent from all preceding sections, 
apart from the definitions made in this introduction, and contain the main new input of this paper,
namely the construction of the required families of pseudorandom majorants for all ($W$-tricked versions of the) 
multiplicative functions from $\mathcal{F}^*$.
The construction itself takes place in Sections \ref{s:majorants-intro}--\ref{s:majorant}. 
Here, the main difficulty lies in establishing suitable majorant functions of the correct average order for bounded 
multiplicative function, see Section \ref{s:majorant}. 
In Section \ref{s:linearform} we recall and introduce all relevant concepts around the notion of families of 
pseudorandom majorants.
The task of checking that the constructed majorant functions do indeed give rise to a family of pseudorandom 
majorants is carried out in Sections \ref{s:linearform} and \ref{s:correlation-condition}.

Section \ref{s:corollary-proofs}, finally, contains the proofs of the results from Section \ref{s:corollaries}, and 
Section \ref{s:applications} discusses the application of our main result to the arithmetic functions
$h_j(n) = |\lambda_{f_j}(n)|$, where $\lambda_{f_j}(n)$ denotes the normalised Fourier coefficients of 
a primitive holomorphic cusp form $f_j$.

\section{Statement of main result} \label{s:main}
This section contains the precise statement of our main result, which we present subsequently to that of 
an easier, but important, special case.

Let $w:\RR_{>0} \to \RR_{>0}$ be any function such that
$$\frac{\log \log x}{\log \log \log x} < w(x) \leq \log \log x$$ 
for all sufficiently large $x$, and define
$W(x) = \prod_{p\leq w(x)} p$.
The asymptotic formula below features an integer multiple $\widetilde{W}(N)$ of $W(N)$ with the property
that for each $h_j \in \mathcal{F}^*$ appearing in the statement, the mean value 
$S_{h_j}(N;q\widetilde{W}(N),A)$ in progressions $A \Mod{q\widetilde{W}(N)}$  
shows a certain amount of regularity as $1 \leq q \leq (\log N)^E$ varies over small integers
and $A$ varies over reduced residues, i.e. $\gcd(A, q\widetilde{W}(N))=1$.
The type of regularity we require is that $S_{h_j}(N;q\widetilde{W}(N),A) \sim S_{h_j}(N;\widetilde{W}(N),A)$
for $A$ and $q$ as above.
The existence of such values of $\widetilde{W}(N)$ was established in \cite{lmm}
and any function of the form $n \mapsto h_j(\widetilde{W}(N) n + A)$ for $A$ with $\gcd(A,\widetilde{W}(N))= 1$ 
will be referred to as a \emph{$W$-tricked} version of $h_j$.
By working with such $W$-tricked versions of each $h_j$, one removes the potentially irregular contribution from small primes 
(i.e.\ primes dividing $\widetilde{W}(N)$), the contribution of which can then be handled separately.

We begin by stating the special case of the main result where $h_1, \dots, h_r \in \mathcal{F}$.
The restriction to $\mathcal{F}$ simplifies the asymptotic formula significantly. 
This version of the result applies to, amongst others, all the examples mentioned at the end of Section \ref{ss:the-class-F}.
\begin{theorem}\label{t:main-F}
Let $N>1$ be an integer parameter and let $r,s, L \geq 2$ and $B_0 \geq 0$ be fixed integers.
Suppose further that we are given the following data:
\begin{itemize}
 \item[(i)] Let $h_1, \dots, h_r \in \mathcal{F}$ be multiplicative functions and let $H>1$ and $\alpha>0$
be constants such that condition (i) and (iii) of Definition \ref{d:M} hold with the given value of $H$ and 
with $\alpha_h = \alpha$ for every $h \in \{h_1, \dots, h_r\}$.
 \item[(ii)] Let $\psi_1, \dots, \psi_r \in \ZZ[X_1, \dots X_s]$ be linear forms that are pairwise 
linearly independent over $\QQ$.
Suppose that the coefficients of the $\psi_j$ are all bounded in absolute value by $L$.
For each value of $N$, let $a_1(N), \dots, a_r(N) \in [-LN(\log N)^{-B_0},LN(\log N)^{-B_0}]$ be integers and 
define for each $j \in \{1, \dots, r\}$ the linear polynomial
$\vphi_j(\n) = \vphi_{j,N}(\n)=  \psi_j(\n) + a_j(N)$.
 \item[(iii)]
Suppose that $\fK \subset [-1,1]^s$ is convex, $\vol(\fK) > 0$, and that
$0 < \vphi_j(\n) \leq N$ for all $1\leq j \leq r$, all $\n \in N\fK$ and all sufficiently 
large $N$.
\end{itemize}
Then, given any $\eps > 0$, there are positive constants $B_1, B_2 = O_{r,s,B_0, H, L, \alpha, \eps}(1)$ and a function 
$\widetilde W: \RR_{>0} \to \NN$, depending on $h_1, \dots, h_r$ but not on the information from (ii) and (iii), such that 
\begin{enumerate}
\item[(1)] $\widetilde{W}(x)$ is divisible by $W(x)=\prod_{p < w(x)} p$ for all $x > 0$, 

\item[(2)] $\widetilde{W}$ satisfies the bound $\widetilde W(x) \leq (\log x)^{B_1}$ for all sufficiently large $x$, and 

\item[(3)] as $N \to \infty$, the following asymptotic formula holds uniformly for 
all $\vphi_1, \dots, \vphi_r$ as above and all $T \in [N(\log N)^{-B_0}, N]$:
\begin{align}\label{eq:main'}
\nonumber
&\frac{1}{\vol T\fK}
\sum_{\n \in \ZZ^s \cap T\fK} \prod_{i=1}^r 
h_i(\vphi_i(\n))
~= \\
\nonumber
&\sum_{\substack{w_1, \dots, w_r 
 \\ p|w_i \Rightarrow p| \widetilde W
 \\ w_i \leq (\log N)^{B_2}}}
\sum_{\substack{A_1,\dots,A_r \\ \in (\ZZ/\widetilde W \ZZ)^*}}
\Bigg(\prod_{i=1}^r 
h_i(w_i)S_{h_i}\Big(N;\widetilde W,A_i\Big)\Bigg)~
\beta_{\bphi}(w_1A_1, \dots, w_rA_r)\\
&\qquad \qquad  +
(\eps +o(1))\left( \frac{1}{(\log N)^r} 
\prod_{j=1}^r  \prod_{p \leq N} 
\left(1 + \frac{|h_i(p)|}{p} \right)
\right)
~,
\end{align}
where $\widetilde W=\widetilde W(N)$, $w = \lcm(w_1, \dots, w_r)$ and
\begin{equation} \label{eq:beta_phi}
\beta_{\bphi}(w_1A_1, \dots, w_rA_r)
=\frac{1}{(w\widetilde W)^s}
\sum_{\substack{\v \in \\ (\ZZ/w \widetilde W \ZZ)^s}}
\prod_{j=1}^r 
\1_{\vphi_j(\v) \equiv w_j A_j ~(w_j \widetilde W)}.
\end{equation}
\end{enumerate}
\end{theorem}

\begin{rem}[on removing the $\eps$]
 The dependence on $\eps$ is an artefact of the generality of the result.
 In cases where the right hand side of \eqref{eq:main'} can be reformulated as
 a closed expression that is independent of $\widetilde{W}$ and $B_2$, this dependence can be removed.
 Such a reformulation always exists if for each 
 $j \in \{1, \dots, r\}$ the set of primitive 
 characters $\chi$ for which $x^{-1} |\sum_{n \leq x} h_j(n) \chi(n)|$ is (close to) maximal does not 
 depend on $x$ as soon as $x$ is sufficiently large.
 (This assumption allows one to replace $S_{h_j}(N, \widetilde{W}, A_j)$ by a short character sum 
 involving a fixed set of characters.)
 We will describe this set-up in Theorem~\ref{t:main'}.
\end{rem}

\begin{rem}[on the parameter $T$]
 In applications it is often essential that the one can vary the cut-off parameter slightly while 
 preserving the shape of the main term as well as uniformity in the error term. 
 For this reason we introduced a second cut-off parameter $T$ that is closely related to the value of $N$ 
 which determines $\widetilde{W}(N)$.
\end{rem}

The full version of our main result extends the class of admissible functions to $\mathcal{F}^*$.
Compared with the statement of Theorem \ref{t:main-F}, part $(i)$ from the assumptions and 
the asymptotic formula \eqref{eq:main'} itself have to be adjusted. 
In order to simplify the asymptotic formula, we also slightly change the assumptions in part $(ii)$. 
This leads to:
\begin{theorem}[Main Theorem]  \label{t:main}
Let $N>1$ be an integer parameter, let $r,s, L \geq 2$ and $B_0 \geq 0$ be integers, let $\delta \in (0,1)$ be a parameter,
fix a value of $c \in (0,1)$ and let $C>1$ be a constant.
Suppose further that we are given the following data:
\begin{itemize}
 \item[(i)] Let $h_1, \dots, h_r \in \mathcal{F^*}$ be multiplicative functions and let $H>1$ and $\alpha>0$
be constants such that condition (i) and (iii) of Definition \ref{d:M} hold with the given value of $H$ and 
with $\alpha_h = \alpha$ for every $h \in \{h_1, \dots, h_r\}$.
For every $j \in \{1, \dots, r\}$, let $t_j = t_{j,N}$, $|t_j| \leq 2 \log N$, 
be such that the function $$h_j^*:n \mapsto h_j(n)n^{-it_j}$$
has the property described in Definition \ref{d:F} with $x=N$ and the given value of $C$.
 \item[(ii)] Let $\psi_1, \dots, \psi_r \in \ZZ[X_1, \dots X_s]$ be linear forms that are pairwise 
linearly independent over $\QQ$. 
Suppose that the coefficients of the $\psi_j$ are all bounded by $L$ in absolute value.
For each value of $N$, let $a_1(N), \dots, a_r(N) \in [-N^c,N^c]$ be integers and 
define the system of linear polynomials $\vphi_j(\n) =  \psi_j(\n) + a_j(N)$, $1 \leq j \leq r$.
 \item[(iii)]
Suppose that $\fK \subset [-1,1]^s$ is convex, $\vol(\fK) > 0$, and that
$\vphi_j(\n) \in (0, N]$ for all $1\leq j \leq r$, all $\n \in N\fK$ and all sufficiently 
large $N$.
\end{itemize}
Then there are positive constants $B_1, B_2 = O_{r,s,B_0,H,L,\alpha,\delta}(1)$, and a function 
$\widetilde W: \RR_{>0} \to \NN$, depending on $h_1, \dots, h_r$ but not on the information from (ii) and (iii), such that 
\begin{enumerate}
\item[(1)] $\widetilde{W}(x)$ is divisible by $W(x)=\prod_{p < w(x)} p$, 

\item[(2)] $\widetilde{W}$ satisfies the bound $\widetilde W(x) \leq (\log x)^{B_1}$ for all $x$, and 

\item[(3)] as $N \to \infty$, the following asymptotic formula holds uniformly 
for all $\vphi_1, \dots, \vphi_r$ as above and
all $T \in [N(\log N)^{-B_0}, N]$, provided $C$ (see assumption (i)) is sufficiently 
large with respect to $B_0, H, r, s, L, \alpha$ and $\delta$:
\begin{align}\label{eq:main}
&\frac{1}{\vol T\fK}
\sum_{\n \in \ZZ^s \cap T\fK} 
\prod_{i=1}^r 
h_i(\vphi_i(\n)) \\
\nonumber
&= \frac{T^{i(t_1 + \dots + t_r)}}{\vol \fK}
\int_{\fK} \vphi_1(\x)^{it_1} \dots \vphi_r(\x)^{it_r} \d \x \\
\nonumber
&\qquad \qquad
\times \sum_{\substack{w_1, \dots, w_r 
 \\ p|w_i \Rightarrow p| \widetilde W
 \\ w_i \leq (\log N)^{B_2}}}
\sum_{\substack{A_1,\dots,A_r \\ \in (\ZZ/\widetilde W \ZZ)^*}}
\Bigg(\prod_{j=1}^r 
 h_j^*(w_j) S_{h_j^*}\Big(N;\widetilde W,A_j\Big)\Bigg)~
 \beta_{\bphi}(w_1A_1, \dots, w_rA_r) \\
 \nonumber
& \quad  +
\left(\kappa(\delta) + o_{N\to\infty}(1)\right)
\left( \frac{1}{(\log N)^r} 
\prod_{j=1}^r  \prod_{p \leq N} 
\left(1 + \frac{|h_j(p)|}{p} \right)
\right)
~,
\end{align}
where $\kappa(\delta) \to 0$ as $\delta \to 0$, and where
$\widetilde W=\widetilde W(N)$, $w = \lcm(w_1, \dots, w_r)$ and
\begin{equation} \label{eq:beta_phi-0}
\beta_{\bphi}(w_1A_1, \dots, w_rA_r)
=\frac{1}{(w\widetilde W)^s}
\sum_{\substack{\v \in \\ (\ZZ/w \widetilde W \ZZ)^s}}
\prod_{j=1}^r 
\1_{\vphi_j(\v) \equiv w_j A_j ~(w_j \widetilde W)}.
\end{equation}
\end{enumerate}
The main term in \eqref{eq:main} can be subsumed in the error term as soon as for one 
$j \in \{1, \dots, r\}$ the sequence of 
functions $h_j^* = h_{j,N}^*:n \mapsto h_j(n)n^{-it_{j,N}}$ satisfies 
$|S_{h_{j,N}^*}(N)|= o_{N \to \infty}(S_{|h_j|}(N))$.
\end{theorem}

\section{A short character sum version of the main theorem and corollaries} \label{s:corollaries}
This section describes a reformulation of Theorem \ref{t:main} in terms of short character sums,
allowing one in special cases to reinterpret the main term in the asymptotic formula as a product 
over local factors.

The starting point for such a reformulation lies in the observation that for many multiplicative 
functions of interest, the asymptotic behaviour of
\begin{equation} \label{eq:character-sum-expansion}
S_{h}(T;\widetilde W(N), A)
= 
\frac{\widetilde W(N)}{\phi(\widetilde W(N))} 
\sum_{\chi \Mod{\widetilde W(N)}}
\chi(A)
\frac{1}{T}
\sum_{n\leq T} h(n) \bar \chi(n)
\end{equation}
is determined by only finitely many characters of bounded conductor,
allowing for the main term of the asymptotic formula \eqref{eq:main} to be simplified.
Before we consider this problem in general, let us state the special case 
which assumes that $S_{h^*}(T;\widetilde W, A)$ is determined by the trivial character modulo
$\widetilde W$ for each function $h = h_j$ appearing in the correlation.

\begin{corollary} \label{cor:chi_0}
Let $C>1$ and $h_1, \dots h_r$ be as in Theorem~\ref{t:main} and assume that, as $N \to \infty$,
$$
S_{h^*}(N; q, A)
= 
\frac{q}{\phi(q)} 
\frac{1}{N}
\sum_{n\leq N} h^*(n) \chi_0(n)
+ o(E_{h}(N;q)),
$$
for all $h \in \{h_1, \dots h_r\}$, all $q \leq (\log N)^C$ with $W(N)|q$, all reduced residues $A \Mod{q}$,
and where $\chi_0$ denotes the trivial character modulo $q$.
Assuming, in addition, that all the assumptions from Theorem~\ref{t:main} hold,
we have the following asymptotic formula:
\begin{align*}
\frac{1}{\vol N\fK}
\sum_{\n \in \ZZ^s \cap N\fK} 
\prod_{j=1}^r 
h_j(\vphi_j(\n)) 
&=
\beta_{\infty}(N) \prod_{p \leq N} \beta_p
+ o\Bigg(
\frac{1}{(\log N)^r} 
\prod_{j=1}^r  \prod_{p \leq N} 
\left(1 + \frac{|h_j(p)|}{p} \right)
\Bigg)
~, 
\end{align*}
as $N \to \infty$,
where
\begin{align} \label{eq:beta_P(B)-cor}
\beta_{p}&=
\Bigg(\prod_{i=1}^r \sum_{k\geq 0} \frac{|h_i(p^k)|}{p^k}\Bigg)^{-1} \\
\nonumber
&\qquad \times \sum_{\substack{ a_1, \dots, a_r \\ \in \NN_0}}
\Bigg( \prod_{j=1}^r h_j(p^{a_j}) \Bigg)
\Bigg( \lim_{m \to \infty} \frac{1}{p^{ms}} \sum_{\substack{\v \in \\ (\ZZ/p^m \ZZ)^s}}
\prod_{j=1}^r 
\Big(\1_{p^{a_{j}}|\vphi_{j}(\v)} - \1_{p^{a_{j}+1}|\vphi_{j}(\v)}\Big)
\Bigg)
\end{align}
and where
$$
\beta_{\infty}(N)
= \frac{N^{i(t_1 + \dots + t_r)}}{\vol \fK}
\int_{\fK} \vphi_1(\x)^{it_1} \dots \vphi_r(\x)^{it_r} \d \x
~ \prod_{j=1}^r S_{|h_j|}(N).
$$
Moreover, we have
$$S_{|h_j|}(N) \asymp \frac{1}{\log N} \prod_{p \leq N} \Big(1 + \frac{|h_j(p)|}{p} \Big)$$
and
$$\beta_{p} = 
\prod_{j=1}^r
\left(1+\frac{|h_j(p)|}{p} \right)^{-1}
\left(1+\frac{h_j(p)}{p} \right)
+ O_{r,H,L}(p^{-2}).$$
Thus, $\beta_{p} = 1 + O_{H,r}(p^{-2})$ and $\beta_{\infty}(N) = \prod_{j=1}^r S_{|h_j|}(N)$ if all the $h_j$ are non-negative.
\end{corollary}

As an example of a function for which \eqref{eq:character-sum-expansion} is not determined 
by $\chi_0$ but nonetheless by finitely many characters, we may consider the function  
$h(n) = \frac{1}{4}r(n) = \1 * \chi_{-1}(n)$, where $\chi_{-1}$
denotes the non-principal character modulo $4$ and where $r$ is the 
function that counts  representations by sums of two squares.

In situations where \eqref{eq:character-sum-expansion} is determined by only finitely many 
characters, our main theorem can be reformulated in terms of short 
character sums, and Theorem \ref{t:main'} below is such a reformulation.
The fact that the character sum can be truncated relies on the following 
consequence of the repulsion of characters phenomenon described in \cite{BGS} 
and refined in \cite{GS-book,GHS}.

\begin{proposition}[Set of characters] \label{p:character-set}
Suppose that $h: \NN \to \CC$ is multiplicative and satisfies conditions (i)--(iii) of Definition \ref{d:M}.
Let $t \in \RR$ and set $h^*: n \mapsto h(n)n^{-it}$.
Let $H\geq 1$ and $\alpha_{h}>0$ be such that part (i) and (iii) of Definition~\ref{d:M} hold, and
define a multiplicative function $h': \NN \to \CC$ by setting $h'(p) = h^*(p)/H$ at primes and
$h'(p^k) = 0$ if $k\geq 2$.
Further, let $\eps = \frac{1}{2}\min(1, \alpha_h/H)$ and define the integer $k=\lceil \eps^{-2} \rceil$. 
Let $C >0$ be a parameter, consider for any $x > 1$ the set of primitive characters of conductor at most 
$(\log x)^{C}$, and enumerate them as $\chi_1, \chi_2, \dots$ in such way that the averages 
$|\frac{1}{x}\sum_{n \leq x} \chi_i(n) h'(n)|$ are in non-increasing order as $i$ increases. 
We let $\mathcal{E}(x,C) := \{ \chi_1, \chi_2, \dots, \chi_{k}\}$ denote the set of the first $k$ characters.

Define for any given value of $N>1$ the set $\mathcal{E}_N = \bigcup_{1 \leq j \leq k'} \mathcal{E}(N^{1/2^j},C),$ 
where $k' = \lceil \log_2(4H) \rceil$.
If $\mathcal{E}_N(q)$ denotes the set of characters modulo $q$ induced from elements of $\mathcal{E}_N$,
then
\begin{align} \label{eq:GS-corollary}
S_{h^*}(y,q,A) = 
\frac{q}{\phi(q)} 
\sum_{\chi \in \mathcal{E}_N(q)}
\chi(A)
\frac{1}{y}
\sum_{n\leq y} h^*(n) \bar \chi(n)
+o(E_{h}(N;q))
\end{align}
uniformly for all $q \leq (\frac{1}{8H}\log N)^{C}$ with $q|W(N)$, all
reduced residues $A \Mod{q}$ and $N^{1/2} \leq y \leq N$.

Moreover, we have $h \in \mathcal{F}^*$, provided for every $C>0$ there exists a function 
$\varphi_C: \NN \to \RR_{\geq 0}$ with $\varphi_C(x) \to 0$ as $x \to 0$
such that
\begin{align} \label{eq:main'-assumption}
\frac{1}{T}\sum_{n \leq T} \chi(n) h^*(n)
= \frac{1}{N}\sum_{n \leq N} \chi(n) h^*(n)
+ O\bigg( \frac{\varphi_C(N)}{\log N} \exp\bigg(\sum_{p\leq N} \frac{|\chi(p) h(p)|}{p}\bigg) \bigg)
\end{align}
for all $N>1$, $T \in (N(\log N)^{-C},N]$ and for every $\chi \in \mathcal{E}_N(q)$ with 
$q \in (1, (\log N)^C]$ and $W(N)|q$.
\end{proposition}

\begin{rem}
 Note that the statement above simplifies if $h'$ is such that the sets
 $\mathcal{E}(x,C)$ are independent of $x$ as soon as $x$ is sufficiently large.
 In this case $\mathcal{E}_N$ is just given by the fixed set 
 $\mathcal{E}(x,C)=\{\chi_1, \dots, \chi_k \}$ of the first $k$ characters in the sequence
 for any $C$ and any sufficiently large $x$.
\end{rem}

The finite set of characters picked out by the proposition above may still be larger than strictly necessary
in order for \eqref{eq:GS-corollary} to hold, 
and results by Elliott~\cite{elliott}, Tenenbaum~\cite{tenen} or Mangerel~\cite{mangerel} relating the mean value of a multiplicative 
function $h$ to that of $|h|$ can be used to further restrict the set $\mathcal{E}_N(q)$.
To illustrate the character of these comparison results, we include the following qualitative lemma,
which is a straightforward consequence of Elliott \cite[Theorem 2 and 4]{elliott}.
In fact, the first part of this lemma follows already from earlier work of 
Elliott and Kish \cite[Lemma 21]{elliott-kish}.
\begin{lemma}[Elliott, Elliott--Kish] \label{l:elliott}
Suppose $h$ is a multiplicative function that satisfies conditions (i) and (iii) 
from Definition \ref{d:M} and $\sum_{p \leq H} \sum_{k \geq 2} |h(p^k)|p^{-k} < \infty$.
Then
\begin{equation} \label{eq:elliott-kish}
 S_{|h|}(x) 
\asymp \frac{1}{\log x} \exp\bigg(\sum_{p \leq x} \frac{|h(p)|}{p}\bigg).
\end{equation}
Moreover, $|S_{h}(x)| =o(S_{|h|}(x))$ unless there exists $t \in \RR$ such that 
\begin{equation}\label{eq:elliott}
\sum_{p \text{ prime}} \frac{|h(p)| - \Re (h(p)p^{it})}{p}< \infty.
\end{equation}
More precisely, if there exists $t$ as above, then 
\begin{equation}\label{eq:elliott-precise}
S_{h}(x) 
= S_{|h|}(x) \frac{x^{-it}}{1-it} 
\prod_{p\leq x} \left(
\frac{1+h(p)p^{it-1}+\dots}{1 + |h(p)| p^{-1} + \dots}\right)
 + o(S_{|h|}(x)).
\end{equation}
\end{lemma}
Under stronger assumptions on the behaviour of
$$
\sum_{y_1 < p \leq y_2} p^{-1}(|h(p)| - \Re (h(p)p^{it})),
\quad \text{or} \quad \sum_{y_1 < p \leq y_2} p^{-1}(|h(p)| - \Re (h(p) \chi(p) p^{it})),
$$
for certain ranges of $y_1, y_2$, Tenenbaum's Th\'eor\`eme 1.3 from \cite{tenen} yields an asymptotic formula 
for $S_{h}(x)$ (or $S_{h \chi}(x)$) in terms of $S_{|h|}(x)$ with \emph{explicit} error terms.
This result allows one to consider twists of $h$ by characters of conductor depending on $x$ and can therefore
be used to check the conditions of the following result in suitable applications.

\begin{theorem}\label{t:main'}
Let $N$, $T$ and $\vphi_1, \dots, \vphi_r$ be as in Theorem \ref{t:main}.
For each $j \in \{1, \dots, r\}$, let $h_j: \NN \to \CC$ be a multiplicative function that satisfies 
conditions (i)-(iii) of Definition~\ref{d:M}, let $t_j \in \RR$, and suppose that 
$h_j^*: n \mapsto h_j(n)n^{-it_j}$ satisfies
condition \eqref{eq:main'-assumption} from Proposition~\ref{p:character-set}. 
(In particular, $h_j \in \mathcal{F}^*$.)
Let $\widetilde{W}: \NN \to \NN$ be as in Theorem \ref{t:main}.
For each $j \in \{1, \dots, r\}$ and every sufficiently large $N$, consider for $h=h_j$ the set 
$\mathcal{E}_{j,N}:=\mathcal{E}_{N}$ of primitive characters described in Proposition~\ref{p:character-set}, and let
$\mathcal{E}_j^*$ denote the set of characters modulo $\widetilde{W}(N)$ induced from elements of
$\mathcal{E}_{j,N}$.
Suppose further that for each $j$ and $N$ there is a subset $\mathcal{E}_{j,N}^+ \subset \mathcal{E}_{j,N}$ 
and a function $\eps(x)$ tending to zero as $x \to \infty$ such that
\begin{align} \label{eq:tenen-assumpt-1}
S_{h^*_j\chi^*}(y)
= S_{|h_j|}(y)
\prod_{p\leq y} \left(
\frac{\sum_{p^k \leq y} h^*_j(p^k)\chi^*(p^k)p^{-k}}{\sum_{p^k \leq y} |h_j(p^k)|p^{-k}}\right)
 + O\bigg( \frac{\eps(y)}{\log y}\exp\Big(\sum_{p\leq y} \frac{|h_j(p)|}{p}\Big)\bigg)
\end{align}
for all $y\in [N^{1/2},N]$ and for every $\chi^* \in \mathcal{E}_j^*$ induced from some 
$\chi \in \mathcal{E}_{j,N}^+$, and 
\begin{align} \label{eq:tenen-assumpt-2}
S_{h^*_j\chi^*}(y)
\ll  \frac{\eps(y)}{\log y}\exp\Big(\sum_{p\leq y} \frac{|h_j(p)|}{p}\Big)
\end{align}
for all $y\in [N^{1/2},N]$ and for every $\chi^* \in \mathcal{E}_j^*$ induced from some
$\chi \in \mathcal{E}_{j,N} \setminus \mathcal{E}_{j,N}^+$. Then:

(i) As $N \to \infty$, the following asymptotic formula holds for $T \in [N(\log N)^{-B_0},N]$:
\begin{align} \label{eq:expansion-1}
\frac{1}{\vol T\fK}
\sum_{\n \in \ZZ^s \cap T\fK} 
\prod_{j=1}^r 
h_j(\vphi_j(\n)) 
=~&\beta_{\infty}(T,N)
\sum_{\substack{\chi_1, \dots, \chi_r \\ \chi_j \in \mathcal{E}_{j,N}}}
\prod_{p \leq N} \beta_{p}(\chi_1, \dots, \chi_r) 
+ o\Bigg( 
\prod_{j=1}^r E_{h_j}(N,1)
\Bigg)
~,  
\end{align}
where
$$
\beta_{\infty}(T,N)
= \frac{T^{i(t_1 + \dots + t_r)}}{\vol \fK}
\int_{\fK} \vphi_1(\x)^{it_1} \dots \vphi_r(\x)^{it_r} \d \x
~ \prod_{j=1}^r S_{|h_j|}(N),
$$
while
\begin{align*}
\beta_p(\chi_1, \dots, \chi_r) =&
\Bigg(\prod_{i=1}^r \sum_{k\geq 0} \frac{|h_i(p^k)|}{p^k}\Bigg)^{-1} \times \\
\times &\sum_{\substack{ a_1, \dots, a_r \\ \in \NN_0 }}
\left(
\prod_{i=1}^r \frac{h_i^*(p^{a_i}) \widetilde\chi_{i}(p^{a_i})}{1-p^{-1}}
\right)
\left(
\lim_{m \to \infty} 
\frac{1}{p^{ms}}
\sum_{\v \in (\ZZ/p^{m}\ZZ)^s}
\prod_{i=1}^r \overline{\widetilde\chi_{i,p}}(\vphi_i(\v))
 \1_{p^{a_i}\|\vphi_i(\v)}  \right),
\end{align*}
where $\chi_j = \prod_p \chi_{j,p}$ denotes the decomposition of $\chi_j \Mod{q_j}$ into 
characters modulo $p^{v_p(q_j)}$, and where, given any Dirichlet character $\chi$, 
we let $\widetilde \chi$ denote the completely multiplicative function defined via
$$
\widetilde \chi (p) = 
\begin{cases}
\chi (p) & :  \chi (p) \neq 0, \cr
1 & \text{otherwise.}
\end{cases}
$$

(ii) Let $B \in (0, w(N)]$ be a cut-off parameter that is sufficiently large in terms of 
$r$, $H$ and the bound $L$ on the coefficients of $\psi_1, \dots, \psi_r$, and let
$Q$ denote product of all primes $p < B$ and of the conductors of the characters in 
$\mathcal{E}_{1,N}^+, \dots, \mathcal{E}_{r,N}^+$. 
Then,
\begin{align} \label{eq:beta_p;p-nmid-Q}
\beta_p(\chi_1, \dots, \chi_r)
&= 
\prod_{j=1}^r
\left(1+\frac{|h_j(p)|}{p} \right)^{-1}
\left(1+\frac{h_j^*(p)\chi_j(p)}{p} \right)
+ O_{H,r}(p^{-2}),
\end{align}
for all $p\nmid Q$, and, writing $\beta_Q(\chi_1, \dots, \chi_r) 
= \prod_{p|Q} \beta_p (\chi_1, \dots, \chi_r)$, we have
\begin{align} \label{eq:expansion-2}
&\frac{1}{\vol T\fK}
\sum_{\n \in \ZZ^s \cap T\fK} 
\prod_{j=1}^r 
h_j(\vphi_j(\n)) \\
\nonumber
&=
(1+O_{H,r}(B^{-1}))
\beta_{\infty}(T,N)
\sum_{\substack{\chi_1, \dots, \chi_r \\ \chi_j \in \mathcal{E}_{j,N}}}
\beta_Q(\chi_1, \dots, \chi_r)
\prod_{\substack{p \leq N\\p\nmid Q}} 
\left(1+\frac{h_j^*(p)\chi_j(p) - |h_j(p)|}{p} \right) \\
\nonumber
& \qquad + o\Bigg( 
\prod_{j=1}^r E_{h_j}(N,1)
\Bigg)
~.
\end{align}
\end{theorem}

\begin{rem}
 Tenenbaum's asymptotic result \cite[Th\'eor\`eme 1.3]{tenen} and the decay estimate given in
 \cite[Corollaire 2.1]{tenen} provide tools for checking the conditions on $S_{h_j \chi^*}(x)$ for
 $\chi \in \mathcal{E}_{j,N}^+$ and $\chi \in \mathcal{E}_{j,N} \setminus \mathcal{E}_{j,N}^+$ in many explicit applications.
\end{rem}

In those cases where the local factors $\beta_p(\chi_1, \dots, \chi_r)$ are in fact independent of the characters 
$\chi_1, \dots, \chi_r$ or when, for instance, $\mathcal{E}_{j,N}$ is independent of $N$ and $\#\mathcal{E}_{j,N}^+ =1$
for all $j \in \{1, \dots, r\}$, then
the main term in the previous theorem becomes a product of local factors, allowing one 
to prove a local-to-global principle.
The latter condition holds for example when $h_j$ is $\chi_j(n)n^{it_j}$-pretentious, 
i.e.\ when $h_j$ is bounded and \eqref{eq:pret} holds for some character $\chi_j$ and some $t_j \in \RR$. 
Asymptotic results for correlations of bounded pretentious multiplicative functions were first proved 
by Frantzikinakis and Host in \cite[Theorem 1.1]{FH} and re-proved with explicit main and error terms
by Klurman and Mangerel in \cite{KM}.
As a corollary to our main result, we obtain the following version of the pretentious case of these results:

\begin{corollary} \label{c:pret}
 Suppose that $h_1, \dots, h_r : \NN \to \CC$ are bounded multiplicative functions with the property that for
 every $1 \leq j \leq r$  there exists a character $\chi_j$ and a real number $t_j$ such that 
 $\sum_p (1 - \Re (h_j(p)\chi_j(p)p^{it_j}))p^{-1} < \infty$. 
 Then, as $T \to \infty$,
 \begin{align*}
\frac{1}{\vol T\fK}
\sum_{\n \in \ZZ^s \cap T\fK} 
\prod_{j=1}^r 
h_j(\vphi_j(\n))
&=
\Big(1 + O_{H,r}(B^{-1})\Big)
~\beta_{\infty}(T)
\prod_{p' \leq B}
\beta_{p'}(\chi_1, \dots, \chi_r) 
\prod_{B < p \leq T} \beta_p \\
&+ o\Bigg(
\frac{1}{(\log T)^r} 
\prod_{j=1}^r  \prod_{p \leq T} 
\left(1 + \frac{|h_j(p)|}{p} \right)
\Bigg)
~, 
\end{align*}
where
$$
\beta_{\infty}(T)
= \frac{T^{i(t_1 + \dots + t_r)}}{\vol \fK}
\int_{\fK} \vphi_1(\x)^{it_1} \dots \vphi_r(\x)^{it_r} \d \x
~ \prod_{j=1}^r S_{|h_j|}(T),
$$
while the local factors $\beta_{p}(\chi_1, \dots, \chi_r)$ are as in Theorem \ref{t:main'} and 
the factors $\beta_p$ are given by
\begin{align*}
\beta_p
&= 
\prod_{j=1}^r
\left(1+\frac{|h_j(p)|}{p} \right)^{-1}
\left(1+\frac{h_j(p)\chi_j(p)p^{it_j}}{p} \right).
\end{align*}
\end{corollary}

While the present paper is mainly concerned with the construction of correct-order pseudo-random majorants, 
no such construction is required in the case of the above corollary, and more generally in 
the setting of Frantzikinakis and Host's work \cite{FH}.
The reason for this lies in the fact that the trivial majorant given by the all-one function $\1$ 
is a pseudo-random majorant of the correct average order for every pretentious multiplicative function.
Working with the all-one function $\1$ as a majorant leads to an error term of the form $o(1)$ instead 
of $o((\log T)^{-r}\prod_{j=1}^r\prod_{p\leq T} (1 + p^{-1}|h_j(p)|))$ in the asymptotic formula 
\eqref{eq:main}.
To be precise, the average order of the majorant appears as a factor in the error term.
This is why a majorant of correct average order is required in order to capture the 
behaviour of functions with small mean values in this asymptotic formula.

\section{A $W$-tricked version of the main theorem} \label{s:minor-arc}
In this section we prove, assuming the existence of suitable majorant functions, a special case of the main theorem. 
The main theorem itself will be deduced from this special case in Section \ref{s:W-reduction}, while most of the 
remaining sections of this paper will be concerned with the construction of majorant functions and verification of the required 
properties: in Sections \ref{s:majorants-intro}--\ref{s:majorant} we construct the majorants, in 
Section \ref{s:linearform} we recall and introduce all relevant concepts around the notion of families of 
pseudorandom majorants, and in Sections \ref{s:linearform} and \ref{s:correlation-condition}, we check that the constructed 
majorant functions give in fact rise to families of pseudorandom majorants.
For references purposes, we summarise the results of Sections \ref{s:majorants-intro}--\ref{s:correlation-condition} 
in the statement below, emphasising that all terms are properly introduced in Section~\ref{s:linearform}.

\begin{theorem}[Existence of suitable majorant functions] \label{t:majorants-s2}
  Let $B, D \geq 1$ be integers, let $N>1$ be an integer parameter, and 
 suppose $W_1, \dots, W_r \in [1, (\log N)^B]$ are integers, 
 each divisible by $W(N)$. 
 Let $W = \lcm(W_1, \dots, W_r)$ and, for each $i \in \{1, \dots, r\}$, let $A_i \in \{1, \dots, W_i\}$ 
 be coprime to $W_i$.  
 Suppose that $h_1, \dots, h_r \in \mathcal{F}^*$ and define for each $i\in \{1, \dots, r\}$
 and for each tuple $\tau = (N, W_1,\dots, W_r, A_1, \dots, A_r)$
 the $W$-tricked function $f_i^{(\tau)}: n \mapsto h_i(W_i n + A_i)$ 
 together with the weight $E_{f_i}^{(\tau)} = E_{h_i}(N;W_i)$. 
 
 Then there exists a family of $D$-pseudorandom majorants\footnote{in the sense of 
 Definition \ref{def:majorant-family}} 
 $\{\nu^{(\tau)}: \{1, \dots, \lfloor N/W \rfloor \} \to \RR_{> 0}\}_{\tau}$, 
 for $f_1^{(\tau)}, \dots, f_r^{(\tau)}$ with weights  $E_{f_1}^{(\tau)}, \dots, E_{f_r}^{(\tau)}$. 
 The majorants $\nu^{(\tau)}$ arise as 
 $$
 \nu^{(\tau)}(n) = \frac{1}{r}\sum_{j=1}^r \frac{\nu_{h_j}^{(N)}(W_jn+A_i)}{E_{h_j}(N;W_i)}, 
 \qquad (n \leq N/W),
 $$
 where the
 $\nu_{h_j}^{(N)}: \{1, \dots, N\} \to \RR_{>0}$ are certain functions that satisfy
 \begin{equation} \label{eq:correct-order-s2}
 |S_{h_j}(T;W_j,A_j)|
 \ll S_{\nu_{h_j}^{(N)}}(T;W_j,A_j)
 \ll E_{h_j}(T;W_j) \sim E_{h_j}(N;W_j)
 \end{equation}
 for all $T \in [N(\log N)^{-B},N]$ and all sufficiently large $N$.
\end{theorem}

The following special case of Theorem \ref{t:main} works in subprogressions whose common difference is an integer 
divisible by the primorial $W(N)$. 
This procedure removes potential irregularities in the behaviour of the multiplicative functions 
$h_j \in \mathcal{F}^*$ that occur when working in progressions to small moduli.
\begin{proposition}\label{p:main}
Let $B_0 \geq 0$ and $B_2, C >0$ be constants, let $r,s,L \geq 2$ be integers, let $\delta \in (0,1)$ and
let $N,T>1$ be integer parameters satisfying $N (\log N)^{-B_0} \leq T \leq N$. 
Let $h_1, \dots, h_r \in \mathcal{F}^*$ be multiplicative functions and, 
for each $j \in \{1, \dots,r\}$, define the function $h_j^*:n \mapsto h_j(n)n^{-it_j}$, where
$t_j = t_{j,N}$ denotes the real number from Definition \ref{d:F} with $x=N$ and the given value of~$C$.
Let $H>1$ and $\alpha>0$ be constants such that condition (i) and (iii) of Definition \ref{d:M} hold with 
the given value of $H$ and with $\alpha_h = \alpha$ for every $h \in \{h_1, \dots, h_r\}$.

Then there exist positive constants $c(r,\delta)$ and $B_1= O_{B_0,B_2,H,\alpha,r,\delta}(1)$,  
and, for each $N$, an integer $\widetilde{W}(N) \leq (\log N)^{B_1}$, divisible by $W(N)= \prod_{p\leq w(N)} p$, 
such that the following holds, provided $C$ was sufficiently large with respect to $H$, $\alpha$, $r$ and 
$c(r,\delta)$.

Let $\vphi_1, \dots, \vphi_r \in \ZZ[X_1, \dots X_s]$ be linear polynomials as in assumption (ii) 
of Theorem~\ref{t:main} for some fixed value of $c \in (0,1)$, or, 
if $h_1, \dots, h_r \in \mathcal{F}$, for $c=1-B_0\log \log N$. 
In particular, suppose that the coefficients of the linear forms $\psi_j = \vphi_j - a_j(N)$ are 
bounded by $L$ in absolute value. 
Let $W_1, \dots, W_r \in [1, (\log N)^{B_1 + B_2}]$ be integers, 
each divisible by $\widetilde{W}(N)$ and such that $W_j/\widetilde{W}(N) \leq (\log N)^{B_2}$,
and let $W'= \lcm(W_1, \dots, W_r)$.
For each $j \in \{1, \dots, r\}$, let $0 < A_j < W_j$ be an integer co-prime to $W_j$.
Finally, suppose that $\fK \subset [-1,1]^s$ is a convex subset with 
$\vol(\fK) > 0$ and such that
$W_j\vphi_j( \frac{T}{W'}\fK) + A_j \subset [1,T]$ for each $1\leq j \leq r$,
all sufficiently large $N$ and all $T \in [N (\log N)^{-B_0}, N]$.
Then, as $N \to \infty$, we have
\begin{align*}
\frac{1}{\vol(\fK) (T/W')^s} 
&\sum_{\substack{\n \in \ZZ^s \cap (T/W')\fK}} 
\prod_{i=1}^r 
h_i(W_i\vphi_i(\n) + A_i) \\
&
= 
\Bigg(\frac{1}{\vol(\fK)} \int_{\fK}
\prod_{\ell=1}^r \Big(\psi_{\ell} (\x)\Big)^{it_{\ell}} \d \x \Bigg)
\prod_{j = 1}^r \Big(\frac{W_j T}{W'}\Big)^{it_j}
S_{h_j^*}\left(N;\widetilde{W}(N),A_j\right) \\
&\quad + (\kappa(\delta) + o_{N\to \infty}(1))\Bigg( \prod_{j=1}^r E_{h_j}(N;\widetilde W (N)) \Bigg),
\end{align*}
uniformly for all $T \in [N (\log N)^{-B_0}, N]$ and all $\vphi_1, \dots, \vphi_r$, $W_1, \dots W_r$ and
$A_1, \dots, A_r$ as above, and where $\kappa$ is a function that satisfies $\kappa(\delta) \to 0$ as
$\delta \to 0$.
\end{proposition}

\begin{rem} \label{rem:prop-main}
 If $h_1, \dots, h_r$ are such that $\widetilde{W}(N)$ and the $t_j$ are independent of $\delta$ and $C$, then
 the term $\kappa(\delta)$ in the error term above can be omitted. 
 This is for instance the case when for every $j \in \{1, \dots, r\}$, both
 $h_j \in \mathcal{F}$ and when the set of primitive characters
 $\chi$ for which $|x^{-1}\sum_{n\leq x} h_j(n)\chi(n)|$ is maximal does not depend on $x$ as soon as 
 $x$ is sufficiently large.
\end{rem}

Observe that Proposition \ref{p:main} is a statement about a family of $W$-tricked versions of the functions 
$h_1, \dots, h_r \in \mathcal{F}^*$. 
To prove this result, we will begin by introducing a suitable choice of $\widetilde W(N)$, 
the existence of which is claimed in the statement.
Our choice of $\widetilde W(N)$ will arise from an application of \cite[Proposition 5.1]{lmm}.
Once we have set up this application and defined $\widetilde W(N)$, we will show that the the main result of \cite{lmm} 
can in fact be applied to obtain uniform bounds on the correlation with nilsequences that apply uniformly to all
$W$-tricked versions of the $h_j$ that arise in the statement of the proposition.
This in turn will allow us to apply the machinery from \cite{GT-linearprimes} and \cite{GTZ} to deduce the proposition.

\begin{proof}
Let $\widetilde W(N)$ denote the integer produced by \cite[Proposition 5.1]{lmm} 
when applied to the collection of functions $h_1, \dots, h_r \in \mathcal{F}^*$,  
with the given values of $H$ and $\alpha$ and with a fixed choice of $E > \max(2 B_2,B_0)$.
Setting $B_1 := \kappa + 2$ for the value of $\kappa = O_{E,H,r,\alpha}(1)$ that arises in the above 
application of \cite[Proposition 5.1]{lmm}, we then have $\widetilde W(N) \leq (\log N)^{B_1}$.
The proof of \cite[Proposition 5.1]{lmm} furthermore implies that $B_1 = \kappa + 2 \geq E > B_2$.
Finally, suppose that the value of $C$ determining the twists $h_1^*, \dots, h_r^*$ satisfies
$C \geq 2E+\kappa +4$ in accordance with the statement of \cite[Proposition 5.1]{lmm}.
Then \cite[Proposition 5.1]{lmm} implies that there exists a function $\eps: \NN \to \RR$, 
$\eps(N) = o_{N \to \infty}(1)$ such that, writing $W=\widetilde{W}(N)$,
\begin{align}\label{eq:major-arc}
\frac{q W}{|I|}
 \sum_{\substack{  m \in I \\ m \equiv A ~(q W)}} h^*(m)
- S_{h^*}(N;W,A)
= O_{E,H,\kappa}\left(\eps(N) E_h(N;W)
\right) 
\end{align}
holds uniformly for all intervals $I \subseteq \{1,\dots, N\}$ of length $|I|> N(\log N)^{-E}$, 
for all integers $0<q \leq (\log N)^{E}$, for all $A \in (\ZZ/q W\ZZ)^*$
and for every $h^*\in \{h_1^*, \dots, h_r^*\}$.

\cite[Theorem 6.1]{lmm}, which is the main result of that paper, takes the value of $\widetilde W(N)$ produced 
by \cite[Proposition 5.1]{lmm} for any given $h \in \mathcal{F}^*$ and 
bounds the correlation of $n \mapsto h(\widetilde W(N)n+ A)$ with nilsequences.
The information on $\widetilde W(N)$ that is used by \cite[Theorem 6.1]{lmm} is exactly the information summarised 
in the above statement about \eqref{eq:major-arc} for $W=\widetilde{W}(N)$.
Inspecting the range of $q$ in the statement on \eqref{eq:major-arc}, 
we observe that whenever we set $W=w \widetilde W(N)$ for some $w \leq (\log N)^{E/2}$, 
then \eqref{eq:major-arc} continues to hold for all $q \in \{1, (\log x)^{E/2}\}$ and 
for all intervals $I$ with $|I|>x(\log x)^{-E/2}$.
The fact that we have \emph{uniform} bounds in \eqref{eq:major-arc} as $W$ ranges over $\{w \widetilde W(N):w \leq (\log N)^{E/2}\}$,
and as $h^*$ ranges over $\{h_1^*, \dots, h_r^*\}$ implies that \cite[Theorem 6.1]{lmm} applies uniformly to $h \in \{h_1, \dots, h_r\}$ 
and with $\widetilde W(N)$ being replaced by any integer of the form $w \widetilde W(N)$ with $w \leq (\log N)^{E/2}$,
so long we replace $E$ by $E':=E/2$ everywhere in the statement except for in the definition of $C$.
To be precise, this application of \cite[Theorem 6.1]{lmm} yields the following:

Let $G/\Gamma$ be a nilmanifold together with a filtration $G_{\bullet}$ of $G$ of degree $d$ 
and let $g \in \mathrm{poly}(\ZZ,G_{\bullet})$ a polynomial sequence.
Suppose that $G/\Gamma$ has a $M_0$-rational Mal'cev basis adapted to 
$G_{\bullet}$ for some $M_0 \geq 2$ and let $G/\Gamma$ be equipped with the 
metric defined by this basis.
Let $F:G/\Gamma \to \CC$ be a $1$-bounded Lipschitz function.
Then \cite[Theorem 6.1]{lmm} implies that, provided $E \geq 1$ is sufficiently large with respect to 
$d$, the dimension of $G$, $\alpha$ and $H$, we have
\begin{align} \label{eq:lmm-main}
&\left|
\frac{W_j}{T}
\sum_{n\leq T/W_j}
\bigg(h_j(W_jn+A_j)-
(W_jn+A_j)^{it_j} S_{h_j^*}(N;W_j,A_j)\bigg)
F(g(n)\Gamma)
\right| \ll_{d,m_G,\alpha,H,E} \\
\nonumber
&\quad (1+ \|F\|_{\mathrm{Lip}})
\left\{ \eps(N) + \frac{1}{\log w(N)} 
+\frac{M_0^{O_{d,m_G}(1)} }{(\log \log N)^{1/(4^{d+1} \dim G)}}
\right\}
E_{h_j}(N;\widetilde{W}(N)),
\end{align}
uniformly for all $j\in \{1, \dots, r\}$, all integers $W_j = w\widetilde W(N)$ with $w \leq (\log N)^{E/2}$,
all integers $0 < A_j < W_j$ with $\gcd(A_j,W_j)=1$ and all $T \in [N(\log N)^{-E/2},N]$.

The bound \eqref{eq:lmm-main} applies in particular to all $W_j$ appearing in the statement of 
Proposition~\ref{p:main} since $W_j/\widetilde W(N) \leq (\log N)^{B_2} \leq (\log N)^{E/2}$
by assumption.
Thus the function
$$n \mapsto 
h_j \left(W_j n + A_j \right)
- \left( W_j n + A_j \right)^{it_j}
S_{h_j^*}(N, W_j, A_j),
$$
defined on the integers $1 \leq n \leq T/W_j$, is orthogonal to nilsequences
provided $E$, and hence $B_2$ and $C$, are sufficiently large with respect to $d$, $\dim G$, $\alpha$ and $H$.
By \eqref{eq:major-arc}, it follows moreover that
$$S_{h_j^*}(N, W_j, A_j)
= S_{h_j^*}(N, \widetilde{W}, A_j) +
o\Big( E_{h_j}(N;\widetilde W)\Big)
$$
where $\widetilde{W}=\widetilde{W}(N)$ and where $E_{h_j}(N;\widetilde{W})$ is as in \eqref{eq:def-E}.
Thus, the function 
$$
\bar h_j: 
n \mapsto 
h_j \left(W_j n + A_j \right)
- \left( W_j n + A_j \right)^{it_j}
S_{h_j^*}(N, \widetilde{W}(N), A_j),
$$
defined on $1 \leq n \leq T/W_j$, is also orthogonal to nilsequences
provided $E$, and hence $B_2$ and $C$, are sufficiently large with respect to $d$, $\dim G$, $\alpha$ and $H$.
Consider the normalisation of $\bar h_j$ given by
$$\tilde h_j : n \mapsto 
\frac{ \bar h_j(n)}{E_{h_j}(N; \widetilde W(N))}.$$
Mertens' estimate and property (1) of Definition \ref{d:M} imply that 
$$E_{h_j}(T; \widetilde W(N)) = (1 + o_{H,E}(1))E_{h_j}(N; \widetilde W(N))$$ for all $T$, $N$ as above.
Since $S_{\nu_{h_j}^{(N)}}(N;\widetilde{W}(N),A) \ll E_{h_j}(N;\widetilde W(N))$ 
by \eqref{eq:correct-order-s2} (cf.\ Proposition \ref{p:average-order}), it follows that the function
$$\tilde \nu_j: n \mapsto \frac{1}{2}\left(
\frac{\nu_{h_j}^{(N)}(W_jn + A_j)} {E_{h_j}(N; \widetilde W(N))} + \1 \right)
$$
is a correct-order majorant for $\tilde h_j$ (in the sense of Proposition \ref{p:average-order}) when restricted to 
$1 \leq n \leq T/W_j$.
Furthermore, Theorem \ref{t:majorants-s2} 
(cf.\ Theorem \ref{t:pseudorandom}) implies that $\tilde \nu^{(\tau)} = \frac{1}{r} \sum_{j=1}^r \tilde \nu_j$ is pseudo-random.

We now seek to apply the (transferred) inverse theorem for uniformity norms from \cite{GTZ}, 
see \cite[Conjecture 1.2 and Theorem 1.3]{GTZ} and \cite[Proposition 10.1]{GT-linearprimes}.
Note that this statement only involves linear sequences, i.e. their degree is equal to the step
of the nilmanifold involved.
If $\mathcal{M}_{r-2,\delta}$ denotes the finite set of $r-2$-step nilmanifolds from the statement of the
the inverse theorem, let $c(r, \delta)$ be the maximum dimension of the elements of this set.
Then \eqref{eq:lmm-main} applies to all linear sequences associated to manifolds in $\mathcal{M}_{r-2,\delta}$
provided $E$ is sufficiently large with respect to the step $r-2$ and $c(r, \delta)$.
Thus, as soon as $N$ is sufficiently large, it follows from the bound \eqref{eq:lmm-main}, 
from Theorem \ref{t:majorants-s2} (cf.\ Theorem \ref{t:pseudorandom}) and
from the inverse theorem for uniformity norms from \cite{GTZ} 
that the function $\tilde h_j$ satisfies
$$
\|\tilde h_j\|_{U^{r-1}[T/W_j]} \leq \delta,
$$
provided $E$, and hence $B_1$ and $C$, are sufficiently large with respect to $r$, $\alpha$ and $H$.
\footnote{If $h_1, \dots, h_r$ are such that $\widetilde{W}(N)$ and the $t_j$ are independent of $E$, then
the $\tilde h_j = \tilde h_{j,N}$ enjoy the property
$\|\tilde h_{j,N}\|_{U^{r-1}[T/W_j]} = o_{N \to \infty}(1)$.
This justifies Remark \ref{rem:prop-main}.
}

This allows us to apply the generalised von Neumann theorem \cite[Proposition 7.1]{GT-linearprimes}
(cf.\ \cite[Lemma 4.4]{GT-linearprimes}) to deduce that
\begin{align} \label{eq:vN}
& 
\EE_{\substack{\n \in \ZZ^s \cap (T/W')\fK}} 
\prod_{i=1}^r h_i(W_i\vphi_i(\n) + A_i) \\
\nonumber
&= 
\EE_{\substack{\n \in \ZZ^s \cap (T/W')\fK}} 
\prod_{j=1}^r 
\Big( E_{h_j}(N; \widetilde W) \tilde h_j(\vphi_j(\n)) 
+ \left( W_j \vphi_j(\n) + A_j \right)^{it_j} S_{h_j^*}(N, \widetilde W, A_j) \Big) \\
\nonumber
&=
\bigg(\prod_{i=1}^r 
E_{h_i}(N; \widetilde W) \bigg)
\EE_{\substack{\n \in \ZZ^s \cap (T/W')\fK}} 
\bigg(
\tilde h_j(\vphi_j(\n)) 
+ \left( W_j \vphi_j(\n) + A_j \right)^{it_j} \frac{S_{h_j^*}(N, \widetilde W, A_j)}{E_{h_j}(N; \widetilde W)}
\bigg)\\
\nonumber
&=
\EE_{\substack{\n \in \ZZ^s \cap (T/W')\fK}} 
\prod_{j=1}^r \left( W_j \vphi_j(\n) + A_j \right)^{it_j} S_{h_j^*}(N, \widetilde W, A_j)
+(\kappa(\delta)+o(1))\bigg( \prod_{j=1}^r  E_{h_j}(N; \widetilde W)\bigg),
\end{align}
where, following \cite{GT-linearprimes}, $\kappa(\delta) \to 0$ as $\delta \to 0$, and
where we used the notation $\EE_{s \in S} = \frac{1}{\# S} \sum_{s \in S}$ for finite sets $S$.

In the case where $h_1, \dots, h_r \in \mathcal{F}$, i.e.\ where $t_1 = \dots = t_r = 0$, 
the above estimate concludes the proof of the proposition, which could have been simplified in many places.

To further simplify the main term of \eqref{eq:vN} in the remaining case where not all $t_j$ vanish, 
our next aim is to show that the summation argument $( W_j \vphi_j(\n) + A_j )^{it_j}$ can be replaced 
by $( W_j \psi_j(\n))^{it_j}$.
More precisely, using the fact that $n^{it_j}$ varies slowly, we will show that, outside an exceptional
set, $( W_j \vphi_j(\n) + A_j )^{it_j}$ can approximated by $( W_j \psi_j(\n))^{it_j}$.
We begin with the exceptional set.
Recall that the linear form $\psi_j = \vphi_j - a_j$ has bounded coefficients
and note that for all $0 < T' < T$ we have:
$$
\sum_{j=1}^r \sum_{\substack{\n \in \ZZ^s \cap T \fK: \\ |\psi_j(\n)| \leq T'}} 1
\ll_r T^{s-1}T'.
$$
Thus, if $\fK^*:=\{ \x \in \fK:  |\psi_i(\x)| > (T/W')^{c^*} \text{ for all } 1\leq i \leq r\}$
for some $c^* \in (0,1)$, then
$$ 
\sum_{\substack{\n \in \\ \ZZ^s \cap (T/W')(\fK \setminus \fK^*)}} 
\bigg|
\prod_{j=1}^r \left( W_j \vphi_j(\n) + A_j \right)^{it_j} 
- \prod_{j=1}^r \left( W_j \psi_j(\n) \right)^{it_j} \bigg|
\ll (T/W')^{s - 1 + c^*}.
$$
We will make use of the fact that $|\psi_j(\n)| > (T/W')^{c^*}$ is large for all 
$1 \leq j \leq r$ and $\n \in (T/W')\fK^*$ to 
show that we may approximate $( W_j \vphi_j(\n) + A_j )^{it_j}$ by $( W_j \psi_j(\n))^{it_j}$.
To do so, recall that $|a_j(N)| \leq N^{c} = (T/W')^{c+o(1)}$ for some $c \in (0,1)$ and all $1 \leq j \leq r$ and $N \geq 1$,
and set $c^* = (1-c)/2$.
Then,
\begin{align*}
& \Big|\log (W_j \vphi_j(\n) + A_j) - \log (W_j \psi_j(\n)) \Big|
= \bigg| \log \bigg( \frac{W_j \vphi_j(\n) + A_j}{W_j \psi_j(\n)} \bigg) \bigg| \\
& = \bigg| \log \bigg( 1 + \frac{W_j a_j + A_j}{W_j \psi_j(\n)} \bigg) \bigg| 
\ll  \frac{|a_j| + 1}{(T/W')^{1-c^*}} 
\ll  (T/W')^{-1+c^* + c} 
 = (T/W')^{-c^*}.
\end{align*}
Recall further that $|t_{j,N}| \leq 2 \log N \ll \log (T/W') \ll_{\eps} (T/W')^{\eps}$.
Setting $\eps = c^*/2$, we have
$$
( W_j \vphi_j(\n) + A_j )^{it_j} 
= ( W_j \psi_j(\n))^{it_j} (1 + O((T/W')^{-c^*/2}))
= ( W_j \psi_j(\n))^{it_j}  + O((T/W')^{-c^*/2})
$$
Putting everything together and recalling that the $\psi_j$ are linear forms, we obtain:
\begin{align} \label{eq:vN-MTsum}
\nonumber
&\EE_{\substack{\n \in \ZZ^s \cap (T/W')\fK}} 
\prod_{j=1}^r \left( W_j \vphi_j(\n) + A_j \right)^{it_j}\\
\nonumber
&= \frac{1 + o(1)}{(T/W')^s \vol \fK}
\Bigg(
\sum_{\substack{\n \in \\ \ZZ^s \cap (T/W')(\fK\setminus\fK^*)}} 
\prod_{j=1}^r \left( W_j \psi_j(\n)\right)^{it_j} 
+ \sum_{\substack{\n \in \\ \ZZ^s \cap (T/W')\fK^*}} 
\prod_{j=1}^r \left( W_j \vphi_j(\n) + A_j \right)^{it_j} \Bigg) \\
\nonumber
& \quad +O((T/W')^{ -c^*}) \\
\nonumber
&= \frac{1 + o(1)}{(T/W')^s \vol \fK}
\sum_{\substack{\n \in \\ \ZZ^s \cap (T/W')\fK^*}} 
\prod_{j=1}^r \left( W_j \psi_j(\n)\right)^{it_j} 
+O((T/W')^{-1+c^*}) + O((T/W')^{ -c^*/2})\\
\nonumber
&= \Bigg(\prod_{\ell = 1}^r\Big(\frac{W_{\ell}T}{W'}\Big)^{it_{\ell}} \Bigg)
\frac{1 + o(1)}{(T/W')^s \vol \fK}
\sum_{\substack{\n \in \\ \ZZ^s \cap (T/W')\fK}} 
\prod_{j=1}^r \Big(\psi_j ((W'/T)\n)\Big)^{it_j} + O((T/W')^{-\min(1-c^*,c^*/2)})
\\
&= \Bigg(\prod_{\ell = 1}^r\Big(\frac{W_{\ell}T}{W'}\Big)^{it_{\ell}} \Bigg)
\frac{1}{\vol \fK}
\int_{\fK}
\prod_{j=1}^r \Big(\psi_j (\x)\Big)^{it_j} \d \x + o_{N \to \infty}(1).
\end{align}
Multiplying through by $\prod_{1 \leq j \leq r} S_{h_j^*}(N, \widetilde W, A_j)$, which is 
bounded by $\prod_{1 \leq j \leq r} E_{h_j}(N; \widetilde W)$, 
Proposition \ref{p:main} now follows from the combination of \eqref{eq:vN} and \eqref{eq:vN-MTsum}.
\end{proof}

\section{Proof of the main result from its special case}
\label{s:W-reduction}

In this section we deduce Theorem \ref{t:main} from the $W$-tricked version presented in Proposition \ref{p:main},
that is to say, assuming the results about pseudorandom majorants we summarised in Theorem \ref{t:majorants-s2}.
Since Proposition \ref{p:main} assumes that the quantities $W_i$, $\widetilde{W}(N)$ as well as $N/T$ are all
bounded above by a fixed power of $\log N$ (an assumption that is essential for the application of the results 
from \cite{lmm}), the proof of Theorem \ref{t:main} will require us to truncate certain summations.
For this purpose we introduce the following exceptional set.
\begin{definition}[Exceptional set]
Let $D > 0$, $N>1$ and let $\mathcal{S}_D(N)$ denote the set all positive 
integers less than $N$ that are divisible by the square of an integer 
$d > (\log N)^D$.
\end{definition}

\begin{lemma}[Exceptional set I] \label{l:exceptional-set}
Let $B_0, D>0$ and let $N, T>1$ be integer parameters with $T \leq N$.
Let $h_1, \dots, h_r \in \mathcal{F}^*$ and let $\alpha$, $H$, $L$ and $\vphi_1, \dots, \vphi_r \in \ZZ[X_1, \dots, X_s]$ be as in 
Theorem~\ref{t:main}. Then
\begin{align}\label{eq:exceptional-bound}
\sum_{\n \in \ZZ^s \cap T\fK} \prod_{i=1}^r &h_i(\vphi_i(\n)) 
\1_{\vphi_i(\n) \not\in \mathcal{S}_D(N)} \\
\nonumber
&= \sum_{\n \in \ZZ^s \cap T\fK} \prod_{i=1}^r h_i(\vphi_i(\n))
+o_{N,T \to \infty} \bigg( \vol(T\fK) \prod_{i=1}^r S_{|h_i|}(N) \bigg), 
\end{align}
provided $D$ is sufficiently large with respect to $\alpha$, $r$, $s$, $L$ and $H$. 
\end{lemma}
\begin{proof}
To start with, we have
\begin{align} \label{eq:exceptional'}
\sum_{\n \in \ZZ^s \cap T\fK} 
\sum_{i=1}^r
\1_{\vphi_i(\n) \in \mathcal{S}_D(N)}
\ll \sum_{d > (\log N)^{D}}
\frac{|\ZZ^s \cap T\fK|}{d^2} 
\ll |\ZZ^s \cap T\fK| (\log N)^{-D}.
\end{align}
Let $h: \NN \to \RR$ be the multiplicative function defined by 
$h(p^j) = \max(|h_1(p^j)|, \dots, |h_r(p^j)|)$. Then $h$ satisfies conditions (i), (ii) and (iii) of 
Definition \ref{d:M} with the given values of $H$ and $\alpha$.
By \cite[Lemma 7.9]{bm}, we have
\begin{align} \label{eq:k-th-moment}
\sum_{\n \in \ZZ^s \cap T\fK} \prod_{i=1}^r h^k(\vphi_i(\n)) 
\ll_{L,r,k}  |\ZZ^s \cap T\fK| (\log T)^{O_{r,k,H}(1)}. 
\end{align}
Since each $h_i$ satisfies condition (iii) of Definition \ref{d:M} with $\alpha_{h_i}=\alpha$, we further have
\begin{equation} \label{eq:hj-mean}
\sum_{n \leq N} |h_i(n)| \gg N(\log N)^{-1 + \alpha} 
\end{equation}
for all $1 \leq i \leq r$. 
Thus \eqref{eq:exceptional-bound} follows from \eqref{eq:exceptional'} and \eqref{eq:k-th-moment}, combined 
with an application of the Cauchy--Schwarz inequality provided $D$ is sufficiently large.
\end{proof}

The following simple observation will allow us to discard most cases in which, 
in the proof of Theorem \ref{t:main}, Proposition \ref{p:main} 
would need to be applied with $W_j> (\log T)^D$ for some $j \in \{1, \dots, r\}$.
\begin{lemma}\label{l:truncation}
 Let $D>1$ and suppose that $M \leq (\log N)^D$ is an integer.
 Then every integer $w \in ((\log N)^{3D} , N]$ composed entirely out of primes dividing $M$ 
 has a square divisor of size at least $(\log N)^{2D}$ and, hence, 
 belongs to $\mathcal{S}_{D}(N)$.
\end{lemma}
\begin{proof}
Factorising $w$ as $w_1w_2^2$ for a square-free integer $w_1$, we have
$w_1 \leq M \leq (\log N)^D$. Hence, $w_2^2>(\log N)^{2D}$. 
\end{proof}
Following these preparations we are now ready for the deduction of Theorem \ref{t:main}.
\subsection{Proof of Theorem \ref{t:main} (assuming Theorem \ref{t:majorants-s2})}
Let $D>1$ be sufficiently large for \eqref{eq:exceptional-bound} to hold, and let 
$B_2>6D$ be a fixed real number.
Thus, $B_2$ depends on $r$, $s$, $L$, $H$ and $\alpha$.
Let $B_1$ and $\widetilde W(N) \leq (\log N)^{B_1}$ be the quantities given by Proposition \ref{p:main}.
Both of these quantities depend in particular on $B_2$ and we have $B_1 \geq B_2$.
Finally, define the sets
$$\mathcal{W}(N)= \{w \in [1,(\log N)^{B_2}]: p|w \Rightarrow p| \widetilde W(N)\}$$
and
$$\mathcal{W}^c(N)= \{w \in ((\log N)^{B_2},N]: p|w \Rightarrow p| \widetilde W(N)\},$$
and let $\mathcal{S}(N)$ denote the set of integers up to $N$ that are divisible by an element from $\mathcal{W}^c(N)$.
The first step in the deduction of Theorem \ref{t:main} is to show that $\mathcal{S}(N)$ is an exceptional set:
\begin{lemma}[Exceptional set II] \label{l:exceptional-set-2}
Under the assumptions of Theorem \ref{t:main}, we have
\begin{align}\label{eq:exceptional-bound-2}
\sum_{\n \in \ZZ^s \cap T\fK} \prod_{i=1}^r 
&h_i(\vphi_i(\n)) 
\1_{\vphi_i(\n) \not\in \mathcal{S}(N)}\\
\nonumber
&= \sum_{\n \in \ZZ^s \cap T\fK} \prod_{i=1}^r h_i(\vphi_i(\n))
+o_{N,T \to \infty} \bigg( \vol(T\fK) \prod_{i=1}^r S_{|h_i|}(N) \bigg), 
\end{align}
for all $T \in [N(\log N)^{-B_0}, N]$.
\end{lemma}
\begin{proof}
It follows from Lemma \ref{l:exceptional-set} and from our assumptions on $D$ that any subset of 
$\mathcal{S}(N) \cap \mathcal{S}_D(N)$ is exceptional. More precisely, we have: 
\begin{align} \label{eq:exceptional-bd-3}
\sum_{\n \in \ZZ^s \cap T\fK} \prod_{i=1}^r |h_i(\vphi_i(\n))| 
\1_{\vphi_i(\n) \in (\mathcal{S}(N) \cap \mathcal{S}_D(N))}
=o_{N,T \to \infty} \bigg( \vol(T\fK) \prod_{i=1}^r S_{|h_i|}(N) \bigg).
\end{align}
This shows in particular that the set of multiples of integers from $\mathcal{W}^c(N) \cap \mathcal{S}_D(N)$ is exceptional.
With this in mind, we seek to study the set 
$$\mathcal{S}'(N):=\mathcal{S}(N) \setminus (\mathcal{S}(N) \cap \mathcal{S}_D(N))$$
by analysing the set $\mathcal{W}^c(N)$.
Lemma \ref{l:truncation} applied with $M = \widetilde W(N)$ shows that 
$$\mathcal{W}^c(N) \cap ((\log N)^{3B_1}, N ] \subset \mathcal{S}_{B_1}(N)\subset \mathcal{S}_{D}(N).$$
Thus, let $\mathcal{W}'(N) := \mathcal{W}^c(N) \cap [(\log N)^{B_2},(\log N)^{3B_1} ]$ and decompose any 
$m \in \mathcal{W}'(N)$ as $m = m_1 m_2$ where $m_1 = \prod_{p < w(N)} p^{v_p(m)}$. 
Then $m_i > (\log N)^{3D}$ for $i=1$ or $i=2$.
If $N$ is sufficiently large and if $m_1 > (\log N)^{3D}$, then Lemma \ref{l:truncation} applied with $M = W(N) \ll (\log N)^{1+o(1)}$ 
shows that $m \in \mathcal{S}_{D}(N)$.

Thus, if $N$ is sufficiently large, then $\mathcal{S}'(N)$ is contained in the set of multiples of 
$m \in \mathcal{W}'(N)$ with $m_2 > (\log N)^{3D}$ and 
it suffices to show that \eqref{eq:exceptional-bd-3} holds with $(\mathcal{S}(N) \cap \mathcal{S}_D(N))$ 
replaced by $\mathcal{S}'(N)$.
We aim to prove this using the strategy from the proof of Lemma \ref{l:exceptional-set} and therefore require a bound on
$$\sum_{\substack{m = m_1 m_2 \in \mathcal{W}'(N) \\ m_2 \geq (\log N)^{3D} }} \frac{1}{m}.$$
Let $M_2$ denote the number of possible values that $m_2$ can take and let us begin by bounding this quantity. 
Since $\widetilde W(N) \leq (\log N)^{B_1}$, the total number of primes $p>w(N)$ dividing $\widetilde W(N)$ is at 
most $\frac{B_1 \log \log N}{\log w(N)}$.
Since $m_2 \leq m \leq (\log N)^{3B_1}$, this implies $\Omega(m_2) \leq \frac{3B_1 \log \log N}{\log w(N)}$.
Taking further into account that $w(N) \geq (\log \log N)/ \log \log \log N$, we deduce that
\begin{align*}
 M_2 &\leq
 \Big(\frac{B_1 \log \log N}{\log w(N)}\Big)^{\frac{3B_1 \log \log N}{\log w(N)}}\\
 &\leq \exp\left(\frac{\left(3 B_1 \log B_1 + \log_{(3)}N - \log_{(3)}N + \log_{(4)}N\right) \log_{(2)}N}{\log_{(3)}N - \log_{(4)}N} \right)\\
 &= \exp\left(\frac{\left(3 B_1 \log B_1 + \log_{(4)}N\right) \log_{(2)}N}{\log_{(3)}N - \log_{(4)}N} \right),
\end{align*}
where $\log_{(k)}N = \log \log \dots \log N$ denotes the $k$-fold logarithm of $N$.
For all sufficiently large $N$ we therefore obtain:
\begin{align*}
 \sum_{\substack{m = m_1 m_2 \in \mathcal{W}'(N) \\ m_2 \geq (\log N)^{3D} }} \frac{1}{m}
 &\leq \frac{M_2}{(\log N)^{3D}} \prod_{p \leq w(N)} \Big(1 + \frac{1}{p} \Big) 
 \ll \frac{M_2 \log w(N)}{(\log N)^{3D}}  \\
 &\ll (\log N)^{\frac{3 B_1 \log B_1 + \log_{(4)}N}{\log_{(3)}N - \log_{(4)}N} + \frac{\log_{(3)}(N)}{\log_{(2)}N}- 3D}
 < (\log N)^{-D}.
\end{align*}
By combining this bound with \cite[Lemma 7.9]{bm} or \eqref{eq:k-th-moment}, it follows by Cauchy--Schwarz and the assumptions on
$D$ that
\begin{align*}
 \sum_{\n \in \ZZ^s \cap T\fK} \prod_{i=1}^r |h_i(\vphi_i(\n))| 
\1_{\vphi_i(\n) \in \mathcal{S}'(N)}
=o_{N,T \to \infty} \bigg( \vol(T\fK) \prod_{i=1}^r S_{|h_i|}(N) \bigg).
\end{align*}
The lemma follows from the above and \eqref{eq:exceptional-bd-3}.
\end{proof}

Returning to the deduction of Theorem \ref{t:main}, 
observe that Lemma \ref{l:exceptional-set-2} implies 
that the expression from Theorem \ref{t:main} satisfies:
\begin{align} \label{eq:deduction-1}
\sum_{\n \in \ZZ^s \cap T\fK} \prod_{i=1}^r h_i(\varphi_i(\n)) 
= 
\sum_{ \substack{ w_1, \dots, w_r  \in \mathcal{W}(N)}}
&\sum_{\n \in \ZZ^s \cap T\fK}
\prod_{i=1}^r 
h_i(\varphi_i(\n)) 
\1_{\gcd(\varphi_i(\n), \widetilde W(N)^{\infty}) = w_i} \\
\nonumber
&+ o_{N,T \to \infty} \bigg( \vol(T\fK) \prod_{i=1}^r S_{|h_i|}(N) \bigg),
\end{align}
where $\gcd(n,m^{\infty}) := \prod_{p|m}p^{v_p(n)}$.
The following definition will simplify the notation as well as making changes of variables.
\begin{definition} \label{def:U}
Given any collection $w_1, \dots, w_r$ of positive integers entirely composed of primes dividing 
$\widetilde W(N)$, let $w=\lcm(w_1, \dots, w_r)$ and
define the \emph{multi-set} of $r$-tuples
$$
\mathcal{U}(w_1,\dots,w_r)
=
\left\{ \left(\varphi_1(\v), \dots, \varphi_r(\v) \right):
\begin{array}{c}
\v \in \{0, \dots w \widetilde W(N) - 1\}^s \cr 
w_i = \gcd( \varphi_i(\v), \widetilde W(N)^{\infty} ), 1\leq i \leq r  
\end{array}
\right\}.
$$
\end{definition}

Taking into account Definition \ref{def:U} as well as the decomposition 
$$\varphi_i(M \m + \v) = M \psi_i(\m) + \varphi_i(\v), \qquad (M \in \ZZ, \v \in \ZZ^r),$$
we claim that:
\begin{lemma} \label{l:deduction}
The main term from \eqref{eq:deduction-1} equals:
\begin{align*}
\sum_{ \substack{ w_1, \dots, w_r \\ \in \mathcal{W}(T)}}
\sum_{ \substack{ (U_1, \dots, U_r) \\ \in 
 \mathcal{U}(w_1, \dots, w_r)}}
\sum_{\substack{\m \in \ZZ^s \cap \fK^*}}
\prod_{i=1}^r 
h_i\left(w \widetilde W \psi_i(\m) + U_i \right)
+o_{N,T \to \infty} \bigg( \vol(T\fK) \prod_i S_{|h_i|}(N) \bigg), 
\end{align*}
where $w = \lcm(w_1, \dots, w_r)$, $\widetilde W =  \widetilde W(N)$ and where
$\fK^* = (T/(w \widetilde W)) \fK $. 
\end{lemma}
\begin{rem*}
 By Shiu's bound \cite[Theorem 1]{shiu}, the error terms in \eqref{eq:deduction-1} and
the above lemma are acceptable when compared to \eqref{eq:main}.
\end{rem*}

\begin{proof}
Recall that the set of $r$-tuples $(U_1, \dots, U_r) \in \mathcal{U}(w_1, \dots, w_r)$ is parametrised by the set of
vectors $\v \in \{0, \dots w \widetilde W(N) - 1\}^r$.
Given $w_1, \dots, w_r$ and $(U_1, \dots, U_r) \in \mathcal{U}(w_1, \dots, w_r)$, let $\v$ be the corresponding vector and 
consider the change of variables $\n = w \widetilde W(N)\m +\v$.
This change of variables replaces $\varphi_i(\n)$ in \eqref{eq:deduction-1} by
$$\varphi_i(w \widetilde W(N)\m +\v) = w \widetilde W(N) \psi_i(\m) + \varphi_i(\v)
= w \widetilde W(N) \psi_i(\m) + U_i,$$
and it replaces the summation range $T \fK$ by
$$\{ \x \in \RR^s: w \widetilde W(N)\x +\v \in  T\fK \} = (1/(w \widetilde W)) (T\fK - \v) = \fK^* - \v /(w \widetilde W),$$
which is a translate of the region $\fK^*$ by a vector with coordinates all bounded by $1$.

Let $\fK^{*\Delta} = \fK^* \Delta (\fK^* - \v /(w \widetilde W))$, where, as usual, $A\Delta B = (A\cup B) \setminus (A\cap B)$.
Since $\v /(w \widetilde W)$ has bounded length, $\fK^{*\Delta}$ is contained in a neighbourhood of bounded diameter of the boundary of $\fK^*$.
Since $\fK^*$ is convex, this implies that $\vol (\fK^{*\Delta}) = O((T/(w \widetilde W))^{s-1})$.
Hence, the error term incurred by replacing $\fK^* - \v /(w \widetilde W)$ by $\fK^*$ is bounded by
\begin{align*}
 \sum_{\m \in \ZZ^s \cap((\fK^* - \v /(w \widetilde W)) \cup \fK^*) } 
 \prod_{i=1}^r \left|h_i\left(w \widetilde W \psi_i(\m) + U_i \right)\right| \1_{\m \in \fK^{*\Delta}},
\end{align*}
which by Cauchy--Schwarz and \eqref{eq:k-th-moment} is bounded by
\begin{align*}
&(\log N)^{O_{r,H}(1)} |\ZZ^s \cap((\fK^* - \v /(w \widetilde W)) \cup \fK^*) |^{1/2} |\ZZ^s \cap \fK^{*\Delta}|^{1/2}\\
&\ll (\log N)^{O_{r,H}(1)} \vol(\fK^*) \frac{\vol(\fK^{*\Delta })^{1/2}}{\vol(\fK^*)^{1/2}} \\
&\ll (\log N)^{O_{r,H}(1)} \vol(\fK^*) (T/(w \widetilde W))^{-1/2}\\
&\ll (\log N)^{O_{r,H}(1)} \vol(T\fK) T^{-1/2} (w \widetilde W))^{-s-1/2}\\
&\ll (\log N)^{O_{r,H}(1)} \vol(T\fK) T^{-1/2} .
\end{align*}

To sum this bound over all $(U_1, \dots, U_r) \in \mathcal{U}(w_1, \dots, w_r)$ and all 
$w_1, \dots, w_r \in \mathcal{W}(N)$, note that the outer summation over 
$w_1, \dots, w_r \in \mathcal{W}(N)$ has at most $(\log N)^{B_2r}$ terms since
$w_i \leq (\log N)^{B_2}$ for every $w_i \in \mathcal{W}(N)$. 
In particular, $w = \lcm(w_1, \dots, w_r) \leq (\log N)^{B_2r}$, which implies that the summation over 
$\mathcal{U}(w_1, \dots, w_r)$ has $(w \widetilde W)^s \leq (\log N)^{(B_1+B_2r)s}$ terms.
Combining these bounds with \eqref{eq:hj-mean} we deduce that the total error term incurred by replacing 
$\fK^* - \v /(w \widetilde W)$ by $\fK^*$ is bounded by
\begin{align*}
&\ll (\log N)^{O_{r,s,H,B_0,B_1,B_2}(1)} N^{-1/2} \vol(T\fK)\\
&\ll (\log N)^{O_{r,s,H,B_0,B_1,B_2}(1)} N^{-1/2} (\log N)^{r(1-\alpha)} \vol(T\fK) \prod_i S_{|h_i|}(N)),
\end{align*}
which is $o_{N,T \to \infty}(\vol(T\fK) \prod_i S_{|h_i|}(N)))$. 
This completes the proof of the lemma.
\end{proof}

We proceed by analysing the main term appearing in Lemma \ref{l:deduction}.
Invoking the multiplicativity of the $h_i$, this expression equals
\begin{align} \label{eq:deduction-ain-term}
\sum_{ \substack{ w_1, \dots, w_r \\ \in \mathcal{W}(N)}}
\left(\prod_{j=1}^r h_j(w_j)\right)
\sum_{ \substack{ (U_1, \dots, U_r) \\ 
\in \mathcal{U}(w_1, \dots, w_r)}}
\sum_{\substack{\m \in \\ \ZZ^s \cap (T/w \widetilde W(N))\fK}}
\prod_{i=1}^r
h_i\left(\frac{w \widetilde W(N)}{w_i} \psi_i(\m) + \frac{U_i}{w_i} \right).
\end{align}
Since $\gcd(W(N), U_i/w_i)=1$ for all $i$, Proposition \ref{p:main} applies to 
the inner summation over $\m$ when setting $W_i = w \widetilde W(N) / w_i$. 
Since $\lcm(w/w_1, \dots, w/w_r) = w$, we have
$W' := \lcm(W_1, \dots, W_r) = w \widetilde W(N)$ and $W_i T/W' = T/w_i$. 
Thus, by Proposition \ref{p:main} the inner sum satisfies
\begin{align*}
\Bigg( \frac{ \vol(\fK)T^s}{(w \widetilde W)^s} \Bigg)^{-1}
&\sum_{\substack{\m \in \\ \ZZ^s \cap (T/w \widetilde W(N))\fK}}
\prod_{i=1}^r
h_i\left(\frac{w \widetilde W(N)}{w_i} \psi_i(\m) + \frac{U_i}{w_i} \right) \\
& = 
\Bigg(\int_{\fK} \prod_{\ell = 1}^r (\psi_{\ell}(\x))^{it_{\ell}} \d \x \Bigg)
\prod_{j=1}^r
\Big(\frac{T}{w_j} \Big)^{it_j}
S_{h_j^*}\left(N;\widetilde{W}(N),\frac{U_j}{w_j}\right) \\
& \qquad + 
(\kappa(\delta) + o(1))\Bigg(\prod_{j = 1}^r E_{h_j}(N,\widetilde{W}(N)) \Bigg) ,
\end{align*}
where $h_j^*(m) = h_j(m) m^{-it_j}$.
Recalling the definition of $\mathcal{U}(w_1, \dots, w_r)$,
making the change of variables $A_j = U_j/w_j$ and noting that
$h_j(w_j)/w_j^{it_j}=h_j^*(w_j)$, we deduce that 
\eqref{eq:deduction-ain-term} equals
\begin{align*}
&\beta_0 \sum_{\substack{w_1, \dots, w_r \\ \in \mathcal{W}(N) }}
\sum_{\substack{A_1,\dots,A_r \\ \in (\ZZ/\widetilde W \ZZ)^*}}
\Bigg(\prod_{j=1}^r 
h_j^*(w_j)S_{h_j^*}\Big(N;\widetilde W,A_j\Big)\Bigg)
\frac{\vol(\fK) T^s}{(w\widetilde W)^s}
\sum_{\substack{\v \in \\ (\ZZ/w \widetilde W \ZZ)^s}}
\prod_{i=1}^r 
\1_{\vphi_i(\v) \equiv w_i A_i \Mod{w_j \widetilde W}} \\
&+
(\kappa(\delta) + o(1))
T^s
\Bigg(\prod_{j = 1}^r E_{h_j}(N,\widetilde{W}(N)) \Bigg)
\sum_{\substack{A_1,\dots,A_r \\ \in (\ZZ/\widetilde W \ZZ)^*}}
\sum_{\substack{w_1, \dots, w_r \\ \in \mathcal{W}(N) }}
\Bigg(\prod_{i=1}^r |h_i(w_i)|\Bigg)
\beta_{\boldsymbol \vphi}(w_1A_1, \dots, w_rA_r),
\end{align*}
where
$$\beta_0 = T^{i(t_1 + \dots + t_r)}\int_{\fK} \prod_{\ell = 1}^r (\psi_{\ell}(\x))^{it_{\ell}} \d \x $$
and where $\beta_{\boldsymbol \vphi}(w_1A_1, \dots, w_rA_r)$ is defined in \eqref{eq:beta_phi-0}.

Since the main term in the expression above agrees with the one in \eqref{eq:main}, it now
remains to show that the error term in that expression is of the shape
$$
(\kappa(\delta) + o(1))
\left( \frac{T^{s}}{(\log N)^r} 
\prod_{i=1}^r  \prod_{p \leq N} 
\left(1 + \frac{|h_i(p)|}{p} \right)
\right).
$$
Recalling the definition of $E_{h_j}(N,\widetilde{W}(N))$ from \eqref{eq:def-E}, this will follow, 
provided we can show that:
$$
\sum_{\substack{A_1,\dots,A_r \\ \in (\ZZ/\widetilde W \ZZ)^*}}
\sum_{\substack{w_1, \dots, w_r \\ \in \mathcal{W}(N) }}
\Bigg(\prod_{i=1}^r |h_i(w_i)|\Bigg)
\beta_{\boldsymbol \vphi}(w_1A_1, \dots, w_rA_r)
\ll \prod_{j=1}^r  \prod_{p|\widetilde{W}(N)}
\left(1 + \frac{|h_i(p)|}{p} \right).
$$
To this end, note that
\begin{align} \label{eq:bounding_main_error}
 \nonumber
&\sum_{\substack{A_1,\dots,A_r \\ \in (\ZZ/\widetilde W \ZZ)^*}}
\sum_{\substack{w_1, \dots, w_r \\ \in \mathcal{W}(N) }}
\Bigg(\prod_{i=1}^r |h_i(w_i)|\Bigg)
\beta_{\boldsymbol \vphi}(w_1A_1, \dots, w_rA_r)\\
\nonumber
&\leq \sum_{\substack{w_1, \dots, w_r \\ \in \mathcal{W}(N) }}
\Bigg(\prod_{i=1}^r |h_i(w_i)|\Bigg)
\frac{1}{(w\widetilde W)^s}
\sum_{\substack{\v \in \\ (\ZZ/w \widetilde W \ZZ)^s}}
\prod_{i=1}^r 
\1_{\vphi_i(\v) \equiv 0 \Mod{w_i}} \\
&\leq \prod_{p|\widetilde W(N)}
\sum_{\substack{a_1, \dots, a_r \in \NN_0 }}
\Bigg(\prod_{i=1}^r |h_i(p^{a_i})|\Bigg)
\frac{1}{p^{a_{\max}s}}
\sum_{\substack{\v \in \\ (\ZZ/p^{a_{\max}} \ZZ)^s}}
\prod_{i=1}^r 
\1_{\vphi_i(\v) \equiv 0 \Mod{p^{a_i}}}, 
\end{align}
where $a_{\max} = \max(a_1, \dots, a_r)$.

If $n(\mathbf a)$ denotes the number of non-zero coordinates of 
$\mathbf a = (a_1, \dots, a_r)$, then the inner sum in the final expression above satisfies 
(cf.\ \cite[Lemma 1.3]{GT-linearprimes} and \cite[eqn (5.6)]{bm}):
\begin{equation} \label{eq:div_density_bounds}
\frac{1}{p^{a_{\max}s}}
\sum_{\substack{\v \in \\ (\ZZ/p^{a_{\max}} \ZZ)^s}}
\prod_{i=1}^r 
\1_{\vphi_i(\v) \equiv 0 \Mod{p^{a_i}}}
=
\begin{cases}
=1, &\mbox{if $n(\mathbf a)=0$,}\\
= p^{-\max_i \{a_i\}}, &  \mbox{if $p\gg_L 1$ and $n(\mathbf a)=1$,}\\
\leq p^{-\max_{i\neq j}\{a_i+a_j\}},
      & \mbox{if $p\gg_L 1$ and $n(\mathbf a)>1$,}\\
\ll_{L} p^{-\max_i \{a_i\}}, &
 \mbox{otherwise,}
\end{cases}
\end{equation}
where $L = \max_{1 \leq i \leq r} \{\|\vphi_i\|,r,s\}$
and where $\|\vphi_i\|$ denotes the maximum modulus of the coefficients of
$\vphi_i$.
Let $L_0 \gg_{L} 1$ be such that the second alternative applies to $p > L_0$ and
suppose further that $L_0 > H$.
Since, given any $k \in \NN_0$ there are at most $k^r$ tuples $(a_1, \dots, a_r)$ with $a_{\max} = k$
and taking into account that $|h_j(p^k)| \leq H^k$ and $h_j(p^k) \ll_{\eps} p^{\eps k}$,
we deduce that \eqref{eq:bounding_main_error} is bounded by
\begin{align*}
&\leq \prod_{p' \leq L_0}
\Bigg(1 + O_{L} \Big( \sum_{k\geq 1} \frac{{p'}^{k\eps}k^r}{{p'}^k}\Big) \Bigg) 
\prod_{\substack{p > H \\ p|\widetilde W(N)}}
\Bigg(1 + \frac{|h_1(p)| + \dots + |h_r(p)|}{p} + O_{L} \Big( \sum_{k\geq 2} \frac{H^{k}k^r}{p^k} \Big)\Bigg) \\
&\ll_{L,H,r} \prod_{p|\widetilde W(N)}
\Bigg(1 + \frac{|h_1(p)| + \dots + |h_r(p)|}{p} + O_{L,H,r}(p^{-2}) \Bigg) 
\ll_{L,H,r} \prod_{p|\widetilde W(N)} \prod_{i=1}^r
\Bigg(1 + \frac{|h_i(p)|}{p} \Bigg).
\end{align*}
This completes the proof of Theorem \ref{t:main} and leaves us with the task of establishing 
Theorem \ref{t:majorants-s2}, i.e.\ the existence of families of pseudorandom majorants for elements of 
$\mathcal{F}^*$.

\section{Majorants for multiplicative functions} \label{s:majorants-intro}
The following three sections are devoted to the construction of suitable majorant functions for elements 
of the class $\mathcal{F}^*$.
We will define in Section \ref{s:linearform} what it means for a majorant function to be pseudorandom.
Here, we will show that for every $N \in \NN$, every $\gamma \in (0,1/2)$ and every $h \in \mathcal{F}^*$, 
there exists a function $\nu_h= \nu_h^{(N)}: \{1, \dots, N\} \to \RR_{\geq0}$ with the following properties:
\begin{enumerate}
 \item (\emph{Majorisation}). There exists an absolute constant $C>0$ such that 
 $$|h(n)| \leq C \nu_h^{(N)}(n)$$ 
 for all $N \in \NN$ and all $n \in \{1,\dots ,N\}$.
 (In Section \ref{s:linearform} we will see that it in fact suffices that this condition holds for each $N$
 outside certain exceptional subsets of $\{1,\dots ,N\}$).
 \item (\emph{Truncated divisor sum structure}).
 Outside an exceptional set, we have $$\nu_h^{(N)}(n) = \sum_{d|n, d \leq N^{\gamma}} \lambda_d$$ for suitable weights $\lambda_d$.
 \item (\emph{Average order condition}).
 If $B>1$ is a constant and if $1 < W' \leq (\log N)^{B}$ is an integer that is divisible by the primorial $W(N)$, then
 \begin{align*}
 |S_h(N;W',A)|
 \ll S_{\nu_h^{(N)}}(N;W',A)
 \ll_B \frac{W'}{ \phi(W')}
 \frac{1}{\log N} 
 \prod_{\substack{p \leq N \\ p \nmid W'}}
 \left( 1 + \frac{|h(p)|}{p}\right)
 \end{align*}
 whenever $\gcd(A,W')=1$.
\end{enumerate}
Condition (2) is not a necessary condition. We will, however, make essential use of the truncated divisor sum
structure in order to show in Section \ref{s:linearform} that the majorants that we are about to construct 
are indeed pseudorandom. 
Condition (3) will be established in Section \ref{s:average-order}. 
This condition is important since, when applying the machinery from \cite{GT-linearprimes} to prove asymptotic 
results on correlations of the form \eqref{eq:aim}, 
the mean value $S_{\nu_{h_j}^{(N)}}(N)$ of the chosen majorant for each individual function $h_j$ 
will enter as a factor in the error term of the asymptotic formula.
For this reason, we require our majorant functions to have as small as possible average order.

\subsection*{Initial set-up}
For any function $h \in \mathcal{F}^*$, we obtain a simple majorant function 
$h': \NN \to \RR$ by setting $h'(n) := h^{\sharp}(n)h^{\flat}(n)$, where
$h^{\sharp}$ and $h^{\flat}$ are multiplicatively defined by
$$
h^{\sharp}(p^k) = \max(1, |h(p)|, \dots , |h(p^k)|)
$$
and
$$
h^{\flat}(p^k) = \min(1, |h(p^k)|),
$$
respectively. 
It is immediate that $|h(n)| \leq h'(n)$ for all $n \in \NN$.
The function $h^{\sharp}$ belongs to the class of functions for which 
pseudo-random majorants were already constructed in \cite[\S7]{bm}.
Our pseudo-random majorant for $h$ will arise as a product of separate majorants 
for $h^{\sharp}$ and $h^{\flat}$.
Thus, our main task in the present work is to construct general pseudo-random majorants for  
bounded multiplicative functions.

\section{A majorant for $h^{\sharp}$}\label{s:largemajorant}
Before we turn to the case of bounded multiplicative functions, let us  
record the known family of pseudo-random majorants $\nu_{\sharp}^{(N)}: \{1, \dots, N\} \to \RR_{\geq 0}$ 
for $h^{\sharp}$. 

For this purpose, set $g^{\sharp} = \mu * h^{\sharp}$ and define for any
$\gamma \in (0,1/2)$ and any sufficiently large $N$ the truncation
\begin{equation} \label{eq:h-sharp-truncation}
 h_{\gamma}^{(N)}(m) 
= \sum_{d \in \NN} \1_{d|m}g^{\sharp}(d)
  \chi\Big(\frac{\log d}{\log N^{\gamma}}\Big)
\end{equation}
of the convolution $h^{\sharp} = \1*g^{\sharp}$, where $\chi:\RR \to \RR_{\geq 0}$ is a 
smooth function with support in $[-1,1]$, which is monotone on $[-1,0]$ and 
$[0,1]$ and has the property that $\chi(x)=1$ for 
$x\in [-\frac{1}{2},\frac{1}{2}]$. 
Note that $g^{\sharp}(d)= 0$ whenever $d$ has a prime factor $p$ such that $h^{\sharp}(p) = 1$,
i.e.\ such that $|h(p)| \leq 1$.
Thus, if we define the set of primes 
$$\mathcal{P}_{\sharp} : \{p \text{ prime }: |h(p)| > 1\}$$
and the set of integers
$$\<\mathcal{P}_{\sharp}\> = \{m \in \NN : p | m \Rightarrow p \in \mathcal{P}_{\sharp} \},$$
then we may restrict the summation in \eqref{eq:h-sharp-truncation} to $d \in \<\mathcal{P}_{\sharp}\>$.

With the truncation \eqref{eq:h-sharp-truncation} of $h^{\sharp}$ in place, \cite[Proposition 7.6]{bm} shows that  
for any fixed value of $\gamma \in (0, 1/2)$ we have
$$h^{\sharp}(n) \ll \nu_{\sharp}^{(N)}(n)+ \1_{n \in \mathcal{S}} h^{\sharp}(n)$$ 
for $n \leq N$, where the exceptional set is defined as
$$
\mathcal{S} = \left\{
n \leq N : 
\Big(\exists p. v_p(n) \geq \max\{2, C_1 \log_p \log N\}\Big)   
\text{ or}
\prod_{p \leq N^{1/(\log \log N)^3}} p^{v_p(n)} 
\geq N^{\gamma/\log\log N}
\right\},
$$
and where the majorant outside the exceptional set is given by
\begin{equation} \label{eq:nu_sharp} 
 \nu_{\sharp}^{(N)}(m) =
 \sum_{\kappa = 4/\gamma}^{[(\log \log N)^3]}
 \sum_{\lambda = \lceil \log_2 \kappa - 2\rceil}^{[\log_2((\log \log N)^3)]}
 \sum_{u \in U(\lambda ,\kappa)}
 H^\kappa \1_{u|m} h^{\sharp}(u)
 {h}_{\gamma}^{(N)}\left( \frac{m}{ \prod_{p|u} p^{v_p(m)}} \right) ,
\end{equation}
where each set $U(\lambda,\kappa)$ is a sparse subset of the integers 
up to $N^{\gamma}$ that is defined as follows.
Set $\omega(\lambda,\kappa)
= \left\lceil \frac{\gamma \kappa(\lambda + 3 - \log_2 \kappa)}{200}
\right\rceil
$
and $I_\lambda=[N^{1/2^{\lambda+1}}, N^{1/2^{\lambda}}]$. 
Then
$$
U(\lambda,\kappa) =
\begin{cases}
\{1\},     &\text{if }\kappa=4/\gamma \text{ and } \lambda =   \log_2\kappa -2,\\
\emptyset, &\text{if }\kappa=4/\gamma \text{ and } \lambda\neq \log_2\kappa -2,\\
\bigg\{ p_1 \dots p_{\omega(\lambda,\kappa)} :
  \begin{array}{l}
   p_i \in I_\lambda \mbox{ distinct primes} \\
   h^{\sharp}(p_i) \not= 1
  \end{array}
\bigg\}, & \mbox{if }\kappa > 4/\gamma.
\end{cases}
$$
We note for later reference that $u \in \<\mathcal{P}_{\sharp}\>$ whenever $u \in U(\lambda,\kappa)$
and that $p \geq N^{1/(\log \log N)^3}$ for any prime divisor $p|u$ of any $u \in U(\lambda,\kappa)$ 
that appears in the definition of $\nu_{\sharp}$.

An important technical property of the above majorant construction is that the set of 
integers $u \in \bigcup_{\kappa, \lambda} U(\lambda,\kappa)$ is fairly sparse.
In particular, one can show (cf.\ the computation at the end of \cite[\S7]{lmd}) 
that whenever $f$ is a multiplicative function that is bounded at primes
(e.g.\ $f(n)=h^{\sharp}(n)$ or $f(n)=h^{\sharp}(n)d(n)^r$), then
\begin{equation}\label{eq:c_0}
\sum_{\boldsymbol\kappa} \sum_{\boldsymbol\lambda} \sum_{\u}
\prod_{j=1}^r
\frac{H^{\kappa_j} |f(u_j)|}{u_j}
< \infty.
\end{equation}
Moreover, the following stronger version holds:
\begin{equation}\label{eq:c_0-lcm}
\sum_{\boldsymbol\kappa} \sum_{\boldsymbol\lambda} \sum_{\u}
\frac{H^{\kappa_1 + \dots + \kappa_r} |f(u_1) \dots f(u_r)|}{\lcm(u_1, \dots, u_r)}
< \infty.
\end{equation}
To prove \eqref{eq:c_0-lcm}, note that $\gcd(u_1, u_2) = 1$ for $u_i \in U(\lambda_i, \kappa_i)$, 
$i=1,2$ unless $\lambda_1=\lambda_2$.
Thus, \eqref{eq:c_0-lcm} is bounded by
\begin{equation*}
\sum_{\substack{\bd \in \NN^\ell, \ell>1:\\ d_1 + \dots +d_\ell = r}}
\prod_{j=1}^\ell \Bigg(
\sum_{\kappa_1, \dots, \kappa_{d_j}} 
\sum_{\lambda = \lceil \log_2 \max_i \kappa_i -2 \rceil }^{[\log_2((\log \log N)^3)]} 
\sum_{\substack{\u \in \NN^{d_j} \\ u_i \in U(\lambda, \kappa_i)}}
\frac{H^{\kappa_1 + \dots + \kappa_{d_j}} |f(u_1) \dots f(u_{d_j})|}{\lcm(u_1, \dots, u_{d_j})} \Bigg),
\end{equation*}
and it suffices to show that for every $1 \leq d_j \leq r$, the respective factor above converges.
To show this, let $C>1$ be a constant such that $|f(p)| \leq C$ for all primes and let $1 \leq d \leq r$.
Then, the above factor satisfies
\begin{align*}
&\sum_{\boldsymbol\kappa} \sum_{\lambda = \lceil \log_2 \max_i \kappa_i -2 \rceil }^{[\log_2((\log \log N)^3)]} 
\sum_{\substack{\u \in \NN^d: \\ u_i \in U(\lambda, \kappa_i)}}
\frac{ H^{\kappa_1 + \dots + \kappa_d} |f(u_1) \dots f(u_d)|}
     {\lcm(u_1, \dots, u_d)}\\
&\leq \sum_{\boldsymbol\kappa} \sum_{\lambda = \lceil \log_2 \max_i \kappa_i -2 \rceil }^{[\log_2((\log \log N)^3)]} 
\sum_{\substack{\u \in \NN^d: \\ u_i \in U(\lambda, \kappa_i)}}
\prod_{j=1}^r
\frac{ H^{\kappa_1 + \dots + \kappa_d} C^{\omega(\lambda, \kappa_1) + \dots + \omega(\lambda, \kappa_d) } }
     {\lcm(u_1, \dots, u_d)}.
\end{align*}
Since 
$\omega(\lcm(u_1, \dots, u_d)) 
\geq \max_i \omega(\lambda, \kappa_i) 
\geq \frac{1}{d}\sum_{1 \leq i \leq d} \omega(\lambda,\kappa_i) =: \omega^*(\lambda, \boldsymbol{\kappa})$, the above is bounded by
\begin{align*}     
&\leq \sum_{\boldsymbol\kappa} \sum_{\lambda = \lceil \log_2 \max_i \kappa_i -2 \rceil }^{[\log_2((\log \log N)^3)]} 
H^{\kappa_1 + \dots + \kappa_d} C^{\omega(\lambda, \kappa_1) + \dots + \omega(\lambda, \kappa_d) }
\sum_{t \geq \omega^*(\lambda, \boldsymbol{\kappa})} \frac{1}{t!}
\Big( \sum_{p \in I_{\lambda}} \frac{1}{p}\Big)^{t}.
\end{align*}
Recall that each interval $I_{\lambda}$ is of the form $[y,y^2]$ with
$y > N^{1/O((\log \log N)^3)}$, i.e.\ $y \to \infty$ as $N \to \infty$. Thus, the above is bounded by
\begin{align} \label{eq:c_0-lcm-proof-1}
\nonumber
&\leq \sum_{\boldsymbol\kappa} \sum_{\lambda = \lceil \log_2 \max_i \kappa_i -2 \rceil }^{[\log_2((\log \log N)^3)]} 
H^{\kappa_1 + \dots + \kappa_d} C^{\omega(\lambda, \kappa_1) + \dots + \omega(\lambda, \kappa_d) }
\sum_{t \geq \omega^*(\lambda, \boldsymbol{\kappa})} \frac{1}{t!}
\Big(\log 2 + o(1)\Big)^{t} \\
&\leq \sum_{\boldsymbol\kappa} \sum_{\lambda \geq 1} 
H^{\kappa_1 + \dots + \kappa_d} C^{\gamma \lambda \kappa^*/200}
\sum_{t \geq \gamma \lambda \kappa^*/200} \frac{1}{t!}
\Big(\log 2 + o(1)\Big)^{t},
\end{align}
where $\kappa^* = \frac{1}{d}  \sum_{1\leq i \leq d} \kappa_i$.
Using a standard tail estimate for the exponential series, we have
\begin{align*}
\sum_{t \geq \gamma \lambda \kappa^*/200} \frac{1}{t!}
\Big(\log 2 + o(1)\Big)^{t}
&\ll \frac{1}{[\gamma \lambda \kappa^*/200]!} 
    \Big(\log 2 + o(1)\Big)^{\gamma \lambda \kappa^*/200}\\
&\ll \Big(\frac{200 e (\log 2 + o(1))}{\gamma \lambda \kappa^*}\Big)^{\gamma \lambda \kappa^*/200}\\
&\ll \prod_{1 \leq i \leq d}
     \Big(\frac{200 e (\log 2 + o(1))}{\gamma \lambda \kappa^*}\Big)^{\gamma \lambda \kappa_i/(200d)}\\
&\ll \prod_{1 \leq i \leq d}
     \Big(\frac{200 e d (\log 2 + o(1))}{\gamma \lambda \kappa_i}\Big)^{\gamma \lambda \kappa_i/(200d)}.
\end{align*}
Thus, \eqref{eq:c_0-lcm-proof-1} is bounded by
\begin{align*}
 \ll \prod_{1 \leq i \leq d} \sum_{\kappa_i} \sum_{\lambda \geq 1} 
 H^{\kappa_i}
 \Big(\frac{200 e d C (\log 2 + o(1))}{\gamma \lambda \kappa_i}\Big)^{\gamma \lambda \kappa_i/(200d)}.
\end{align*}
This expression, finally, is bounded by
\begin{align*}
&\ll  \Bigg( \sum_{\kappa} 
 \frac{H^{\kappa}}{\kappa^{\gamma \kappa/(200d)}} \sum_{\lambda \geq 1} 
 \Big(\frac{200 e d C (\log 2 + o(1))\gamma^{-1}}{ \lambda }\Big)^{\gamma \lambda \kappa/(200d)} \Bigg)^d\\
&\ll  \Bigg( \sum_{\kappa} 
 \frac{H^{\kappa}}{\kappa^{\gamma \kappa/(200d)}} 
 \Bigg(
 \sum_{\lambda \geq 1} 
 \Big(\frac{200 e d C (\log 2 + o(1))\gamma^{-1}}{ \lambda }\Big)^{\gamma \lambda /(200d)} 
 \Bigg)^{\kappa}
 \Bigg)^d,
\end{align*}
which is seen to converge, since $\sum_{j\geq 1} A^j j^{-\alpha j} < \infty$ for all constants $A,\alpha>0$.
This completes the proof of \eqref{eq:c_0-lcm}.

\section{Majorants for bounded multiplicative functions}\label{s:majorant}
Any non-negative bounded multiplicative function $h^{\flat}$ has the property that whenever 
$m|n$ is a divisor that is coprime to its cofactor $n/m$, then $h^{\flat}(n) 
\leq h^{\flat}(m)$.
The key step in turning this simple observation into the construction of a 
pseudo-random majorant is to find a systematic way of assigning to any 
integer $n$ a suitable divisor $m$.
The main property this map must have is that the pre-image of any divisor $m$
should be easily reconstructible, a property which will allow us to swap the order 
of summation in later computations.
The following type of assignment already featured in Erd\H{o}s's work \cite{erdos} 
on the divisor function:

Given a cut-off parameter $N$ and an integer $n \in [N^{\gamma},N]$, let  
$D_{\gamma}(n)$ denote the largest divisor of $n$ that is of the form
$$\prod_{p < Q} p^{v_p(n)}, \quad (Q \in \NN)$$
but does not exceed $N^{\gamma}$.
If $m \in [1, N^{\gamma})$ is an integer then its inverse image takes the 
form
$$D_{\gamma}^{-1} (m) =
\bigg\{mm' \in [1,N] : 
\begin{array}{l}
P^+(m) < P^-(m')~~ \text{ and} \cr
mQ^{v_Q(m')}>N^{\gamma} \text{ if } Q=P^-(m')
\end{array}
\bigg\},$$
where $P^+(m)$, resp.\ $P^-(m)$, denote the largest, resp.\ smallest, 
prime factor of an integer $m$. 
For our purpose it turns out to be of advantage to restrict attention to divisors 
$$m \in \<\mQ\> = \{m: p|m \implies p \in \mQ\},$$
where\footnote{We note as an aside that by Lemma \ref{l:elliott} the values of $h_{\flat}$ at higher 
prime powers do not influence the asymptotic order of $S_{h_{\flat}}(N)$ and need for this reason 
not be taken into account in the construction of $\nu_{\flat}$.}
$$ \mQ = \{p: |h(p)| < 1\}.$$
Thus, if $n \in [1,N]$ is an integer that factorises as $n=mq$ with 
$m \in \<\mQ\>$ and $p \not\in \mQ$ for all prime divisors $p|q$, then we set
$$D'_{\gamma}(n)
=
\begin{cases}
m & \text{if }  m \leq N^{\gamma}\cr
D_{\gamma}(m) & \text{if }  m > N^{\gamma}
\end{cases}.
$$
Our next aim is to show that a sufficiently smoothed version of the function 
$D'_{\gamma}$ can be written as a truncated divisor sum.
To detect whether a given divisor $m|n$ is of the form
$\prod_{p\leq Q} p^{v_p(n)}$ for some $Q$ or, equivalently, whether 
$q=\frac{n}{m}$ has no prime factor $p \leq Q$, we make use of a sieve 
majorant similar to the one considered in \cite[Appendix D]{GT-linearprimes}.
The two essential differences are that the parameter corresponding to $Q$ cannot 
be fixed in our application and that the divisor sum will be restricted to 
elements of the set $\<\mQ\>$.
Thus, let $\sigma_{\flat}: \RR \times \NN \to \RR_{\geq 0}$ be defined as
$$
\sigma_{\flat}(Q; q) =
\Bigg( \sum_{\substack{\delta|q\\ \delta\in \<\mQ\> }} \mu(\delta) 
\chi\left( \frac{\log \delta}{\log Q} \right) \Bigg)^2,
$$
where $\chi:\RR \to \RR_{\geq 0}$ is a smooth function with support in $[-1,1]$ 
and the property $\chi(x)=1$ for $x\in [-\frac{1}{2},\frac{1}{2}]$.
This yields a non-negative function with the property that 
$\sigma_{\flat}(Q; q) = 1$ if $q$ is free from prime factors $p \in \mQ$ with 
$p \leq Q$.
Setting, for $1 \leq n \leq N$,
\begin{align*}
 \nu'(n) = 
 \sum_{\substack{ m|n \\ m\in\<\mQ\> \\ m < N^{\gamma}}}
 h^{\flat}(m)
 \sigma_{\flat}\left(N^{\gamma}; \frac{n}{m}\right)
 +\sum_{
   \substack{ Q \leq N^{\gamma} \\ Q \in \mQ } }
 \sum_{\substack{ m|n \\ m\in\<\mQ\> \\ p|m \Rightarrow p<Q}} 
 \sum_{Q^k|n} 
 h^{\flat}(mQ^k)
 \1_{m < N^{\gamma}}
 \1_{Q^k m \geq N^{\gamma}}
 \sigma_{\flat}\left(Q; \frac{n}{m Q^k}\right),
\end{align*}
we obtain a (preliminary) majorant $\nu': \NN \to \RR_{\geq 0}$ for 
$h^{\flat}$.
The first of two small modifications consists of inserting smooth cut-offs for 
$m$ and $Q^km$, leading to the majorant function $\nu'': \NN \to \RR_{\geq 0}$ 
defined as
\begin{align*}
 \nu''(n) = 
 & \sum_{\substack{ m|n \\ m\in \< \mQ\> }}
 h^{\flat}(m)
 \chi\left( \frac{\log m}{\log N^{\gamma}} \right)
 \sigma_{\flat} \left(N^{\gamma}; \frac{n}{m}\right) \\
&+
 \sum_{
   \substack{ Q \in \mQ} }
 \sum_{\substack{ m|n \\ m\in \<\mQ\> \\ p|m \Rightarrow p<Q}} 
 \sum_{Q^k|n}
 h^{\flat}(m Q^k)
 \lambda\left( \frac{\log mQ^k}{ \log N^{\gamma}} \right)
 \sigma_{\flat} \left(Q; \frac{n}{m Q^k}\right),
\end{align*}
where $\lambda$ is a smooth cut-off of the interval $[1, 2]$ which is 
supported in $[1/2, 4]$ and takes the value $1$ on the interval 
$[1, 2]$.
To carry out the second simplification, we exclude a sparse exceptional set
related to the one from Section \ref{s:largemajorant}.
If $Q \leq N^{\gamma/(\log \log N)^3}$, then an integer $n < N$ certainly belongs to the 
exceptional set $\mathcal{S}$ from Section \ref{s:largemajorant} if it has a 
divisor of the form $Q^k m > N^{\gamma}$, where $p|m \Rightarrow p<Q$.
If $Q > N^{\gamma/(\log \log N)^3} = (\log N)^{\gamma (\log N)/(\log \log N)^4}$ 
and if $N$ is sufficiently large, then $Q^2>(\log N)^{C_1}$ so that any multiple 
of $Q^2$ again belongs to $\mathcal{S}$.
Thus, by defining
\begin{align} \label{eq:nu_flat}
 \nu_{\flat}^{(N)}(n) =
& \sum_{\substack{ m|n \\ m\in \< \mQ\> }}
 h^{\flat}(m)
 \chi\left( \frac{\log m}{\log N^{\gamma}} \right)
 \sigma_{\flat} \left(N^{\gamma}; \frac{n}{m}\right) \\
\nonumber
&+ \sum_{
   \substack{
   Q|n \\ Q \in \mQ 
   \\ Q > N^{\gamma/(\log \log N)^3} }}
 \sum_{\substack{ m|n \\ m \in \< \mQ \> \\ p|m \Rightarrow p<Q}}
 h^{\flat} (Qm)
 \lambda\left( \frac{\log Qm}{\log N^{\gamma}} \right)
 \sigma_{\flat} \left(Q; \frac{n}{Qm}\right),
\end{align}
we obtain a non-negative function with the property that
$$
h^{\flat}(n) \ll  \nu_{\flat}^{(N)}(n)  + \1_{\mathcal{S}}(n)
$$
for any integer $n \in [1,N]$, provided $N$ is sufficiently large with respect to 
the value of $C_1$ from the definition of $\mathcal{S}$.
By construction (cf. \cite{erdos} or \cite[Lemmas 3.2 and 3.3]{lmd}) it 
follows that
$\sum_{n\leq N} \1_{\mathcal{S}}(n) \ll N (\log N)^{-C_1/2}$.
Thus, choosing $C_1$ sufficiently large, we may in view of 
\eqref{eq:elliott-kish} and property (iii) of Definition \eqref{d:M} ensure that
$$
\sum_{n\leq N} (1+|h(n)|)\1_{\mathcal{S}}(n) = o(\sum_{n \leq N} |h(n)|),
$$
i.e., that $\mathcal{S}$ is indeed an exceptional set.

\section{The product $\nu_{\sharp} \nu_{\flat}$ is pseudo-random} \label{s:linearform}
The aim of this section is to show that the majorant functions constructed in the previous section 
are pseudorandom, much in the sense of \cite[\S 6]{GT-linearprimes}.
However, since we are working with a larger class of functions, several adjustments to this notion are
required and we include for this reason a complete definition.

\begin{definition}[Family of D-pseudorandom majorants, cf.\  \cite{GT-linearprimes}] \label{def:majorant-family}
Let $(N_{\tau})_{\tau \in \mathcal{T}}$ be an unbounded sequence of natural numbers indexed by 
some index set $\mathcal{T}$. 
Let   
$$(f_1^{(\tau)}, \dots, f_r^{(\tau)}: \{1, \dots, N_{\tau}\} \to \CC)_{\tau \in \mathcal{T}}$$ 
be a family of $r$-tuples of functions, and let 
$D \geq r$ be an integer.
Then a family of functions $$(\nu^{(\tau)}: \{1, \dots, N_{\tau}\} \to \RR_{>0})_{\tau \in \mathcal{T}}$$ 
is called a family of \emph{$D$-pseudorandom majorants} for the given family of functions if
the following four conditions hold:
 \begin{enumerate}
 \item (\emph{Normalisation}). We have
 $$
 \frac{1}{N_{\tau}} \sum_{n\leq N_\tau} \nu^{(\tau)}(n) = 1 +o_{N_\tau \to \infty}(1), \qquad (\tau \in \mathcal{T}).
 $$
 \item (\emph{Majorisation condition}). There exists an absolute constant $C>0$ 
 and weights  $E_{f_1}^{(\tau)}, \dots, E_{f_r}^{(\tau)} \in \RR_{>0}$ such that 
 $$|f_i^{(\tau)}(n)| \leq C \nu^{(\tau)}(n) E_{f_i}^{(\tau)}$$ 
 for all $\tau \in \mathcal{T}$, all $n\leq N_{\tau} \setminus \mathcal{S}(N_\tau)$ and every $1 \leq i \leq r$, 
 where $\mathcal{S}(N_{\tau})$ is any exceptional set with the property that, 
 if $B \geq 0$ is fixed and $T \in [N_{\tau} (\log N_{\tau})^{-B}, N_{\tau}]$, then, as $N_{\tau}, T \to \infty$, we have
 \begin{align*}
&\sum_{\n \in \ZZ^s \cap T\fK} \prod_{i=1}^r f_i^{(\tau)}(\vphi_i(\n)) 
\1_{\vphi_i(\n) \not\in \mathcal{S}(N)}\\
&= \sum_{\n \in \ZZ^s \cap T\fK} \prod_{i=1}^r f_i^{(\tau)}(\vphi_i(\n))
+o_{T \to \infty} \bigg( \vol(T\fK) \prod_{i=1}^r S_{|f_i^{(\tau)}|}(T) \bigg)
\end{align*}
for any system of linear polynomials $\vphi_1, \dots, \vphi_r \in \ZZ[x_1, \dots, x_s]$ 
as in Theorem \ref{t:main} with $r,s,L \leq D$,
and for any $\mathfrak{K} \subset \RR^s$ as in Theorem \ref{t:main}.
 \item (\emph{Linear forms condition}). 
 For each $T \in \NN$, let $T' \in (T, O_D(T))$ be a prime that is sufficiently large for the following to hold.
 Let $1 \leq d,t \leq D$ and let $\phi_1, \dots \phi_t \in \ZZ[X_1, \dots, X_d]$ be any system of 
 linear polynomials such that $|\phi_1(\0)|, \dots, |\phi_d(\0)| \leq DT$ and such that all other coefficients 
 are bounded by $D$ in absolute value. Suppose further that  $\phi_i(\n) \geq 0$ holds for all 
 $\n \in \{1, \dots, T\}^d$ and all $1\leq i \leq t$. 
 Then $0 \leq \phi_i(\n) < T'$ for all $\n \in \{1, \dots, T \}^d$, i.e.\ 
 the embedding $\iota: \ZZ/T'\ZZ \hookrightarrow \NN_0$ that sends each class to its smallest non-negative 
 representative induces the identity
 $$\iota (\phi_i(\n) \Mod{T'}) =  \phi_i(\n), \qquad (\n \in \{1, \dots, T \}^d).$$
 
 The family $\{\nu^{(\tau)} \}_\tau$ then satisfies the $D$-linear forms condition if the following holds:
 For any constant $B \geq 0$, any sufficiently large $N_{\tau}$,
 any $T\in [N_{\tau}(\log N_{\tau})^{-B},N]$ and any $T' \ll_D T$ as above, 
 the family of functions ${\tilde \nu}_{\tau,T}: \{1, \dots, T'\} \to \RR_{>0}$, defined via 
 ${\tilde \nu}_{\tau,T}(n) = \frac{1}{2}\{1 + \nu^{(\tau)}(n)\1_{n\leq T} + \1_{n > T}\}$ for $1 \leq n \leq T'$,
 satisfies the following estimate.
 Identifying $\ZZ/T'\ZZ$ with $\{1, \dots, T'\}$, we have, as $N_{\tau} \to \infty$,
 $$
 \frac{1}{{T'}^d}\sum_{\n \in (\ZZ/T'\ZZ)^{d}} \prod_{1 \leq j \leq t} {\tilde\nu}_{\tau,T}(\phi_j(\n)) 
 = (1 + o_{B,D}(1))
 \Bigg( \frac{1}{T'} \sum_{n \in \ZZ/T'\ZZ}  {\tilde\nu}_{\tau,T}(n) \Bigg)^t
 $$
 for all systems $\phi_1, \dots, \phi_t \in \ZZ[X_1, \dots, X_d]$ of linear polynomials 
 as above with the additional property that their non-constant parts form a system of 
 pairwise linearly independent forms.
 \item (\emph{Correlation condition}).
 The family of functions $(\nu^{(\tau)})_{\tau \in \mathcal{T}}$ satisfies the $D$-correlation condition if,
 for every constant $B \geq 0$, for every sufficiently large $N_\tau$, 
 all integers $T \in [N_{\tau}(\log N_\tau)^{-B}, N_{\tau}]$ and primes $T'$ as in (3),
 and for every $1 < m \leq D$, there exists a function $\varpi_m : \ZZ/T' \ZZ \to \RR_{>0}$ such that
 for all $q \in \NN$ and all $m$, $N_{\tau}$, $T$ and $T'$ as before the moment condition
 $$ \frac{1}{T'} \sum_{n \in \ZZ/ T'\ZZ} (\varpi_m(n))^q \ll_{m,q,B} 1$$
 holds, and such that 
 $$ \frac{1}{T'} \sum_{n \in \ZZ/ T'\ZZ} \prod_{1\leq j\leq m} {\tilde \nu}_{\tau,T}(n+a_j)
 \ll_{m,q,B} \sum_{1\leq i <j \leq m} \varpi_m(a_i - a_j)$$
 for all tuples $(a_1, \dots, a_m) \in (\ZZ/ T'\ZZ)^m$
 and for the function ${\tilde \nu}_{\tau,T}$  defined in (3).
 \end{enumerate}
\end{definition}

Our aim is to prove the following result:

\begin{theorem} \label{t:pseudorandom}
 Let $D \geq 1$ be an integer, let $\gamma \in (0,1/2)$ and let $N>1$ be an integer parameter.
 Further, let $h_1, \dots, h_r \in \mathcal{F}^*$ and let $B>1$ be a constant.
 Suppose $W_1, \dots, W_r \in (1, (\log N)^{B})$ are integers, each divisible by $W(N)$ and 
 for every $i \in \{1, \dots, r\}$, let $A_i$ be a reduced residue modulo $W_i$. 
 Writing $\tau = (N, W_1, \dots, W_r, A_1, \dots, A_r)$, we define the $W$-tricked function
 $f_i^{(\tau)}: n \mapsto h_i(W_i n + A_i)$ together with the weight
 $E_{f_i}^{(\tau)} = E_{h_i}(N;W_i)$, where 
 \begin{equation*} 
 E_{h_i}(N;Q) := \frac{1}{\log N} \frac{Q}{\phi( Q )}
 \prod_{p \leq N, p \nmid Q} \left(1 + \frac{|h_i(p)|}{p} \right).
 \end{equation*}
 For each $1 \leq i \leq r$ and $N$, let $\nu_{h_i}^{(N)} = \nu_{\sharp}^{(N)} \nu_{\flat}^{(N)}$ denote the 
 majorant function for $h_i$ with parameter $\gamma$, as constructed in the previous three sections.\footnote{See, 
 in particular, \eqref{eq:nu_sharp} and \eqref{eq:nu_flat}.} 
 Let $W=\lcm(W_1, \dots, W_r)$ and set $N_\tau = \lfloor N/ W \rfloor$.
 Then the family of functions $\nu^{(\tau)} : \{1, \dots, N_\tau \} \to \RR_{> 0}$ defined by
 $$\nu^{(\tau)}(n) = 
 \frac{1}{r}
 \sum_{i=1}^r
 \frac{\nu_{h_i}^{(N)}(W_in+A_i)}{E_{f_i}^{(\tau)}}$$
 is a family of $D$-pseudorandom majorants for $f_1^{(\tau)}, \dots, f_r^{(\tau)}$ with weights 
 $E_{f_1}^{(\tau)}, \dots, E_{f_r}^{(\tau)}$,
 provided $\gamma$ was chosen sufficiently small with respect to $D$.
\end{theorem}

Since the majorisation condition holds for our majorant by construction, 
the proof of Theorem \ref{t:pseudorandom} reduces to establishing 
Propositions \ref{p:average-order}, \ref{p:linear-forms} and Proposition \ref{p:correlation} below.
To see this, one splits the majorant ${\tilde \nu}^{(T,N)}$ from Definition \ref{def:majorant-family} (3)
into its individual parts and proceeds as described in the subsection 
`Construction of the enveloping sieve' of \cite[Appendix D]{GT-linearprimes}. 

\begin{proposition}[Correct average order] \label{p:average-order}
Let $N >1$, $B>1$, 
$T\in[N(\log N)^{-B},N]$, and $W \in (1, (\log N)^B)$ with $W(N)|W$ be integers.
Suppose that $h \in \mathcal{F}^*$ and let $\nu_{h}^{(N)} = \nu_{\sharp}^{(N)} \nu_{\flat}^{(N)}$ 
denote the majorant function for $h$ with parameter $\gamma$, defined via 
\eqref{eq:nu_sharp} and \eqref{eq:nu_flat}.
Then
\begin{align} \label{eq:correct-average-order}
|S_h(T;W,A)|
\ll_B S_{\nu^{(N)}_h}(T;W,A)
\ll_B 
E_h(T;W)
\end{align}
whenever $\gcd(A,W)=1$.
\end{proposition}

\begin{proposition}[Linear forms estimate]\label{p:linear-forms}
Let $N>1$ be an integer parameter, let $B > 1$ be a constant and supposed that $T \in [N(\log N)^{-B},N]$.
For each $1 \leq i \leq r$, 
let $1 \leq W_i \leq (\log N)^B$ be an integer that is divisible by 
$W(N)$ and let $0 \leq A_i \leq W_i$ be coprime to $W(N)$.
Let $D>1$ be constant, let $1\leq r, s \leq D$ be integers and let   
$\phi_1, \dots, \phi_r \in \ZZ[X_1, \dots X_s]$ be linear polynomials
whose non-constant parts are pairwise linearly independent and 
which have the property that $|\phi_1(\0)|, \dots, |\phi_r(\0)| \leq DT$ 
and that all other coefficients are bounded by $D$ in absolute value.
Suppose that $\fK \subset [-1,1]^s$ is a convex body with 
$\vol(\fK) > 0$ and that
$W_i\phi_i(T\fK)+A_i \subset [1,W_iT]$ for each $1\leq i \leq r$.
Then, as $N \to \infty$, 
the following estimate holds uniformly in $T$, $W_1, \dots, W_r$, and $\phi_1, \dots, \phi_r$ as above,
provided the parameter $\gamma$ defining $\nu_{h_1}^{(N)}, \dots, \nu_{h_r}^{(N)}$ is sufficiently small with respect to $D$:
$$
\frac{1}{T^s\vol \fK}
\sum_{\n \in \ZZ^s \cap T\fK} \prod_{i=1}^r 
\nu_{h_i}^{(N)}(W_i\phi_i(\n)+A_i)
= 
\left(1 + o(1)\right)
\prod_{i=1}^r
\left( \frac{1}{T} \sum_{n \leq T} \nu_{h_i}^{(N)}(W_in+A_i)\right).
$$
\end{proposition}

\begin{proposition}[Correlation estimate] \label{p:correlation}
Let $B,D>1$ be integers, let $N > 1$ an integer parameter,
let $\nu_{h_1}^{(N)}, \dots, \nu_{h_r}^{(N)}$ 
and $\tau = (N, W_1, \dots, W_r, A_1, \dots, A_r)$ be as in Theorem \ref{t:pseudorandom},
and write $W = \lcm(W_1, \dots,W_r)$.
Then there exists for any $1 \leq d \leq D$ a family of functions 
$\sigma_{\tau,d}: \{-\lfloor N/W \rfloor,\dots, \lfloor N/W \rfloor\} \to \RR_{\geq 0}$
which satisfies the moment bounds
$$
\frac{W}{T}\sum_{-T/W \leq m \leq T/W} \sigma^q_{\tau,d}(m) \ll_{d,q,B} 1, \qquad (q \in \NN),
$$
uniformly for all $\tau$ as in Theorem \ref{t:pseudorandom} and all $T\in [N(\log N)^{-B},N]$, 
and which is such that for all $T\in [N(\log N)^{-B},N]$,
for every $d$-tuple $(i_1, \dots, i_{d}) \in \{1,\dots,r\}^{d}$ 
and every choice of (not necessarily distinct) $a_1,\dots,a_{d} \in \{1, \dots, \lfloor T/W \rfloor\}$, we have
\begin{align} \label{eq:correlation-est}
 \frac{W}{T}
 \sum_{1 \leq m \leq T/W} \prod_{j=1}^d 
 \frac{\nu_{h_{i_j}}^{(N)}(W_{i_j}(m+a_j) + A_{i_j}) }{E_{h_{i_j}}(N;W_{i_j})}
 \leq  \sum_{1 \leq j < j' \leq d} \sigma_{\tau,d}(a_j - a_{j'}),
\end{align}
provided the parameter $\gamma$ defining $\nu_{h_1}^{(N)}, \dots, \nu_{h_r}^{(N)}$ is small enough
with respect to $D$.
\end{proposition}

\begin{proof}[Proof of Proposition \ref{p:linear-forms}]
For each $1 \leq i \leq r$, write $\tilde \phi_i (\n) = W_i\phi_i(\n)+A_i$.
Inserting all definitions and combining the two parts of the majorant $\nu^{\flat}_{h_j}$
into one by extending the summation in 
$Q_j$ to include $Q_j=1$, we obtain
\begin{align}\label{eq:linear}
&\frac{1}{T^s\vol \fK}
 \sum_{\n \in \ZZ^s \cap T\fK} \prod_{j=1}^s \nu_{h_j}(\tilde \phi_j(\n)) \\
\nonumber
 & =\frac{1}{T^s\vol \fK} \sum_{\n \in \ZZ^s \cap T\fK} \prod_{j=1}^s 
   \nu^{\sharp}_{h_j}(\tilde \phi_j(\n))
   \nu^{\flat}_{h_j}(\tilde \phi_j(\n)) \\
 \nonumber  
 &=  \frac{1}{T^s\vol \fK}
 \sum_{\n \in \ZZ^s \cap T\fK} \prod_{j=1}^r \Bigg\{
 \sum_{\kappa_j = 4/\gamma}^{[(\log \log N)^3]}
 \sum_{\lambda_j 
        = \lceil \log_2 \kappa_j - 2\rceil}^{[\log_2((\log \log N)^3)]}
 \sum_{\substack{u_j \in \<\mathcal{P}_{\sharp}^{(j)}\>  \\ u_j \in U_j(\lambda_j ,\kappa_j)}}
 \sum_{\substack{d_j \in \<\mathcal{P}_{\sharp}^{(j)}\>  \\ \gcd(d_j,u_jW_j)=1}}\\
 \nonumber
& \quad \qquad \times
  \sum_{\substack{ Q_j \in \mQ^{(j)} \cup \{1\}: \\ Q_j = 1 \text{ or} \\ Q_j > N^{\gamma/(\log \log N)^3}}}
 \sum_{\substack{ m_j \in \<\mQ^{(j)}\>:
                 \\ \gcd(m_j,W_j) = 1, \\ (Q_j>1 \wedge p|m_j) \Rightarrow  p<Q_j }}
 \sum_{\substack{\delta_j, \delta'_j \in \<\mQ^{(j)}\> 
 \\ \gcd(\delta_j \delta'_j,W_j)=1}}
  ~~\1_{\lcm(\delta_jm_jQ_j,\delta'_jm_jQ_j,u_j, d_j)|\tilde \phi_j(\n)} 
\\
\nonumber
& \qquad \qquad \times H^{\kappa_j} h^{\sharp}_j(u_j)
 g_j^{\sharp}(d_j) \mu(\delta_j) \mu(\delta'_j)
 h^{\flat}_j(m_j) h^{\flat}_j(Q_{j}) \\
 \nonumber
& \qquad \qquad  \times \lambda\left(\frac{\log Q_jm_j}{\log N^{\gamma}}\right)
 \chi\left(\frac{\log d_j}{\log N^{\gamma}}\right)
 \chi\left(\frac{\log \delta_j}{\log Q_j^* }\right)
 \chi\left(\frac{\log \delta'_j}{\log Q_j^*}\right) \Bigg\},
\end{align} 
where $Q_j^* := Q_j (N^{\gamma})^{\1_{Q_1=1}}$ takes the value $Q_j$ unless $Q_j=1$ and thus $Q_j^*=N^{\gamma}$.

Our aim is to show that in the final expression above the summation in $\n$ and the product can be interchanged
at the expense of a small error, i.e.\ that \eqref{eq:linear} equals
\begin{align} \label{eq:linear-1}
 &(1+o(1))
 \sum_{\boldsymbol{\kappa},\boldsymbol{\lambda} ,\u, \Q}
 \sum_{~\m, \boldsymbol\delta,\boldsymbol\delta', \bd~}
 \prod_{j=1}^r
 \frac{1}{T^s\vol \fK}
 \sum_{\n_j \in \ZZ^s \cap T\fK} \Bigg\{
 ~~\1_{\Delta_j|\tilde \phi_j(\n_j)}  H^{\kappa_j} h^{\sharp}(u_j)
 g_j^{\sharp}(d_j)  \\
\nonumber 
& \qquad \times 
 h^{\flat}(Q_{j}) 
 h^{\flat}(m_j) 
 \mu(\delta_j) 
 \mu(\delta'_j)
 \lambda\left(\frac{\log Q_jm_j}{ \log N^{\gamma}}\right)
 \chi\left(\frac{\log d_j}{\log N^{\gamma}}\right)
 \chi\left(\frac{\log \delta_j}{\log Q_j^*}\right)
 \chi\left(\frac{\log \delta'_j}{\log Q_j^*}\right) \Bigg\},
\end{align}
where
$$
\Delta_j= \lcm(\delta_j m_jQ_j,\delta'_j m_jQ_j, u_j, d_j) = d_ju_jQ_jm_j \lcm(\delta_j,\delta'_j),
$$
and where $\boldsymbol \kappa$, $\boldsymbol \lambda$, $\u$, $\Q$, $\m$, $\boldsymbol{\delta}$, $\boldsymbol{\delta'}$ 
and $\bd$ run over $r$-tuples of integers that satisfy component-wise the summation conditions imposed in \eqref{eq:linear}.

The proof that `\eqref{eq:linear} = \eqref{eq:linear-1}' 
follows the approach of \cite[\S9]{bm} closely, but the situation here bears a few extra difficulties. 
All essential tools in this analysis were derived or developed starting out from 
material in \cite[Appendix D]{GT-linearprimes}. 

Multiplying out the product in \eqref{eq:linear}, collecting together $r$-tuples and slightly rearranging them, 
the first step in our proof, as in \cite[\S9]{bm}, is to show that the summation over $r$-tuples
$\boldsymbol \kappa$, $\boldsymbol \lambda$, $\u$, $\Q$, $\m$, $\boldsymbol{\delta}$, $\boldsymbol{\delta'}$ and $\bd$
can be restricted to such tuples that satisfy the additional conditions that
for every $1\leq i \leq r$ we have
$\gcd(Q_i m_i \delta_i \delta'_i d_i, \prod_{j=1}^r u_j)=1$ and 
$\gcd(m_i \delta_i \delta'_i d_i, \prod_{j=1}^r Q_j )=1$,
and that all tuples $(u_1 \dots, u_r)$ consist of pairwise co-prime integers and 
all prime-tuples $(Q_1, \dots, Q_r)$ have pairwise distinct entries.

To prove this, observe first that all the above coprimality conditions are automatically satisfied for the factors with 
index $j=i$.
Thus, if any choice of tuples 
$\boldsymbol \kappa$, $\boldsymbol \lambda$, $\u$, $\Q$, $\m$, $\boldsymbol{\delta}$, $\boldsymbol{\delta'}$ and $\bd$ 
violates the above additional conditions, then at least one of the following three alternatives holds.
There is a prime factor $p|u_j$ for some $1 \leq j \leq r$ which divides two entries of $(\Delta_1, \dots, \Delta_r)$, or, 
for some $1 \leq j \leq r$, the prime $Q_j$ either divides two entries of $(\Delta_1, \dots, \Delta_r)$ or we have 
$Q_j^2|\Delta_i$ for some $1 \leq i \leq r$.
Recall that all prime factors $p$ of any $u_j \in U(\kappa_j,\lambda_j)$ satisfy the lower bound 
$p\geq N^{(\log \log N)^{-3}}$, and that also $Q_j > N^{(\log \log N)^{-3}}$.

We seek to apply the Cauchy--Schwarz inequality to show that the contribution of the excluded choices $r$-tuples
makes a negligible contribution.
If $\mathcal{T}$ denotes the set of all integers that are divisible by the square of a prime from the interval 
$[N^{(\log \log N)^{-3}} , N^{\gamma})$, and if $\1_{\mathcal{T}}$ denotes the corresponding characteristic function,  then
\begin{align*}
\sum_{\n \in \ZZ^s \cap T\fK}
\1_{\mathcal{T}}\Big(\prod_{1\leq i \leq r} \tilde\phi_i(\n)\Big) 
&=
\sum_{\substack{T^{(\log \log T)^{-3}} \\ \leq p < T^{\gamma}}} 
\sum_{\n \in \ZZ^s \cap T\fK}
\1_{p^2|\prod_i \tilde\phi_i(\n)} \\
&\ll_r 
\sum_{\substack{N^{(\log \log N)^{-3}} \\ \leq p < N^{\gamma}}} 
\frac{|\ZZ^s \cap T\fK|}{p^2}
\ll_r |\ZZ^s \cap T\fK| \exp\left(-\frac{\log N}{(\log \log N)^{3}}\right) .
\end{align*}
Next, we seek a bound on on the second moment of \eqref{eq:linear} with respect to the summation in $\n$.
Note that $H^{\kappa_j} \leq  H^{(\log \log N)^3}$, that 
$|\mu(\delta_j) \mu(\delta'_j) h^{\flat}(m_jQ_{j})| \leq 1$, and that $|\chi|, |\lambda| \leq 1$.
Further, 
$$0 \leq g^{\sharp}_j(d) < h^{\sharp}_j(d) \leq h^{\sharp}_j(\phi_j(\n))$$
for all divisors $d| \phi_j(\n)$.
Finally,
\begin{equation} \label{eq:d_7-bound}
 \sum_{\boldsymbol{\kappa},\boldsymbol{\lambda} ,\u, \Q}
 \sum_{~\m, \boldsymbol\delta,\boldsymbol\delta', \bd~}
  \prod_{j=1}^r \1_{\lcm(\delta_jm_jQ_j,\delta'_jm_jQ_j,u_j, d_j)|\tilde \phi_j(\n)} 
\leq \prod_{j=1}^r d_7(\tilde \phi_j(\n)),
\end{equation}
since $\Delta_j=\lcm(\delta_jm_jQ_j,\delta'_jm_jQ_j,u_j, d_j)|\tilde \phi_j(\n)$ implies the decomposition of
$\phi_j(\n)$ into seven factors as follows:
$$\phi_j(\n) = \gcd(\delta_j,\delta'_j) \frac{\delta_j}{\gcd(\delta_j,\delta'_j)}
 \frac{\delta'_j}{\gcd(\delta_j,\delta'_j)}  (Q_j m_j) u_j  d_j  \frac{\phi_j(\n)}{\Delta_j}, $$
where $Q_j m_j$ is regarded as one factor since $Q_j = P^+(Q_j m_j)$ is uniquely determined.
Thus,
\begin{align} \label{eq:p-linear-forms-2nd-moment}
\nonumber
&\sum_{\n \in \ZZ^s \cap T\fK} \bigg(
 \sum_{\boldsymbol{\kappa},\boldsymbol{\lambda} ,\u, \Q}
 \sum_{~\m, \boldsymbol\delta,\boldsymbol\delta', \bd~}
  \prod_{j=1}^r \1_{\Delta_j|\tilde \phi_j(\n)} 
  H^{\kappa_j} h^{\sharp}_j(u_j)
 g_j^{\sharp}(d_j)
\bigg)^2 \\
\nonumber
&\leq H^{2r(\log \log N)^3} 
\sum_{\n \in \ZZ^s \cap T\fK} 
\bigg( \prod_{j=1}^r
d_7(\tilde \phi_j(\n))h^{\sharp}_j(\tilde \phi_j(\n))g^{\sharp}_j(\tilde \phi_j(\n)) \bigg)^2\\
&\ll H^{2r(\log \log N)^3} |\ZZ^s \cap T\fK| (\log N)^{O_{H,r}(1)},
\end{align}
where the last step follows from the $k$-th moment bound \cite[Proposition 7.9]{bm}, applied to the function
$f(n) = d_7(n)(h^{\sharp}_j(n))^2$.
Indeed, the function $f$ satisfies the conditions of \cite[Proposition 7.9]{bm}. 
The proof of the proposition is easily adjusted to apply to a system of forms 
$(W_i \phi_i(\m) + A_i)_{1 \leq i \leq r}$, where $(\phi_i(\m))_{1 \leq i \leq r}$ has the same properties as 
the system $(\psi_i(\m))_{1 \leq i \leq r}$ in the statement.
(The $W$-trick would in fact allow us to remove the assumption that $|f(n)| \ll_{\eps} n^{\eps}$ from the statement.)

Putting everything together, the contribution from the excluded tuples is bounded above by
\begin{align*}
 &\frac{1}{|\ZZ^s \cap T\fK|}
 \sum_{\n \in \ZZ^s \cap T\fK} 
 \1_{\mathcal{T}}\Big(\tilde\phi_1(\n) \dots \tilde\phi_r(\n)\Big) 
 \bigg(
 \sum_{\boldsymbol{\kappa},\boldsymbol{\lambda} ,\u, \Q}
 \sum_{~\m, \boldsymbol\delta,\boldsymbol\delta', \bd~}
  \prod_{j=1}^r \1_{\Delta_j|\tilde \phi_j(\n)} 
  H^{\kappa_j} h^{\sharp}_j(u_j)
 g_j^{\sharp}(d_j) \bigg) \\
&\ll  H^{r(\log \log N)^3}  (\log N)^{O_{H,r}(1)}
\exp\left(-\frac{\log N}{2(\log \log N)^{3}}\right)
\ll_{H,r} \exp\left(-\frac{\log N}{3(\log \log N)^{3}}\right).
\end{align*}

(2) Similarly as above, one can also show that the summations in \eqref{eq:linear-1} can be restricted to tuples
such that for every $1\leq i \leq r$ we have
$\gcd(Q_i m_i \delta_i \delta'_i d_i, \prod_{j=1}^r u_j)=1$ and 
$\gcd(m_i \delta_i \delta'_i d_i, \prod_{j=1}^r Q_j )=1$,
and that all tuples $(u_1 \dots, u_r)$ consist of pairwise co-prime integers and 
all prime-tuples $(Q_1, \dots, Q_r)$ have pairwise distinct entries.
To see this, we bound
\begin{align*}
 &\sum_{\substack{\n_1, \dots, \n_r \\ \in \ZZ^s \cap T\fK}} 
 \1_{\mathcal{T}}\Big(\prod_{1\leq i \leq r} \tilde\phi_i(\n_i)\Big) 
 \bigg(
 \sum_{\boldsymbol{\kappa},\boldsymbol{\lambda} ,\u, \Q}
 \sum_{~\m, \boldsymbol\delta,\boldsymbol\delta', \bd~}
  \prod_{j=1}^r \1_{\Delta_j|\tilde \phi_j(\n_j)} 
  H^{\kappa_j} h^{\sharp}_j(u_j)
 g_j^{\sharp}(d_j) \bigg),
\end{align*}
in essentially the same way as before.

(3)
For any system of linear forms $\vphi_1, \dots, \vphi_r \in \ZZ[x_1, \dots, x_s]$ and for any prime $p$ let
\begin{equation} \label{eq:def-alpha}
 \alpha_{\boldsymbol{\vphi}}(p^{c_1}, \dots, p^{c_r}) 
= \frac{1}{p^{ms}} \sum_{\n \in (\ZZ/p^m\ZZ)^s} \prod_{i=1}^r 
\1_{p^{c_i}|\vphi_i(\n)},
\end{equation}
where $m = \max(c_1, \dots, c_r)$, and extend 
$\alpha_{\boldsymbol{\vphi}}$ to composite arguments multiplicatively 
(cf.\ \cite[p.1831]{GT-linearprimes} and \cite[Definition 8.4]{lmr}).

Suppose that for each $i \in \{1, \dots, r\}$, we are given coprime integers $W_i$ and $A_i$ and let 
$\tilde \vphi_i(\m) = W_i \vphi_i(\m) +A_i$.
If $n(\mathbf c)$ denote the number of non-zero components of 
$\mathbf c$ and if $\bphi$ has finite complexity in the sense of 
\cite{GT-linearprimes}, i.e., if the linear forms $\vphi_i$ 
are pairwise linearly independent, then the following asymptotic formulae hold
(cf. \cite[eqn (5.6)]{bm}):
\begin{equation}\label{eq:ev-alpha-1}
\alpha_{\boldsymbol{\tilde \vphi}}(p^{c_1},\dots,p^{c_r}) \begin{cases}
=1, &\mbox{if $n(\mathbf c)=0$,}\\
=0, &\mbox{if  $\exists i: c_i>0$ and $p | W_i$,}\\
= p^{-\max_i \{c_i\}}, &  \mbox{if $p\gg_L 1$, $(c_i>0 \Rightarrow p \nmid W_i)$ and $n(\mathbf c)=1$,}\\
\leq p^{-\max_{i\neq j}\{c_i+c_j\}},
      & \mbox{if $p\gg_L 1$ and $n(\mathbf c)>1$,}\\
\ll_{L} p^{-\max_i \{c_i\}}, &
 \mbox{otherwise,}
\end{cases}
\end{equation}
where $$L = \max_{1 \leq i \leq r} \{\|\vphi_i\|,r,s\}$$
and where $\|\vphi_i\|$ denotes the maximum modulus of the coefficients of
$\vphi_i$.
Writing, as before,
$$
\Delta_j= \lcm(\delta_j m_jQ_j,\delta'_j m_jQ_j, u_j, d_j),
$$
it follows by means of a lattice point counting argument (cf.\ \cite[Appendix A]{GT-linearprimes}) that
\begin{equation} \label{eq:lattice-point-count}
\frac{1}{x^s\vol \fK}\sum_{\n \in \ZZ^s \cap x\fK} \prod_{j=1}^r 
\1_{\Delta_j|\tilde \phi_j(\n)}
= \alpha_{\boldsymbol{\tilde \phi}}(\Delta_1, \dots, \Delta_r) 
\left(1 + O\left(\frac{x^{-1+ O_D(\gamma)}}{\vol \fK}\right)\right). 
\end{equation}
Thus, the expression \eqref{eq:linear} equals
\begin{align}\label{eq:linear-2}
\begin{split}
 \left(1 + O\left(\frac{T^{-1+O_D(\gamma)}}{\vol \fK}\right)\right) 
 &\times \\
 \dsum_{\boldsymbol{\kappa},\boldsymbol{\lambda} ,\u, \Q}
 \dsum_{~\m, \boldsymbol\delta,\boldsymbol\delta', \bd~}
 &\alpha_{\boldsymbol{\tilde \phi}}(\Delta_1, \dots, \Delta_r)
  \prod_{j=1}^r  
  H^{\kappa_j} h^{\sharp}_j(u_j)
 g_j^{\sharp}(d_j) \mu(\delta_j) \mu(\delta'_j)
 h^{\flat}_j(m_jQ_{j}) \\
& \times 
 \lambda\left(\frac{\log Q_jm_j}{\log N^{\gamma}}\right)
 \chi\left(\frac{\log d_j}{\log N^{\gamma}}\right)
 \chi\left(\frac{\log \delta_j}{\log Q_j^*}\right)
 \chi\left(\frac{\log \delta'_j}{\log Q_j^*}\right)
\end{split}
\end{align}
with the same summation conditions on all components as in \eqref{eq:linear} and 
where $\dsum$ indicates that the additional co-primality conditions from (1) also are in place.
Due to the additional co-primality conditions from (1), the argument of \eqref{eq:linear-2} equals
$$
\alpha_{\boldsymbol{\tilde\phi}}(\Delta_1, \dots, \Delta_r)
= \frac{\alpha_{\boldsymbol{\tilde\phi}}
(\tilde \Delta_1, \dots, \tilde\Delta_r)}{u_1 \dots u_r Q_1 \dots Q_r},$$
where $\tilde \Delta_i = \Delta_i/(Q_iu_i)$ for each $i$.

Applying the asymptotic evaluation \eqref{eq:ev-alpha-1} of $\alpha_{\boldsymbol{\tilde \phi}}$ 
and \eqref{eq:lattice-point-count} to the system of length $r=1$ in \eqref{eq:linear-1} and
taking into account step (2), our original aim translates as follows:
our task is to show that, although $\alpha_{\boldsymbol{\tilde \phi}}$ is not multiplicative
across its components if $r>1$, the expression \eqref{eq:linear-2} equals
\begin{align}\label{eq:linear-3}
\nonumber
& 
 \left(1 + o(1) + O\left(\frac{T^{-1+O_D(\gamma)}}{\vol \fK}\right)\right)
 \dsum_{\boldsymbol{\kappa},\boldsymbol{\lambda} ,\u, \Q}
 \dsum_{~\m, \boldsymbol\delta,\boldsymbol\delta', \bd~}
 \prod_{j=1}^r  
 \frac{  H^{\kappa_j} h^{\sharp}_j(u_j)
 g_j^{\sharp}(d_j) \mu(\delta_j) \mu(\delta'_j)
 h^{\flat}_j(m_jQ_{j})}{\Delta_j}
\\
& \quad \times
 \lambda\left(\frac{\log Q_jm_j}{\log N^{\gamma}}\right)
 \chi\left(\frac{\log d_j}{\log N^{\gamma}}\right)
 \chi\left(\frac{\log \delta_j}{\log Q_j^*}\right)
 \chi\left(\frac{\log \delta'_j}{\log Q_j^*}\right),
\end{align}
again with the summation conditions from \eqref{eq:linear} and the coprimality conditions from step (2) in place.
Note that \eqref{eq:linear-3} no longer features the linear polynomials $\phi_i$, 
a fact that can only be achieved because we are working with a $W$-trick.

(4) 
Following \cite[App.\ D]{GT-linearprimes},
we essentially replace $\chi(\log m / \log Q)$ and $\lambda(\log m / \log Q))$ by
multiplicative functions in $m$, using the Fourier-type transforms 
$$
e^x \chi(x) = \int_{\RR} \theta(\xi) e^{-ix\xi}d\xi, \quad
e^x \lambda(x) = \int_{\RR} \theta'(\xi) e^{-ix\xi}d\xi,
$$
which define rapidly decaying functions $\theta, \theta': \RR \to \RR$.
Since $\chi$ and $\lambda$ have compact support, Fourier inversion and integration by parts shows that
\begin{equation}\label{eq:decay}
\theta(\xi), \theta'(\xi) \ll_E (1+|\xi|)^{-E} 
\end{equation}
for all $E>0$. 
Thus, (cf.\ \cite[paragraph before (D.6)]{GT-linearprimes}), 
setting $I = [-(\log N^{\gamma})^{1/2},(\log N^{\gamma})^{1/2}]$, one obtains
\begin{equation} \label{eq:chi-trans}
\chi\left( \frac{\log m}{\log Q}\right)
= \int_I m^{-\frac{1+i\xi}{\log Q}} \theta(\xi) d\xi 
+ O_E\left( m^{-\frac{1}{\log Q}} (\log N^{\gamma})^{-E} \right)
\end{equation}
and 
\begin{equation} \label{eq:lambda-trans}
\lambda\left( \frac{\log m}{\log Q}\right)
= \int_I m^{-\frac{1+i\xi}{\log Q}} \theta'(\xi) d\xi 
+ O_E\left( m^{-\frac{1}{\log Q}} (\log N^{\gamma})^{-E} \right).
\end{equation}
Inserting these new expressions for $\chi$ and $\lambda$ at all instances in 
\eqref{eq:linear-2} (resp.\ \eqref{eq:linear-3}) and multiplying out, we obtain a 
main term and an error term.
Any integral occurring in the main term runs over the compact interval $I$ and,
thanks to the factor $\alpha_{\boldsymbol\phi}(\Delta_1, \dots, \Delta_r)$
(resp. $\prod_j \Delta_j^{-1}$),
the summation over 
$\boldsymbol{\kappa}, 
\boldsymbol{\lambda},
\u,
\Q,\m, 
\boldsymbol\delta,
\boldsymbol\delta',
\bd$ 
is absolutely convergent.
Thus, in the main term of \eqref{eq:linear-2} (resp.\ \eqref{eq:linear-3}) we can swap sums and integrals. 
Taking into account steps (1)--(3), as well as the convergence of \eqref{eq:c_0}, 
the expression \eqref{eq:linear-2} then becomes
\begin{align}\label{eq:linear-4}
\nonumber
 & (1+o(1)) \sum_{\boldsymbol\kappa} \sum_{\boldsymbol\lambda} \dsum_{\u}
\prod_{j=1}^r
\frac{H^{\kappa_j} h^{\sharp}_j(u_j)}{u_j}\\
\nonumber
 & \times
 \int_I \dots \int_I
 \dsum_{\Q} 
 \dsum_{\m,\boldsymbol\delta,\boldsymbol\delta',\bd}
 \frac{ \alpha_{\boldsymbol{\tilde \phi}}(\tilde\Delta_1, \dots, \tilde\Delta_r)}
 {Q_1,\dots,Q_r}
 \left(\prod_{j=1}^r J_j^*~
 \theta'(\xi_{3,j})
 \theta(\xi_{4,j})
 \theta(\xi_{5,j})
 \theta(\xi_{6,j})\right)
 \bd \boldsymbol{\xi}\\
 &+ O_E\left(\frac{1}{(\log N)^E}
  \sum_{\Q, \m, \boldsymbol\delta,\boldsymbol\delta',\bd}
  \left(\prod_{i=1}^r 
  \frac{H^{\Omega(d_i)}}
  {(Q_i m_i d_i \delta_i \delta'_i)^{1/\log T^{\gamma}}}\right)
  \frac{ \alpha_{\boldsymbol{\tilde \phi}}(\tilde\Delta_1, \dots, \tilde\Delta_r)}
 {Q_1,\dots,Q_r}
 \right),
\end{align}
where
\begin{align} \label{eq:J_j-1}
 J_j^* = h^{\flat}(Q_j) Q_j^{-\frac{1+i\xi_3}{\gamma\log N}} J_j
\end{align}
and
\begin{align} \label{eq:J_j-2}
 J_j=
 g_j^{\sharp}(d_j)
 h^{\flat}_j(m_j)
 m_j^{-\frac{1+i\xi_3}{\gamma\log N}}
 d_j^{-\frac{1+i\xi_4}{\gamma\log N}}
 \mu(\delta_j) \mu(\delta'_j)
 \delta_j^{-\frac{1+i\xi_5}{\log Q_j^*}}
 {\delta'}_j^{-\frac{1+i\xi_6}{\log Q_j^*}}
\end{align}
with $Q_j^*= Q_j(N^{\gamma})^{\1_{Q_j=1}}$, as before.

Similarly, the expression \eqref{eq:linear-3} takes the same form as \eqref{eq:linear-4} but with
$\alpha_{\boldsymbol\phi}(\tilde\Delta_1, \dots, \tilde\Delta_r)$ replaced by
$(\tilde\Delta_1 \dots \tilde\Delta_r)^{-1}$
in both instances.
To proceed further, the following two lemmas are required
\begin{lemma} \label{lem:lin-1}
 Suppose $\boldsymbol{\kappa}, \boldsymbol{\lambda}$, 
 $u_i\in U(\kappa_i,\lambda_i)$ and $Q_i \geq N^{(\log \log N)^{-3}}$ for $1\leq i \leq r$
 are fixed and let $M \geq 1$ be any given integer, e.g.\ $M=\prod_{1 \leq j \leq r} Q_ju_j$.
 Let $J_j$ for $1 \leq j \leq r$ be the quantity defined in \eqref{eq:J_j-2}.
 Then, uniformly in $M$,
\begin{align*}
 \dsum_{\m, \boldsymbol\delta,\boldsymbol\delta',\bd}
 \alpha_{\boldsymbol{\tilde\phi}}(\tilde\Delta_1, \dots, \tilde\Delta_r)
 \left(\prod_{j=1}^r J_j\right)
 = (1+o(1)) \prod_{j=1}^r
 \dsum_{m_j, \delta_j,\delta'_j,d_j}
 \frac{J_j}{\tilde\Delta_j},
\end{align*}
 where $o(1)$ is uniform in all parameters and where $\sum'$ indicates that all components are co-prime to $M$.
\end{lemma}

\begin{lemma} \label{lem:lin-2}
Suppose $\boldsymbol \kappa$, $\boldsymbol \lambda$ and $\u$ with $u_i \in U(\kappa_i, \lambda_i)$, $1 \leq i \leq r$
are fixed, and suppose that $M \geq 1$ is any given integer, e.g.\  $M =\prod_{1 \leq i \leq r} u_i$.
Let $J_j^*$ for $1 \leq j \leq r$ be the quantity defined in \eqref{eq:J_j-1}.
Then we have
 \begin{align} \label{eq:lin-2}
& \int_I \dots \int_I
 \prod_{j=1}^r 
 \left|  
 \dagsum_{\substack{Q_j,m_j,\delta_j,\delta'_j,d_j  }}
 \frac{J_j^*}{Q_j \tilde\Delta_j}
 \theta'(\xi_{3,j})
 \theta(\xi_{4,j})
 \theta(\xi_{5,j})
 \theta(\xi_{6,j})
 \right|
 \bd \boldsymbol{\xi}_j \\ \nonumber
 &\ll  \Big(\frac{M}{\phi(M)}\Big)^r \prod_{j=1}^r \frac{W_j}{\phi(W_j)}
  \frac{1}{\log N}
 \prod_{\substack{p < N^{\gamma} \\ p \nmid W_jM}} 
 \left(1 +\frac{|h_j(p)|}{p} \right), 
 \end{align}
 where $Q_j \in \mQ^{(j)}$ with $Q_j>N^{\gamma/(\log \log N)^3}$,
 where $m_j,\delta_j,\delta'_j \in \<\mQ^{(j)} \>$, where $d_j \in \< \mathcal{P}_{\sharp}^{(j)}\>$,
 and where the $\sum^{\dagger}$ indicates that the summation is restricted to
 integers co-prime to $W_jM$.
\end{lemma}
The two lemmas above correspond to \cite[Lemmas 9.4 and 9.3]{bm}.
The proof of Lemma~\ref{lem:lin-1} is very similar to that of 
\cite[Lemmas 9.4]{bm} rather fairly lengthy so that we omit it here.
The proof of Lemma \ref{lem:lin-2}, in contrast, needs to work with significantly 
weaker assumptions on the Dirichlet series involved than the corresponding one 
from \cite{bm}.
We will therefore carry out this proof below.
Before we do so, let us, however, show how to deduce the equality of 
\eqref{eq:linear-2} and \eqref{eq:linear-3} and, thus, complete 
the proof of Proposition \ref{p:linear-forms}.

By applying first Lemma \ref{lem:lin-1} with $M = \prod_{1\leq j \leq r} u_jQ_j$ 
and then Lemma \ref{lem:lin-2} with $M' = 1$, we obtain 
\begin{align*}
&\sum_{\boldsymbol\kappa} \sum_{\boldsymbol\lambda} \dsum_{\u}
\prod_{i=1}^r
\frac{H^{\kappa_i} h^{\sharp}_i(u_i)}{u_i}\\
&\int_I \dots \int_I
 \dsum_{\Q}
 \dsum_{\m,\boldsymbol\delta,\boldsymbol\delta',\bd}
 \frac{J_1^* \dots J_r^*}{Q_1 \dots Q_r}
 \bigg(
 \alpha_{\boldsymbol{\tilde \phi}}(\tilde\Delta_1, \dots, \tilde\Delta_r)
 - \frac{1}{\tilde\Delta_1 \dots \tilde\Delta_r}
 \bigg)
 \prod_{j=1}^r \theta'(\xi_{3,j}) \dots \theta(\xi_{6,j})
 \bd \boldsymbol{\xi} \\
 &= o\Bigg( 
 \sum_{\boldsymbol\kappa} \sum_{\boldsymbol\lambda} \dsum_{\u}
 \Bigg( \prod_{i=1}^r \frac{H^{\kappa_i} h^{\sharp}_i(u_i)}{u_i} \Bigg)
 \int_I \dots \int_I 
 \prod_{j=1}^r \sum_{Q_j} \bigg|
 \dsum_{m_j,\delta_j,\delta'_j,d_j}
 \frac{J_j^*}{Q_j \tilde\Delta_j}
 \theta'(\xi_{3,j}) \dots \theta(\xi_{6,j})
  \bigg|
 \d\boldsymbol\xi  \Bigg) \\
 &= o\Bigg( \sum_{\boldsymbol\kappa} \sum_{\boldsymbol\lambda} \dsum_{\u}
\prod_{j=1}^r
\frac{H^{\kappa_j} h^{\sharp}_j(u_j)d(u_j)^r}{u_j}
\frac{W_j}{\phi(W_j)}\frac{1}{\log N}
 \prod_{\substack{p < N^{\gamma} \\ p\nmid W_j}} 
 \left(1 +\frac{|h_j(p)|}{p} \right)\Bigg)\\
 &= o\Bigg( 
\frac{W_j}{\phi(W_j)}\frac{1}{\log N}
 \prod_{\substack{p < N^{\gamma} \\ p\nmid W_j}} 
 \left(1 +\frac{|h_j(p)|}{p} \right)\Bigg),
\end{align*}
where the last step follows from \eqref{eq:c_0}.
This completes the the proof of Proposition \ref{p:linear-forms}, 
assuming Lemma \ref{lem:lin-2}.
\end{proof}

It remains to prove Lemma \ref{lem:lin-2}. We emphasise that, although this lemma 
corresponds to \cite[Lemma 9.3]{bm}, the proof below is significantly stronger in 
that it does not rely on any assumption on the behaviour close to $s=1$ of 
the Dirichlet series attached to $h^{\sharp}$ or to $h^{\flat}$.

\begin{proof}[Proof of Lemma \ref{lem:lin-2}]
Most of this proof only considers the $j$-th factor of the integrand.
To simplify the notation, we abbreviate $h = h_j$, $h^{\flat} = h_j^{\flat}$,
$h^{\sharp} = h_j^{\sharp}$ and $g = g_j^{\sharp} = h_j^{\sharp} * \mu$ throughout the proof.
We proceed by decomposing the sum over $Q_j, m_j,\delta_j,\delta'_j,d_j$ into an Euler product, keeping 
in mind that the contribution of higher prime powers only affects the implied 
constant. 
The latter holds since whenever $N$ sufficiently large with respect to $H$, we have
\begin{align} \label{eq:higher-prime-powers}
\prod_{p\nmid W_j} \bigg(1 + \frac{H^2}{p^2} +  \frac{H^3}{p^3} + \dots \bigg)
&\leq  \prod_{p > w(N)} \bigg(1 + \frac{H^2}{p^2}\Big(1-\frac{H}{p}\Big)^{-1}\bigg)
\leq  \prod_{p > w(N)} \bigg(1 + \frac{2 H^2}{p^2}\bigg) \\
&\ll \exp \Big( 2H^2 \sum_{p>w(N)} \frac{1}{p^2}\Big) = O_H(1),
\end{align}
where we used that $W(N)|W_j$.

Thus, recalling the definitions of $J_j^*$ and $J_j$ from \eqref{eq:J_j-1} and \eqref{eq:J_j-2}, it follows that
\begin{align} \label{eq:lin-2-proof1}
& \left| \dagsum_{\substack{Q_j, m_j,\delta_j,\delta'_j,d_j }}
 \frac{J_j^*}{Q_j \tilde\Delta_j} \right|
=
 \left|
 \dagsum_{Q_j, m_j,\delta_j,\delta'_j,d_j}
 \frac{J_j^*}{Q_j m_j d_j \lcm(\delta_j,\delta'_j)} \right| \\
\nonumber 
&\ll \Bigg|  
 \prod_{\substack{p' \not\in \mQ \\ p' \nmid W_jM}}
 \left( 1 +  \frac{g(p')}{{p'}^{1+\frac{1+i\xi_4}{\gamma\log N}}} \right)
 \Bigg| \times \\
\nonumber 
& \quad \times 
 \Bigg\{
 \Bigg|  
 \prod_{\substack{ p \in \mQ \\ p \nmid W_jM}} 
 \left( 
 1 + \frac{h^{\flat}(p)}{p^{1+\frac{1+i\xi_3}{\gamma\log N}}}
 -\frac{1}{p^{1+\frac{1+i\xi_5}{\gamma \log N}}}
 -\frac{1}{p^{1+\frac{1+i\xi_6}{\gamma \log N}}}
 +\frac{1}{p^{1+\frac{2+i\xi_5+i\xi_6}{\gamma \log N}}}
 \right) 
 \Bigg|\\
\nonumber 
&\qquad + \Bigg|  
 \sum_{\substack{Q \in \mQ \\ Q\nmid W_jM}} \frac{h^{\flat}(Q)}{Q^{1+\frac{1 + i \xi_3}{\gamma \log N}}} 
 \prod_{\substack{p \in \mQ \\ p \nmid W_jM}} 
 \left( 
 1 + \frac{h^{\flat}(p)\1_{p<Q}}{p^{1+\frac{1+i\xi_3}{\gamma\log N}}}
 -\frac{1}{p^{1+\frac{1+i\xi_5}{\log Q}}}
 -\frac{1}{p^{1+\frac{1+i\xi_6}{\log Q}}}
 +\frac{1}{p^{1+\frac{2+i\xi_5+i\xi_6}{\log Q}}}
 \right) \Bigg| \Bigg\}.
\end{align}
Next, we observe that, in the above bound, we may by invoking \eqref{eq:higher-prime-powers}
split all factors of the form $(1+\frac{a_{1,p}}{p^{1+s_1}} + \dots + \frac{a_{4,p}}{p^{1+s_4}})$
into products of individual factors of the form $(1+\frac{a_{i,p}}{p^{1+s_i}})$.

To bound above those resulting products $\prod_p (1+\frac{a_{i,p}}{p^{1+s_i}})$ all of whose coefficients 
are non-negative, we will frequently make use of the following estimate.
Let $\mathcal{A}$ be a set of primes $(a_p)_{p \in \mathcal{A}}$ a sequence of non-negative real numbers that are 
bounded above by some constant $C>1$. Then, whenever $c>0$ is a constant, we have 
\begin{equation} \label{eq:prod-pos-coeff}
 \prod_{\substack{p \in \mathcal{A}}} 
 \left(1 +\frac{a_p}{p^{1+\frac{c}{\log y}}} \right)
 \ll \prod_{\substack{p \in \mathcal{A}\\p\leq y}} 
 \left(1 +\frac{a_p}{p} \right).
\end{equation}
To see this, note that
\begin{align*}
 \prod_{\substack{p \in \mathcal{A}\\p > y}} 
 \left(1 +\frac{a_p}{p} \right)
 &\ll \exp \Big( \sum_{p>y} \frac{a_p}{p^{1+\frac{c}{\log y}}} \Big)
 \ll \exp \Big( \sum_{k>\log_2 y} \frac{2^k}{k}\frac{C}{2^{k(1+\frac{c}{\log y})}} \Big)\\
 &\ll \exp \Big( \frac{C}{\log_2 y} \int_{\log_2 y}^{\infty} e^{-cx / \log_2 y} \d x \Big)
 \ll \exp \Big( \frac{C}{c} \exp(-c) \Big) \ll 1.
\end{align*}

When bounding the contribution from those products $\prod_p (1+\frac{a_{i,p}}{p^{1+s_i}})$ that have negative coefficients, 
i.e. when $a_{i,p}=-1$ for all $p$, we need to be much more careful.
For these products we have the following bound:
\begin{align*}
 \Bigg|  
 \prod_{\substack{p \in \mQ \\ p \nmid W_j M }} 
\left( 
 1 -\frac{1}{p^{1+\frac{1+i\xi}{\gamma \log N}}}
 \right) 
 \Bigg| 
 &\ll \Bigg|\zeta^{-1}\left(1+\frac{1+i\xi}{\gamma \log N}\right)
 \prod_{\substack{p \not\in \mQ \\ p \nmid W_jM}} 
 \left( 
  1 + \frac{1}{p^{1+\frac{1+i\xi}{\gamma \log N}}}
 \right) 
 \prod_{\substack{q \text{ (prime)} \\ q | W_jM}} 
 \left( 
  1 + \frac{1}{q^{1+\frac{1+i\xi}{\gamma \log N}}}
 \right) 
 \Bigg|\\
 &\ll \left|\frac{1+i\xi}{\gamma \log N}\right|
 \prod_{\substack{ p \not\in \mQ \\ p \nmid W_jM}} 
 \left( 
  1 + \frac{1}{p^{1+\frac{1}{\gamma \log N}}}
 \right) 
 \prod_{\substack{q < N^{\gamma}  \\ q | W_jM}} 
 \left( 
  1 + \frac{1}{q}
 \right)\\
&\ll \left|\frac{1+i\xi}{\gamma \log N}\right|
 \prod_{\substack{ p \not\in \mQ \\ p<N^{\gamma},~ p \nmid W_jM}} 
 \left( 
  1 + \frac{1}{p}
 \right) 
 \prod_{\substack{q < N^{\gamma} \\ q | W_jM}} 
 \left( 
  1 + \frac{1}{q}
 \right)\\
&\ll |1+i\xi|
 \prod_{\substack{p \in \mQ\\p<N^{\gamma},~ p \nmid W_jM}} 
 \left( 
  1 - \frac{1}{p}
 \right).
\end{align*}
A similar bound holds when $\gamma \log N$ is replaced by $\log Q$. 
Combined with \eqref{eq:lin-2-proof1}, these two bounds yield:
\begin{align} \label{eq:lin-2-proof2}
& \left|
 \dagsum_{Q_j, m_j,\delta_j,\delta'_j,d_j}
 \frac{J_j^*}{Q_j m_j d_j \lcm(\delta_j,\delta'_j)} \right|\\
\nonumber 
&\ll   \left|(1+i\xi_5)(1+i\xi_6)\right|
  \prod_{\substack{ p' \not\in \mQ \\ \\ p' \nmid W_jM}}
 \left( 1 +  \frac{g(p')}{{p'}^{1+\frac{1}{\gamma\log N}}} \right) \times \\
\nonumber 
& \qquad \times \Bigg\{
 \prod_{\substack{ p \in \mQ \\ p \nmid W_jM}} 
 \left( 
 1 + \frac{h^{\flat}(p)}{p^{1+\frac{1}{\gamma\log N}}}\right)
 \left(1+\frac{1}{p^{1+\frac{2}{\gamma \log N}}} \right)
 \prod_{\substack{ p \in \mQ \\p<N^{\gamma},~ p \nmid W_jM}} 
 \left(1-\frac{1}{p} \right)^{2}
 \\
\nonumber 
&\qquad \qquad+ 
 \sum_{\substack{Q \in \mQ \\ Q \nmid W_jM}} \frac{h^{\flat}(Q)}{Q^{1+\frac{1}{\gamma \log N}}} 
 \prod_{\substack{ p \in \mQ \\ p < Q,~ p \nmid W_jM }} 
 \left(1 +\frac{h^{\flat}(p)}{p^{1+\frac{1}{\gamma \log N}}} \right)
 \left(1 - \frac{1}{p} \right)^2 
 \prod_{\substack{ p'' \in \mQ \\ p'' \nmid W_jM}} 
 \left(1 +\frac{1}{{p''}^{1+\frac{2}{\log Q}}} \right) \Bigg\}.
\end{align} 
Using \eqref{eq:prod-pos-coeff}, the above is bounded by
\begin{align} \label{eq:lin-2-proof3}
&\ll   \left|(1+i\xi_5)(1+i\xi_6)\right|
 \quad \prod_{\substack{p' \not\in \mQ \\ p'<N^{\gamma},~ p' \nmid W_jM}}
 \left( 1 + \frac{g(p')}{p'} \right) \\
 \nonumber
&\quad \Bigg\{
 \prod_{\substack{ p \in \mQ \\ p <N^{\gamma} \\ p \nmid W_jM}} 
 \left( 
 1 + \frac{h^{\flat}(p)}{p}\right)
 \left(1-\frac{1}{p} \right)
+ 
 \sum_{\substack{Q \in \mQ \\ Q\nmid W_jM}} \frac{h^{\flat}(Q)}{Q^{1+\frac{1}{\gamma \log N}}} 
 \prod_{\substack{p \in \mQ \\ p<Q \\ p \nmid W_jM}} 
 \left(1 +\frac{h^{\flat}(p)}{p^{1+\frac{1}{\gamma \log N}}} \right)
 \left( 1 - \frac{1}{p} \right) \Bigg\}.
\end{align}
Note that
\begin{align*}
&\sum_{\substack{Q \in \mQ \\ Q\nmid W_jM}} \frac{h^{\flat}(Q)}{Q^{1+\frac{1}{\gamma \log N}}} 
 \prod_{\substack{p \in \mQ \\ p<Q,~ p \nmid W_jM}} 
 \left(1 +\frac{h^{\flat}(p)}{p^{1+\frac{1}{\gamma \log N}}} \right)
 \left( 1 - \frac{1}{p} \right) \\
& \ll 
 \sum_{\substack{Q \in \mQ \\ Q\nmid W_jM}} 
 \frac{h^{\flat}(Q)}{Q^{1+\frac{1}{\gamma \log N}}} \left( 1 - \frac{1}{Q^{1+\frac{1}{\gamma \log N}}} \right)
 \prod_{\substack{p \in \mQ \\ p<Q,~ p \nmid W_jM}} 
 \left(1 +\frac{h^{\flat}(p)}{p^{1+\frac{1}{\gamma \log N}}} \right)
 \left( 1 - \frac{1}{p^{1+\frac{1}{\gamma \log N}}} \right) \\
&= \prod_{\substack{p \in \mQ \\ p \nmid W_jM}} 
 \left(1 +\frac{h^{\flat}(p)}{p^{1+\frac{1}{\gamma \log N}}} \right)
 \left( 1 - \frac{1}{p^{1+\frac{1}{\gamma \log N}}} \right) \\
&= \prod_{\substack{p \in \mQ \\ p \nmid W_jM}} 
 \left(1 +\frac{h^{\flat}(p)}{p^{1+\frac{1}{\gamma \log N}}} \right)
 \prod_{\substack{p' \in \mQ \\ p' \nmid W_jM}} 
 \left( 1 - \frac{1}{{p'}^{1+\frac{1}{\gamma \log N}}} \right) \\
&\ll \prod_{\substack{p \in \mQ \\ p < N^{\gamma}, \, p \nmid W_jM}} 
 \left(1 +\frac{h^{\flat}(p)}{p} \right)
 \left( 1 - \frac{1}{p} \right).
\end{align*}
Thus, \eqref{eq:lin-2-proof3} is bounded by
\begin{align*}
&\ll  \left|(1+i\xi_5)(1+i\xi_6)\right|
  \prod_{\substack{p' \not\in \mQ \\ p'<N^{\gamma},~ p' \nmid W_jM}}
 \left( 1 +  \frac{g(p')}{p'} \right) 
 \prod_{\substack{p<N^{\gamma}\\ p \in \mQ,\, p \nmid W_jM }} 
 \left( 1 + \frac{h^{\flat}(p)}{p}\right)
 \left(1-\frac{1}{p} \right)\\  
&\ll   \left|(1+i\xi_5)(1+i\xi_6)\right|
  \prod_{\substack{p<N^{\gamma}\\ p \nmid W_jM}}
 \left( 1 + \frac{|h(p)|}{p} \right) \left( 1-\frac{1}{p} \right)\\  
 & \ll_{\gamma} |1+i\xi_5||1+i\xi_6|~
 \frac{1}{\log N} \frac{W_jM}{ \phi(W_jM)}
 \prod_{\substack{ p < N \\ p \nmid W_jM}} 
 \left(1 +\frac{|h_j(p)|}{p} \right).
 \end{align*}
 Inserting this bound for \eqref{eq:lin-2-proof1} into \eqref{eq:lin-2} and taking the decay properties \eqref{eq:decay}
 of the function $\theta$ and $\theta'$ into account,  yields the bound claimed in Lemma \ref{lem:lin-2}.
\end{proof}

\subsection{The average order of $\nu^{\sharp} \nu^{\flat}$} \label{s:average-order}
As a consequence of the proof of Proposition \ref{p:linear-forms}, 
we shall now deduce Proposition \ref{p:average-order}, that is, that $\nu^{\sharp} \nu^{\flat}$
has the `correct' average order.
To start with, note that the first bound of \eqref{eq:correct-average-order} is immediate 
since $\nu^{\sharp} \nu^{\flat}$ is, 
outside of the sparse set $\mathcal{S}$, a majorant for $|h|$.
The second bound follows from the upper bound \eqref{eq:linear-4} on 
$$\frac{1}{T^s \vol \fK}
\sum_{\n \in \ZZ^s \cap T\fK} \prod_{i=1}^s \nu_{h_i}(\phi_i(\n)),
$$
where we are only interested in the case $s=1$.
To see this, we apply first Lemma \ref{lem:lin-1} and then Lemma \ref{lem:lin-2} 
to the integrand, and finally recall that the outer sum \eqref{eq:c_0} 
converges.

\section{Proof of Proposition \ref{p:correlation}} \label{s:correlation-condition}
To complete the proof of Theorem \ref{t:pseudorandom}, it remains to prove Proposition \ref{p:correlation}.

Recall that $\tau = (N, W_1, \dots, W_r, A_1, \dots, A_r)$ and that
$d \in \{1, \dots, D\}$ and $(i_1, \dots, i_d) \in \{1,\dots,r\}^d$.
 A slight adaptation of \cite[Lemma 9.9]{GT-longprimeAPs} yields the following.
\begin{lemma}  \label{l:GTsigma}
Let $\Delta:\ZZ \to \ZZ$ denote the polynomial
$$\Delta(m)= \prod_{1 \leq j < j' \leq d} 
\bigg(Wm + \frac{W}{W_{i_j}}A_{i_j} - \frac{W}{W_{i_{j'}}}A_{i_{j'}}\bigg)
$$
where $W = \lcm(W_1, \dots, W_r)$.
Suppose that the family of functions 
$$\sigma_{\tau,d}: \{-\lfloor N/W \rfloor,\dots,\lfloor N/W \rfloor\} \to \RR_{\geq 0},$$  
satisfies the two conditions $\sigma_{\tau,d}(0) \ll_q (N/W)^{1/q}$ for all $q\in \NN$ and $\tau$
as in Theorem \ref{t:pseudorandom},
and
$$
\sigma_{\tau,d}(m) 
= \exp \left( \sum_{p>w(N),~ p|\Delta(m)} O_d(p^{-1/2}) \right)
$$
for $m\not=0$.
Then $\EE_{-T/W\leq m \leq T/W} \sigma_{\tau,d}^q(m) \ll_q 1$ for all $q \in \NN$, all $\tau$ and 
$T \in [N(\log N)^{-B},N]$.
\end{lemma}

Note that whenever the collection $(a_j)_{1 \leq j \leq d} \in \{1, \dots, \lfloor T/W \rfloor\}^d$ contains two identical elements, 
then $\sigma_{\tau,d}(0)$ appears in the bound \eqref{eq:correlation-est} we seek to establish.
Following \cite{GT-longprimeAPs, GT-linearprimes} closely, 
we seek to handle this case using the fact that the above lemma allows us to choose
$\sigma_{\tau,d}(0)$ to be rather large.
To start with, note that the generalised divisor function $d_k$ satisfies 
condition (i) of Definition \ref{d:M} with $H=k$.
Further, given $h \in \mathcal{F}^*$, condition (i) from Definition \ref{d:M} implies the bound
$$|h(n)| \leq H^{(\log n)/\log w(N)} = n^{(\log H)/\log w(N)},$$
valid for all $n$ coprime to $W(N)$.
Thus, recalling \eqref{eq:d_7-bound} and \eqref{eq:p-linear-forms-2nd-moment}, we have
\begin{align*}
 \nu_h^{\sharp}(n) \nu_h^{\flat}(n)
 &\leq \sum_{\kappa, \lambda, u, Q} \sum_{m,\delta,\delta',d} \1_{\lcm(\delta m Q, \delta' m Q, u,d)|n}
 H^{\kappa} h^{\sharp}(u) g^{\sharp}(d) \\
 & \leq H^{(\log \log N)^3} d_7(n) h^{\sharp}(n) 
 \leq H^{(\log \log N)^3} N^{(\log (7H))/\log w(N)}
\end{align*}
for all sufficiently large $N \in \NN$ and all $n \leq N$ coprime to $W(N)$.
Hence, it follows that
\begin{align*}
\frac{W}{T}\sum_{1 \leq m \leq T/W} \prod_{j=1}^d 
  \nu_{h_{i_j}}(W_{i_j}(m+a_j) + A_{i_j})
\ll \exp(O_{H}(1) D (\log N)/ \log w(N))
\end{align*}
provided that $1 \leq d \leq D$. 
Setting $$\sigma_{\tau,d}(0) = \exp (C_H D \log N / \log \log N)$$ for some absolute constant $C_H$ 
thus ensures that \eqref{eq:correlation-est} holds whenever $1 \leq d \leq D$, and that the condition 
on $\sigma_{\tau,d}(0)$ in Lemma \ref{l:GTsigma} is satisfied.

In the remaining case where the $a_j$ are pairwise distinct, the system
of linear forms is less degenerate and we may employ the same techniques
used to prove Proposition~\ref{p:linear-forms}.
The key observation is that whenever a prime $p$ divides two distinct
polynomials $W_{i_j}(m+a_j) + A_{i_j}$ and $W_{i_{j'}}(m+a_{j'}) + A_{i_{j'}}$ at the integer $m$,
then it divides $\Delta(a_j-a_{j'})$.
Moreover, we claim that if $\Delta := \prod_{j \not= j'} \Delta(a_j-a_{j'})$ and the $a_j$ are pairwise distinct, then
\begin{align} \label{eq:correlation-proof}
 &\frac{W}{T}
 \sum_{1 \leq m \leq T/W} \prod_{j=1}^d 
 \nu_{h_{i_j}}^{(N)}(W_{i_j}(m+a_j) + A_{i_j}) \\
 \nonumber
 &\ll 
 \Bigg(\prod_{j=1}^d E_{h_{i_j}}(N;W_{i_j}) \Bigg)
 \prod_{\substack{p| \Delta \\ p > w(N)}} \Bigg(
 \sum_{c_1, \dots, c_d} \alpha_{\boldsymbol{\phi}}(p^{c_1},\dots,p^{c_d})
 (2^4H)^{c_1 + \dots + c_d}
 \Bigg),
\end{align}
where $\boldsymbol \phi = (\phi_{i_1}, \dots, \phi_{i_d}): \ZZ \to \ZZ^d$ 
is the system of linear forms $\phi_j(m) = W_j (m+a_j) + A_j$
for $1 \leq j \leq r$, and where $E_{h_{i_j}}(N;W_{i_j})$ was defined in \eqref{eq:def-E}.
We shall take \eqref{eq:correlation-proof} on trust for now and defer its proof to the very end of this section.
Note that, instead of \eqref{eq:ev-alpha-1}, the given system $\boldsymbol \phi$ (of infinite complexity) satisfies
\begin{align}\label{eq:ev-alpha-inf-complexity}
\alpha_{\boldsymbol{\phi}}(p^{c_1},\dots,p^{c_d})
\begin{cases} 
\leq  p^{- \max_i c_i} & \text{for all } p \text{ and } \mathbf{c}; \cr
= p^{- \max_i c_i} & \text{if } n(\mathbf c)=1 \text{ and }(c_i>0 \Rightarrow p \nmid W_i);\cr
= 0 & \text{if } \exists i: c_i >0 \text{ and } p|W_i; \cr
= 0 & \text{if }p\nmid \Delta \text{ and } n(\mathbf c)>1. 
\end{cases}
\end{align}

Since there are at most $(j+1)^{d}$ integer tuples $(c_1, \dots, c_d)$ with $\max_i c_i =j$
and since 
$(j+1)^{d} (2^4H)^{dj}  \leq p^{j/2}/2$ for all $j \geq 1$ and all $p > w(N)$ as soon as $N$ is sufficiently large,
we have
\begin{align} \label{eq:alpha-c_1--c_d}
 \sum_{c_1, \dots, c_d} \alpha_{\boldsymbol{\phi}}(p^{c_1},\dots,p^{c_d}) (2^4H)^{c_1 + \dots + c_d}
 \leq 1 + \frac{1}{2} p^{-1/2} \sum_{j \geq 0} p^{-j/2}
 \leq 1 + p^{-1/2}.
\end{align}
Thus, it follows from \eqref{eq:correlation-proof} that
\begin{align*}
 &\frac{W}{T}
 \sum_{1 \leq m \leq T/W} \prod_{j=1}^d 
 \frac{\nu_{h_{i_j}}^{(N)}(W_{i_j}(m+a_j) + A_{i_j}) }{E_{h_{i_j}}(N;W_{i_j})}
 \ll \prod_{\substack{p| \Delta \\ p > w(N)}} (1 + p^{-1/2}) 
 \ll \exp \bigg( \sum_{\substack{p| \Delta \\ p > w(N)}} O( p^{-1/2}) \bigg).
\end{align*}
Since $\exp(x_1+ \dots +x_d) \leq \exp(d \max_i x_i)$ the above is bounded by
$$\ll  \sum_{j \not= j'} \exp \bigg( \sum_{\substack{p| \Delta(a_j -a_{j'}) \\ p > w(N)}} O_d( p^{-1/2}) \bigg).$$
The proposition now follows from Lemma \ref{l:GTsigma}.

It remains to prove \eqref{eq:correlation-proof}.
Using the lattice point counting formula \eqref{eq:lattice-point-count}
and writing $L'_j = L_j u_j$ and $L_j = \lcm(\delta_j,\delta'_j) m_jQ_j d_j$ for $1 \leq j \leq d$, we deduce that
\begin{align*}
 &\frac{W}{T}
 \sum_{1 \leq m \leq T/W} \prod_{j=1}^d 
 \nu_{h_{i_j}}^{(N)}(W_{i_j}(m+a_j) + A_{i_j}) \\
 &= (1+o(1)) 
 \sum_{\boldsymbol\kappa,\boldsymbol\lambda,\u, \Q}
 \sum_{~\m, \boldsymbol\delta,\boldsymbol\delta', \bd~}
 \alpha_{\boldsymbol{\phi}}(L'_1, \dots, L'_d)
  \prod_{j=1}^d 
  H^{\kappa_j} h^{\sharp}(u_j)
 g_j^{\sharp}(d_j) \mu(\delta_j) \mu(\delta'_j)
 h^{\flat}(m_jQ_{j}) \\
& \qquad \qquad \qquad \qquad \times 
 \lambda\left(\frac{\log Q_jm_j}{\log N}\right)
 \chi\left(\frac{\log d_j}{\log N^{\gamma}}\right)
 \chi\left(\frac{\log \delta_j}{\log Q_j^*}\right)
 \chi\left(\frac{\log \delta'_j}{\log Q_j^*}\right).
\end{align*} 
Invoking the transformations \eqref{eq:chi-trans} and \eqref{eq:lambda-trans}, the above equals
 \begin{align} \label{eq:corr-proof-2}
 &\sum_{\boldsymbol\kappa,\boldsymbol\lambda,\u}
 H^{\kappa_j} h^{\sharp}(u_j)
 \int_I \dots \int_I
 \sum_{\Q,\m,\boldsymbol\delta,\boldsymbol\delta',\bd}
 \alpha_{\boldsymbol\phi}(L'_1,\dots,L'_d)
 \left(\prod_{j=1}^d  J_j^*~
 \theta'(\xi_{3,j})
 \theta(\xi_{4,j})
 \theta(\xi_{5,j})
 \theta(\xi_{6,j})\right)
 \bd \boldsymbol{\xi} + \mathcal{E},  
\end{align}
where
$$ J_j^* = h^{\flat}(Q_j) Q_j^{-\frac{1+i\xi_3}{\gamma\log N}} 
 g_j^{\sharp}(d_j)
 h^{\flat}(m_j)
 m_j^{-\frac{1+i\xi_3}{\gamma\log N}}
 d_j^{-\frac{1+i\xi_4}{\gamma\log N}}
 \mu(\delta_j) \mu(\delta'_j)
 \delta_j^{-\frac{1+i\xi_5}{\log Q_j^*}}
 {\delta'}_j^{-\frac{1+i\xi_6}{\log Q_j^*}}
$$
with $Q_j^*= Q_j(N^{\gamma})^{\1_{Q_j=1}}$, as before,
and where the error term $\mathcal{E}$ is such that for all $E \geq 1$: 
\begin{align} \label{eq:corr-error}
 \mathcal{E} \ll_E \frac{1}{(\log N)^E}
  \sum_{\boldsymbol\kappa,\boldsymbol\lambda,\u~} 
  \sum_{\substack{\Q, \m, \boldsymbol\delta,\boldsymbol\delta',\bd }}
  \alpha_{\boldsymbol\phi}(L'_1, \dots, L'_d)
  \prod_{i=1}^d 
  \frac{H^{\kappa_i} H^{\Omega(d_i) +\Omega(u_i)}}
  {(Q_i m_i d_i \delta_i \delta'_i)^{1/\log N^{\gamma}}}.
\end{align}
To simplify the task of bounding the main term in the expression above, we seek to use the fact that $J^*_j$ 
is multiplicative in order to factorise the integrand in the main term for any fixed vector $\u$ into one factor 
running over primes dividing $\Delta$, one factor running over primes dividing $u_1 \dots, u_d$
and one factor taking into account the contribution from integers coprime to $M:=\Delta u_1 \dots u_d$.
The final factor will then indeed simplify since 
$\alpha_{\phi}(L'_1, \dots, L'_d) = \alpha_{\phi}(u_1, \dots, u_d)\prod_j 1/L_j$ 
or $\alpha_{\phi}(L'_1, \dots, L'_d)=0$ whenever $\gcd(L_1\dots L_d, M)=1$. 
We will consider the error term $\mathcal{E}$ in a second step, using the same approach.

For the purpose of the described factorisation, define the multiplicative function
$$\eta(\Q, \m,\boldsymbol\delta,\boldsymbol\delta',\bd) 
= \alpha_{\boldsymbol\phi}(L_1, \dots, L_d) J^*_1 \dots J^*_d.$$
In view of \eqref{eq:ev-alpha-inf-complexity} and since $|J^*_j| \leq H^{\Omega(d_j)}$, we have
\begin{align} \label{eq:eta-three-factors}
\nonumber
&\Bigg|
\sum_{\substack{\Q:\,Q_j \in \mathcal{P}_{\flat}^{(j)} \cup \{1\}, \\ 
      \gcd (Q_j,M)=1,\\
      (Q_j  = 1) \vee 
       (Q_j > \\ N^{\gamma/(\log \log N)^3 })}}
\sum_{\substack{\m: \, p|m_j \Rightarrow \\
      p \nmid M, \, p\in \mathcal{P}_{\flat}^{(j)}, \\ (Q_j>1 \Rightarrow p<Q_j)}}
\sum_{\substack{\boldsymbol\delta: \, p|\delta_j \Rightarrow \\
      p\nmid M, \, p\in \mathcal{P}_{\flat}^{(j)}}}
\sum_{\substack{\boldsymbol\delta': \, p|\delta'_j \Rightarrow \\
      p \nmid M,\, p\in \mathcal{P}_{\flat}^{(j)} }}
\sum_{\substack{\bd: \, p|d_j \Rightarrow \\
      p \nmid M,\, p\in \mathcal{P}_{\sharp}^{(j)} }}
\alpha_{\boldsymbol\phi}(L_1u_1, \dots, L_du_d) J^*_1 \dots J^*_d 
\Bigg| \\
\nonumber
&\leq 
\Bigg(\prod_{\substack{p|\Delta \\ p \nmid u_1 \dots u_d~}} 
  \starsum_{\substack{\bc_Q, \bc_m, \\ \bc_{\delta},\bc_{\delta'},\bc_{d} \in \NN_0^d}} 
  |\eta(p^{\bc_Q},p^{\bc_m},p^{\bc_{\delta}},p^{\bc_{\delta'}},p^{\bc_{d}})|\Bigg) \\
\nonumber
&\quad \times \Bigg( \prod_{\substack{p|u_1 \dots u_d~}} 
  \starsum_{\substack{\bc_Q, \bc_m, \\ \bc_{\delta},\bc_{\delta'},\bc_{d} \in \NN_0^d}} 
  \frac{H^{c_{d,1} + \dots + c_{d,d}}}
       {p^{\max (1,\max_i (c_{Q,i} + c_{m,i} + c_{\delta,i} + c_{\delta',i} + c_{d,i}))}}\Bigg) 
\\
& \quad \times \Bigg|\prod_{j=1}^d 
\sum_{\substack{Q_j \in \mathcal{P}_{\flat}^{(j)} \cup \{1\}, \\ 
Q_j  \not= 1 \Rightarrow
(Q_j \nmid M \wedge \\ Q_j > N^{\gamma/(\log \log N)^3 })}}
\sum_{\substack{m_j: \, p|m_j \Rightarrow \\
p \nmid M, \, p\in \mathcal{P}_{\flat}^{(j)}, \\ (Q_j>1 \Rightarrow p<Q_j)}}
\sum_{\substack{\delta_j: \, p|\delta_j \Rightarrow \\
p\nmid M, \, p\in \mathcal{P}_{\flat}^{(j)}}}
\sum_{\substack{\delta_j': \, p|\delta'_j \Rightarrow \\
p \nmid M,\, p\in \mathcal{P}_{\flat}^{(j)} }}
\sum_{\substack{d_j: \, p|d_j \Rightarrow \\
p \nmid M,\, p\in \mathcal{P}_{\sharp}^{(j)} }}
\frac{J^*_j}{L_j}\Bigg|, 
\end{align}
where $\starsum$ indicates that $\bc_Q$, $\bc_m$, $\bc_{\delta}$, $\bc_{\delta'}$, $\bc_{d}$ respect the correct 
summation conditions, e.g.\ $c_{m,j} = 0$ unless the following three conditions hold:
$p\in \mathcal{P}_{\flat}^{(j)}$, $p \nmid M$ and, provided $Q_j>1$, then $p<Q_j$. If these conditions hold then
$c_{m,j}$ runs though $\NN_0$; similar restrictions hold for the components of the remaining tuples.
For the later task of bounding $\mathcal{E}$, observe that \eqref{eq:eta-three-factors} remains to hold 
if we replace all instance of $J^*_j$, $1 \leq j \leq d$ by $|J^*_j|$, 
or even by the upper bound $H^{\Omega(d_j)}(Q_j m_j d_j \delta_j \delta'_j)^{-1/\log N^{\gamma}}$ for $|J^*_j|$.

Bounding the second factor in the above bound a similar fashion as in \eqref{eq:alpha-c_1--c_d} and 
dropping the summation restrictions in the first factor,
\eqref{eq:eta-three-factors} is seen to be bounded by:
\begin{align*}
&\ll 
\Bigg(\prod_{\substack{p|\Delta \\ p >w(N)}} \sum_{\substack{\bc_1, \dots, \bc_5 \\ \in \NN_0^d}} 
|\eta(p^{\bc_1},p^{\bc_2},p^{\bc_3},p^{\bc_4},p^{\bc_5})|\Bigg)
\frac{H^{\Omega(u_1 \dots  u_d)}}{\lcm(u_1, \dots,u_d)}
\Bigg|
\prod_{j=1}^d 
\sum_{\substack{Q_j,m_j,\delta_j,\delta_j',d_j \\ \text{coprime to } M}}
\frac{J^*_j}{L_j} \Bigg|.
\end{align*}
Inserting this bound into \eqref{eq:corr-proof-2} and applying first Lemma \ref{lem:lin-2} with 
$M= \Delta u_1 \dots u_d$ to bound the integral, 
and then \eqref{eq:c_0-lcm} to bound outer the summation over $\boldsymbol\kappa$, $\boldsymbol\lambda$ and $\u$, 
we obtain:
\begin{align*}
&\frac{W}{N}
 \sum_{m \in I} \prod_{j=1}^d 
 \nu_{h_{i_j}}^{(N)}(W_{i_j}(m+a_j) + A_{i_j}) 
 \ll \mathcal{E} + \\
&+
 \Bigg(\prod_{\substack{p|\Delta \\ p >w(N)}} \sum_{\substack{\bc_1, \dots, \bc_5 \\ \in \NN_0^d}} 
 |\eta(p^{\bc_1},p^{\bc_2},p^{\bc_3},p^{\bc_4},p^{\bc_5})|\Bigg)
 \sum_{\boldsymbol\kappa,\boldsymbol\lambda,\u}
 \frac{H^{\kappa_j} H^{2\Omega(u_1 \dots u_d)}}{\lcm(u_1, \dots, u_d)}
 \prod_{j=1}^d \frac{\Delta}{\phi(\Delta)} 2^{\omega(u_1 \dots u_d)} E_{h_{i_j}}(N;W_{i_j})\\
&\ll \mathcal{E} + 
 \Bigg(\prod_{\substack{p|\Delta \\ p >w(N)}} \sum_{\substack{\bc_1, \dots, \bc_5 \\ \in \NN_0^d}} 
 \frac{|\eta(p^{\bc_1},p^{\bc_2},p^{\bc_3},p^{\bc_4},p^{\bc_5})|}{(1 - p^{-1})^d}\Bigg) 
 \prod_{j=1}^d E_{h_{i_j}}(N;W_{i_j}).
\end{align*}
Finally,
\begin{align} \label{eq:eta-bound}
\nonumber
&\sum_{\substack{\bc_Q, \bc_m, \\ \bc_{\delta},\bc_{\delta'},\bc_d  \in \NN_0^d}} 
 |\eta(p^{\bc_Q}, p^{\bc_m}, p^{\bc_{\delta}},p^{\bc_{\delta'}},p^{\bc_d})|\\
\nonumber 
&\leq \sum_{\substack{\bc_Q, \bc_m, \\ \bc_{\delta},\bc_{\delta'},\bc_d  \in \NN_0^d}} 
 H^{c_{d,1} + \dots + c_{d,d}}~
 \alpha_{\boldsymbol\phi}(p^{c_{Q,1} + c_{m,1} + \max(c_{\delta,1},c_{\delta',1}) + c_{d,1}},
 \dots, p^{c_{Q,d} + c_{m,d} + \max(c_{\delta,d},c_{\delta',d}) + c_{d,d}})\\
\nonumber 
&\leq \sum_{\substack{c_1, \dots, c_d  \in \NN_0}} H^{c_1 + \dots + c_d} \alpha_{\boldsymbol\phi}(p^{c_1}, \dots, p^{c_d}) 
\sum_{ \substack{ \bc_1, \dots, \bc_5 \in \NN_0^d~}} 
\prod_{i=1}^d \1_{c_{Q,i} + c_{m,i} + \max(c_{\delta,i},c_{\delta',i}) + c_{d,i} = c_i } \\
\nonumber
&\leq \sum_{\substack{c_1, \dots, c_d  \in \NN_0}} H^{c_1 + \dots + c_d} \alpha_{\boldsymbol\phi}(p^{c_1}, \dots, p^{c_d}) 
 \prod_{i=1}^d (c_i + 1)^4 \\
&\leq \sum_{\substack{c_1, \dots, c_d  \in \NN_0}} 
 (2^{4}H)^{(c_1 + \dots + c_d)} 
 \alpha_{\boldsymbol\phi}(p^{c_1}, \dots, p^{c_d}), 
\end{align}
which shows that the bound for the main term agrees with the bound claimed in \eqref{eq:correlation-proof}.

Concerning the task of bounding the error term \eqref{eq:corr-error}, 
we begin by applying the factorisation \eqref{eq:eta-three-factors} with $J^*_j$ replaced by 
$H^{\Omega(d_j)}(Q_j m_j d_j \delta_j \delta'_j)^{-1/\log N^{\gamma}}$
to the summation in $\Q, \m, \boldsymbol\delta,\boldsymbol\delta',\bd$ and use
\eqref{eq:eta-bound} to bound the first factor and \eqref{eq:c_0-lcm} to bound the
summation in $\boldsymbol \kappa, \boldsymbol \lambda, \u$ with the extra weight coming from the 
second factor. This yields:
\begin{align*}
 \mathcal{E} \ll_E 
 \frac{1}{(\log N)^{E}}& \Bigg(
 \prod_{\substack{p|\Delta \\ p >w(N)}} \sum_{\substack{c_1, \dots, c_d \\ \in \NN_0}} 
 (2^{4}H)^{(c_1 + \dots + c_d)} 
 \alpha_{\boldsymbol\phi}(p^{c_1}, \dots, p^{c_d}) \Bigg) \\
 & \times \prod_{j=1}^d 
 \sum_{\substack{Q_j, m_j, \delta_j, \delta'_j, d_j \\ \text{coprime to }W(N)}}
 \frac{H^{\Omega(d_j)}}{(Q_j m_j d_j \delta_j \delta'_j)^{1+1/\log N^{\gamma}}}.
\end{align*}
Invoking \eqref{eq:prod-pos-coeff}, we have
\begin{align*}
 \sum_{\substack{Q_j, m_j, \delta_j, \delta'_j, d_j \\ \text{coprime to }W(N)}}
 \frac{H^{\Omega(d_j)}}{(Q_j m_j d_j \delta_j \delta'_j)^{1+1/\log N^{\gamma}}}
 \ll (\log N^{\gamma})^{O_H(1)},
\end{align*}
and hence 
$$\mathcal{E} \ll_{E,H} (\log N)^{-E} \prod_{\substack{p|\Delta \\ p >w(N)}}
 \sum_{\substack{c_1, \dots, c_d  \in \NN_0}} 
 (2^{4}H)^{(c_1 + \dots + c_d)} 
 \alpha_{\boldsymbol\phi}(p^{c_1}, \dots, p^{c_d}) . $$
This completes the proof of \eqref{eq:correlation-proof} and the proof of Proposition \ref{p:correlation}.

\section{Proofs of Proposition \ref{p:character-set}, Theorem \ref{t:main'}, and Corollaries 
\ref{cor:chi_0} and \ref{c:pret}} \label{s:corollary-proofs}
\begin{proof}[Proof of Proposition \ref{p:character-set}]
The first part, that is \eqref{eq:GS-corollary}, holds by \cite[Corollary 4.2]{lmm},  
which is a corollary to Granville, Harper and Soundararajan \cite[Theorem 1.8]{GHS}.

Since $\# \mathcal{E}_N(q) \ll_{\alpha_h, H} 1$ is finite, it follows from \eqref{eq:GS-corollary} and 
the assumption \eqref{eq:main'-assumption} that
$$
 S_{h^*}(T,q,A) = S_{h^*}(N,q,A)
 + o(E_{h}(N;q))
$$
uniformly for all $T \in (N(\log N)^{-C},N]$. 
Since in the current setting the value of $t$ that defines $h^*$ is fixed and independent of the cut-off $N$, 
all conditions from Definition \ref{d:F} (iv') hold and, hence, $h\in \mathcal{F}^*$.
\end{proof}

\begin{proof}[Proof of Theorem \ref{t:main'}]
Since Proposition \ref{p:character-set} implies that $h_1, \dots, h_r \in \mathcal{F}^*$, Theorem \ref{t:main} 
applies and yields an asymptotic formula of the form \eqref{eq:main} for the correlation 
$$
\frac{1}{\vol N\fK}
\sum_{\n \in \ZZ^s \cap N\fK} 
\prod_{i=1}^r 
h_i(\vphi_i(\n)), 
$$
which involves certain constants $B_1, B_2 >0$ and an integer-valued function $\widetilde{W}$ with the property that 
$\widetilde{W}(N) \leq (\log N)^{B_1}$ and $W(N)|\widetilde{W}(N)$ for all $N>1$.
Our aim is to show that under the assumptions of Theorem \ref{t:main'}, the main term of \eqref{eq:main}, 
that is: 
\begin{align} \label{eq:t-main-term}
\sum_{\substack{w_1, \dots, w_r 
 \\ p|w_i \Rightarrow p| \widetilde W
 \\ w_i \leq (\log N)^{B_2}}}
\sum_{\substack{A_1,\dots,A_r \\ \in (\ZZ/\widetilde W \ZZ)^*}}
\Bigg(\prod_{j=1}^r 
 h_j^*(w_j) S_{h_j^*}\Big(N;\widetilde W,A_j\Big)\Bigg)~
 \beta_{\bphi}(w_1A_1, \dots, w_rA_r),
\end{align}
can be factorised into a product over primes.
To start with, we consider the expression $S_{h^*}(N;\widetilde W,A)$ for
$h \in \{h_1, \dots, h_r\}$ and for any reduced residue $A \Mod{\widetilde W(N)}$.
Omitting the index $N$, let $\mathcal{E}$, $\mathcal{E}^+$ and $\mathcal{E}^*$ denote the respective sets 
of characters from the statement, and, given any $\chi \in \mathcal{E}$ or $\chi \in \mathcal{E}^+$, 
let $\chi^*$ denote the corresponding induced character in $\mathcal{E}^*$, if it exists.
Then, by \eqref{eq:GS-corollary} for $q=\widetilde{W}(N)$, \eqref{eq:tenen-assumpt-2} and
\eqref{eq:tenen-assumpt-1}, we obtain the expansion
\begin{align} \label{eq:S_h(T,W,A)-expansion}
 &S_{h^*}(y,\widetilde{W}(N),A) \\
 \nonumber
 &=
 \frac{\widetilde{W}}{\phi(\widetilde{W})}
 \sum_{\chi \in \mathcal{E}^+} \overline{\chi^*}(A) 
 S_{h^* \chi}(y)
 + o(E_{h}(N;\widetilde{W}(N))) 
 \\
\nonumber 
&= \frac{\widetilde{W}}{\phi(\widetilde{W})} \sum_{\chi \in \mathcal{E}^+}
 S_{|h|}(y) \overline{\chi}(A) 
 \prod_{p} \left(
 \frac{\sum_{p^k \leq y} h^*(p^k)\chi^*(p^k)p^{-k}}{\sum_{p^k \leq y} |h(p^k)|p^{-k}}\right)
 +  o(E_{h}(N;\widetilde{W}(N))) ,
\end{align}
valid for all $y \in (N^{1/2},N]$.
This gives a satisfactory factorisation for $S_{h^*}(N;\widetilde W(N),A)$.
The factor $\overline{\chi (A)}$ may be further decomposed as follows.
If $\vphi(\v)\equiv A w \Mod{p^{v_p(w\widetilde{W})}}$ for all 
$p | \widetilde W(N)$, 
and if $\tilde{\chi}$ denotes the completely multiplicative functions that arises from $\chi$ 
by setting $\tilde{\chi}(p)= \chi(p)$ if $p$ does not divide the conductor and 
$\tilde{\chi}(p)= 1$ otherwise,
then
\begin{align*}
  \overline{\chi}(A) 
&= \overline{\tilde{\chi}\left(\vphi(\v)\right)}
  \prod_{\substack{p|\widetilde W(N)}}
  \tilde{\chi} (p^{v_p(w)}).
\end{align*}

Recalling \eqref{eq:beta_phi} and comparing \eqref{eq:S_h(T,W,A)-expansion} with \eqref{eq:t-main-term} 
suggests to analyse the expression
$$
\sum_{\substack{w_1, \dots, w_r 
 \\ p|w_i \Rightarrow p<w(T)
 \\ w_i \leq (\log T)^{B_2}}}
 \left( \frac{\widetilde W}{\phi( \widetilde W)} \right)^r
 \sum_{\substack{A_1,\dots,A_r \\ \in (\ZZ/\widetilde W \ZZ)^*}}
 \Bigg(\prod_{i=1}^r 
 h^*_i(w_i) \overline{\chi_i}(A_i) \Bigg)~
 \frac{1}{(w\widetilde W)^s}
 \sum_{\substack{\v \in \\ (\ZZ/w \widetilde W \ZZ)^s}}
 \prod_{j=1}^r 
 \1_{\vphi_j(\v) \equiv w_j A_j ~~(w_j \widetilde W)}.
$$
By the Chinese remainder theorem, the inner sum of this expression 
is also multiplicative and we may, in particular, factorise the summation condition 
${\vphi_j(\v) \equiv w_j A_j \Mod{w_j \widetilde W}}$
into congruences modulo prime powers.
To handle the new summation conditions that arise, we will invoke the notion of divisor densities
already introduced in \eqref{eq:def-alpha}.
Thus, if $\bphi = (\vphi_1, \dots, \vphi_r)$ is a system of linear forms,
if $\mathbf{c}=(c_1, \dots, c_r) \in \NN_0^r$ and 
$m = \max \{ c_1, \dots, c_r \}$, we have
$$
\alpha_{\bphi}(p^{c_1},\dots,p^{c_r}) := 
\frac{1}{p^{ms}} \sum_{\mathbf u\in (\ZZ/p^m \ZZ)^s}
\prod_{i=1}^{r} \mathbf 1_{p^{c_i}\mid \vphi_i(\mathbf{u})}.
$$
Recall further that these quantities can be asymptotically evaluated and satisfy \eqref{eq:div_density_bounds}.
Extending $\alpha_{\boldsymbol{\vphi}}$ multiplicatively, it follows from \eqref{eq:div_density_bounds} that
\begin{align*}
 \alpha_{\boldsymbol{\vphi}}(n_1,\dots,n_r) 
 \ll_L (\lcm(n_1, \dots, n_r))^{-1} 
 \ll_L (\max_j n_j)^{-1}.
\end{align*}

We will use the decomposition
$$
\sum_{\substack{w_1, \dots, w_r 
 \\ p|w_i \Rightarrow p| \widetilde W(N)
 \\ w_i \leq (\log N)^{B_2}}}
= \sum_{\substack{w_1, \dots, w_r 
 \\ p|w_i \Rightarrow p| \widetilde W(N)}}
- \sum_{\substack{w_1, \dots, w_r 
 \\ p|w_i \Rightarrow p|\widetilde W(N)
 \\ \exists j. w_j > (\log N)^{B_2}}} 
$$
together with the bound
\begin{align} \label{eq:**}
&\sum_{\substack{w_1, \dots, w_r 
 \\ p|w_i \Rightarrow p|\widetilde W(N)
 \\ \exists j. w_j > (\log N)^{B_2}}}  
 \sum_{\substack{A_1,\dots,A_r \\ \in (\ZZ/\widetilde W \ZZ)^*}}
 \Bigg(\prod_{i=1}^r 
 |h^*_i(w_i) \overline{\chi_i}(A_i)| \Bigg)~
 \frac{1}{(w\widetilde W)^s}
 \sum_{\substack{\v \in \\ (\ZZ/w \widetilde W \ZZ)^s}}
 \prod_{j=1}^r 
 \1_{\vphi_j(\v) \equiv w_j A_j \Mod{w_j \widetilde W}} \\
\nonumber 
&\ll  \sum_{\substack{w_1, \dots, w_r 
 \\ p|w_i \Rightarrow p|\widetilde W(N)
 \\ \exists j. w_j > (\log N)^{B_2}}}  
 \prod_{p|\widetilde W(N)}  |h_1(w_1) \dots h(w_r) )| ~
 \alpha_{\bphi}(w_1, \dots, w_r) \\
\nonumber 
&\ll_{L,\eps} 
\sum_{\substack{w_1, \dots, w_r 
 \\ p|w_i \Rightarrow p|\widetilde W(N)
 \\ \exists j. w_j > (\log N)^{B_2}}}  
 (\max_j w_j)^{-1 + \eps}
  \ll_{L, \eps} 
 \sum_{\substack{w > (\log N)^{B_2}
 \\ p|w \Rightarrow p|\widetilde W(N)}} 
 \prod_{p|w} (v_p(w))^r ~ p^{-(1 - \eps)  v_p(w)} .
\end{align}
To bound the above, we make use of the following inequalities.
For all $a \geq 1$, $r \geq 1$ and $\eps >0$ we have
\begin{align*}
a^r \leq p^{a \eps} \quad \text{if} \quad p > r^{r/\eps} \qquad \text{and }
\qquad
a^r \ll_{r,\eps} p^{a \eps} \quad \text{for all } p.
\end{align*}
To prove the first bound, note that if 
$p > r^{r/\eps}$ and $1 \leq a \leq r$, then $a^r \leq r^r \leq p^{\eps} \leq p^{a \eps}$.
Furthermore, if $p > r^{1/\eps}$ and $2 \leq r < a$, then 
$a^r < r^a < p^{a \eps}$. 
The second bound can be proved with implied constant given by 
$C^*_{r,\eps}=\max_{a \leq C_{r,\eps}}(C_{r,\eps}, a^r)$, where $C_{r,\eps}=\frac{2^r r}{2^\eps -1}$.
Indeed, $a^r \leq C^*_{r,\eps} 2^{\eps a}$ holds for all $a \leq C_{r,\eps}$.
Further, if $a \geq C_{r,\eps} =\frac{2^r r}{2^\eps -1}$ and $a^r \leq C_{r,\eps} 2^{\eps a}$, then
$$\left(\frac{a+1}{a}\right)^r 
= (1 + a^{-1})^r \leq 1 + 2^r \frac{r}{a}
<  2^{\eps},$$
and hence $(a+1)^r \leq 2^{\eps} a^r \leq C_{r,\eps} 2^{\eps}2^{\eps a} = C_{r,\eps} 2^{\eps (a+1)}$, as required.

Setting $\eps = 1/8$, we deduce from the above two bounds that \eqref{eq:**} is bounded by
\begin{align} \label{eq:**-bound}
\nonumber
& \ll_{L, \eps} 
 \sum_{\substack{w > (\log N)^{B_2}
 \\ p|w \Rightarrow p|\widetilde W(N)}} 
 \prod_{p|w} (v_p(w))^r ~ p^{-(1 - \eps)  v_p(w)} 
\ll_{L, \eps} 
 \sum_{\substack{w > (\log N)^{B_2}
 \\ p|w \Rightarrow p|\widetilde W(N)}} 
 \prod_{p|w} p^{-(1 - 2\eps)  v_p(w)} \\ 
\nonumber 
&\ll_{L} 
 \sum_{\substack{w > (\log N)^{B_2}
 \\ p|w \Rightarrow p|\widetilde W(N)}} 
 w^{-3/4}
 \ll_{L} 
 \sum_{d > (\log N)^{B_2/3}} d^{-3/2}
 \prod_{p|\widetilde W(N)}
 \left(1 + p^{-3/4} \right) \\
&\ll_{L} (\log N)^{-B_2/6}~ 2^{\omega(\widetilde W(N))} 
\ll_{L} (\log N)^{-B_2/6 + o(1)},
\end{align}
since 
$\omega(\widetilde W(N)) \ll \pi(\log \log N) + \sum_{p|\widetilde{W}(N), p>\log \log N}1 
\ll (\log \log N)(\log \log \log N)^{-1}$.

Note that by the bound \eqref{eq:**-bound} on \eqref{eq:**}, we have
\begin{align} \label{eq:CRT-appl}
\nonumber
&\sum_{\substack{w_1, \dots, w_r 
 \\ p|w_i \Rightarrow p| \widetilde W(N)
 \\ w_i \leq (\log N)^{B_2}}}
 \left( \frac{\widetilde W}{\phi( \widetilde W)} \right)^r
 \sum_{\substack{A_1,\dots,A_r \\ \in (\ZZ/\widetilde W \ZZ)^*}}
 \Bigg(\prod_{i=1}^r 
 h_i^*(w_i) \overline{\chi_i}(A_i) \Bigg)~
 \frac{1}{(w\widetilde W)^s}
 \sum_{\substack{\v \in \\ (\ZZ/w \widetilde W \ZZ)^s}}
 \prod_{j=1}^r 
 \1_{\vphi_j(\v) \equiv w_j A_j ~(w_j \widetilde W)} \\
\nonumber
&=(1 + o(1)) \times \\
& \nonumber
\sum_{\substack{w_1, \dots, w_r 
 \\ p|w_i \Rightarrow p | \widetilde W(N)}}
 \left( \frac{\widetilde W}{\phi( \widetilde W)} \right)^r
 \sum_{\substack{A_1,\dots,A_r \\ \in (\ZZ/\widetilde W \ZZ)^*}}
 \Bigg(\prod_{i=1}^r 
 h_i^*(w_i) \overline{\chi_i}(A_i) \Bigg)~
 \frac{1}{(w\widetilde W)^s}
 \sum_{\substack{\v \in \\ (\ZZ/w \widetilde W \ZZ)^s}}
 \prod_{j=1}^r 
 \1_{\vphi_j(\v) \equiv w_j A_j ~(w_j \widetilde W)} \\
&=(1 + o(1)) \times \\
& \nonumber 
\sum_{\substack{w_1, \dots, w_r 
 \\ p|w_i \Rightarrow p | \widetilde W(N)}}
 \left( \frac{\widetilde W}{\phi( \widetilde W)} \right)^r
 \Bigg(\prod_{i=1}^r 
 h_i^*(w_i) \widetilde{\chi_i}(w_i) \Bigg)~
 \frac{1}{(w\widetilde W)^s}
 \sum_{\substack{\v \in \\ (\ZZ/w \widetilde W \ZZ)^s}}
 \prod_{j=1}^r \overline{\widetilde{\chi_j}(\vphi_j(\v))}
 \1_{w_j\| \vphi_j(\v)},
\end{align}
where, as before, $\widetilde{\chi_i}(p) = \chi_i(p)$ if $p$ does not divide the conductor of $\chi_i$, and 
$\widetilde{\chi_i}(p) = 1$ otherwise.
Letting $q_j$ denote the conductor of $\chi_j$, we now decompose 
$\chi_j = \prod_{p} \chi_{j,p}$ into characters modulo $p^{v_p(q_j)}$ and define for each new factor
$\widetilde\chi_{j,p}$ as before.
Then $$\widetilde\chi_j(m) = \prod_{p|q_j} \widetilde\chi_{j,p} (m_p)$$ whenever $m \equiv m_p \Mod{p^{v_p(q_j)}}$ 
for all $p|q_j$.
Using the Chinese remainder theorem, this allows us to factorise \eqref{eq:CRT-appl} as follows.
Let $Q$ be the product of all primes $p<B$ and of the conductors of all the $\chi_i$,
and recall the definition of $\beta_p(\chi_1, \dots, \chi_r)$ from the statement of the theorem.
Then the main term of \eqref{eq:CRT-appl} equals
\begin{align} \label{eq:factorisation}
& \prod_{p|\widetilde{W}(N)}
\beta_p(\chi_1, \dots, \chi_r) \prod_{j=1}^r\sum_{k\geq0} \frac{|h_j(p^k)|}{p^k} \\
\nonumber 
&=
\prod_{\substack{p| \widetilde W(N)
\\ p \nmid Q }}
\sum_{\substack{ a_1, \dots, a_r \\ \in \NN_0 }}
\prod_{i=1}^r \frac{h_i^*(p^{a_i}) \chi_i(p^{a_i})}{1-p^{-1}}
\left(
\lim_{m \to \infty} 
\frac{1}{p^{ms}}
\sum_{\v \in (\ZZ/p^{m}\ZZ)^s} \prod_{j=1}^r
\left( \1_{p^{a_j}|\vphi_j(\v)} - \1_{p^{a_j+1}|\vphi_j(\v)} \right) \right) \times\\
\nonumber
& \qquad \times
\prod_{\substack{p'| Q }}
\sum_{\substack{ a_1, \dots, a_r \\ \in \NN_0 }}
\prod_{i=1}^r \frac{h_i^*(p'^{a_i}) \widetilde\chi_{i,p'}(p'^{a_i})}{1-p'^{-1}}
\left(
\lim_{m \to \infty} 
\frac{1}{p'^{ms}}
\sum_{\v \in (\ZZ/p'^{m}\ZZ)^s}
\prod_{i=1}^r \overline{\widetilde\chi_{i,p'}}(\vphi_i(\v))
 \1_{p'^{a_i}\|\vphi_i(\v)}  \right)
\end{align}
In view of \eqref{eq:def-alpha} and \eqref{eq:div_density_bounds}, let $B>1$ 
be sufficiently large in terms of $L$ so that the second and third bound of \eqref{eq:div_density_bounds} apply 
to every $p \geq B$, and suppose also that $B > 2 r H^r$.

For every $p \nmid Q$ (not assuming $p|\widetilde W(N)$), we then have:
\begin{align} \label{eq:p>B'-factor}
\nonumber
&\beta_p(\chi_1, \dots, \chi_r) \prod_{j=1}^r\sum_{k\geq0} \frac{|h_j(p^k)|}{p^k}\\
\nonumber
&=\sum_{\substack{a_1, \dots, a_r \\  \in \NN_0}}
\prod_{i=1}^r \frac{h_i^*(p^{a_i}) \chi_i(p^{a_i})}{1 - p^{-1}}
\left(
\lim_{m \to \infty} 
\frac{1}{p^{ms}}
\sum_{\v \in (\ZZ/p^{m}\ZZ)^s}
\prod_{j=1}^r
\left( \1_{p^{a_j}|\vphi_j(\v)} - \1_{p^{a_j+1}|\vphi_j(\v)} \right) \right)\\
&= \sum_{\substack{a_1, \dots, a_r \\  \in \NN_0}}
\prod_{i=1}^r \frac{h_i^*(p^{a_i}) \chi_i(p^{a_i})}{1 - p^{-1}}
\sum_{\beps \in \{0,1\}^r}
(-1)^{n(\beps)} 
\alpha_{\bphi}(p^{a_1 + \eps_1},\dots,p^{a_r+\eps_r}).
\end{align}
The expression \eqref{eq:p>B'-factor} can now be asymptotically evaluated with the help of \eqref{eq:div_density_bounds}.
Note that whenever the bound \eqref{eq:div_density_bounds} on 
$\alpha_{\bphi}(p^{a_1 + \eps_1},\dots,p^{a_r+\eps_r})$
takes the form $p^{-k}$ or $O_L(p^{-k})$ for a given $k > 1$, 
then there are at most $2^rk^r$ admissible choices of 
$(a_1, \dots, a_r)$ and $(\eps_1, \dots, \eps_r)$, and for each of these choices we have
$|h_i^*(p^{a_i+\eps_i})| < H^k$ for every $i \in \{1, \dots, r\}$.
Thus, the expression \eqref{eq:p>B'-factor} equals
\begin{align} \label{eq:p>B'-factor-final}
\nonumber
& \left(1 - \frac{1}{p} \right) ^{-r} 
\left( 1 - \frac{r}{p} + \sum_{i=1}^r \frac{h_i^*(p) \chi_i(p)}{p} + 
O \left( \sum_{k\geq2} \frac{k^r H^{rk}}{p^k} \right) \right)\\
&= \prod_{i=1}^r 
   \left(1 + \frac{h_i^*(p) \chi_i(p)}{p} \right) 
   + O_{H,r}(p^{-2}) + O \left( \sum_{k\geq2} \frac{k^r H^{rk}}{p^k} \right).
\end{align}
The second error term above also equals $O_{H,r}(p^{-2})$.
To see this, we use the fact that $k^r \leq r^k$ whenever $k \geq r \geq 2$. 
This inequality follows easily by induction: the base case holds trivially, and if $k^r \leq r^k$
for a given $k \geq r$, then
$$(k+1)^{r} 
\leq k^r \left(1 + \frac{1}{k}\right)^r 
\leq k^r \sum_{j=0}^r \binom{r}{j} k^{-j}
\leq k^r \sum_{j=0}^r \frac{1}{j!}
\leq k^r \sum_{j=0}^r \frac{1}{2^j}
< 2 k^r 
\leq r k^r \leq r^{k+1}.
$$
Recalling that $p>B>2 r H^r$, the sum in the second error term above thus satisfies
$$
\sum_{k\geq2} \frac{k^r H^{rk}}{p^k}
\leq \sum_{2 \leq k < r} \frac{k^r H^{rk}}{p^k}
+ \sum_{k\geq r} \left( \frac{r H^r}{p} \right)^k
= O_{H,r}(p^{-2}),
$$
as claimed.
Comparing with \eqref{eq:p>B'-factor},
we deduce that for $p\nmid Q$ we have
\begin{align} \label{eq:beta_p-asymp}
\nonumber
\beta_p(\chi_1, \dots, \chi_r)
&= 
\prod_{j=1}^r
\left(1+\frac{|h_j(p)|}{p} \right)^{-1}
\left(1+\frac{h_j^*(p)\chi_j(p)}{p} \right)
+ O_{H,r}(p^{-2})\\
&= 
\prod_{j=1}^r
\left(1+\frac{h_j^*(p)\chi_j(p) - |h_j(p)|}{p} \right)
+ O_{H,r}(p^{-2}).
\end{align}
confirming \eqref{eq:beta_p;p-nmid-Q}.

Observe further that for $p\leq N$ with $p \nmid \widetilde{W}(N)$,
\begin{align} \label{eq:beta_p-large-p}
\nonumber
  \frac{\sum_{k:p^k \leq N} h_i^*(p^k)\chi_i^*(p^k)p^{-k}}{\sum_{k:p^k \leq N} |h_i(p^k)|p^{-k}}
 &= \prod_{j=1}^r
\left(1+\frac{h_j^*(p)\chi_j(p) - |h_j(p)|}{p} \right)
+ O_{H,r}(p^{-2})\\
 &= \beta_p(\chi_1, \dots, \chi_r) +  O_{H,r}(p^{-2})
\end{align}

To complete the proof of Theorem \ref{t:main'}, we insert \eqref{eq:S_h(T,W,A)-expansion}
into \eqref{eq:t-main-term} and then appeal to \eqref{eq:CRT-appl},
\eqref{eq:factorisation}, \eqref{eq:beta_p-asymp} and \eqref{eq:beta_p-large-p} to deduce that:
\begin{align*}
&\sum_{\substack{w_1, \dots, w_r 
 \\ p|w_i \Rightarrow p|\widetilde W(N)
 \\ w_i \leq (\log N)^{B_2}}}
\sum_{\substack{A_1,\dots,A_r \in \\ (\ZZ/\widetilde W(N) \ZZ)^*}}
\Bigg(\prod_{i=1}^r 
h_i^*(w_i) S_{h_i^*}\Big(N;\widetilde W,A_i\Big)\Bigg)~
\frac{1}{(w\widetilde W)^s}
\sum_{\substack{\v \in \\ (\ZZ/w \widetilde W \ZZ)^s}}
\prod_{j=1}^r 
\1_{\vphi_j(\v) \equiv w_j A_j \Mod{w_j \widetilde W}} \\
&= (1+o(1))
\sum_{\substack{\chi_1, \dots, \chi_r \\ \chi_j \in \mathcal{E}_j^+}} 
\prod_{i=1}^r 
S_{|h_i|}(N) 
\prod_{p\leq N}
 \left(
 \frac{\sum_{p^k \leq N} h_i^*(p^k)\chi_i^*(p^k)p^{-k}}{\sum_{p^k \leq N} |h_i(p^k)|p^{-k}}\right)
\\
&\qquad \qquad \qquad \qquad \times\prod_{p' | \widetilde W(N)} 
\beta_{p'}(\chi_1, \dots, \chi_r)
\prod_{j=1}^r
\left(\sum_{p'^k \leq N} |h_j(p'^k)|p'^{-k}\right) \\
&= (1+o(1))
\sum_{\substack{\chi_1, \dots, \chi_r \\ \chi_j \in \mathcal{E}_j^+}} 
\prod_{i=1}^r 
S_{|h_i|}(N) 
\prod_{\substack{p\leq N \\ p\nmid \widetilde{W}(N)}}
 \left(
 \frac{\sum_{p^k \leq N} h_i^*(p^k)\chi_i(p^k)p^{-k}}{\sum_{p^k \leq N} |h_i(p^k)|p^{-k}}\right)
 \prod_{p' | \widetilde W(N)} 
\beta_{p'}(\chi_1, \dots, \chi_r)
\\
&= (1+o(1))
\sum_{\substack{\chi_1, \dots, \chi_r \\ \chi_j \in \mathcal{E}_j^+}} 
\prod_{i=1}^r 
S_{|h_i|}(N) 
\prod_{\substack{p\leq N \\ p\nmid \widetilde{W}(N)}}
 \left(
 \beta_{p}(\chi_1, \dots, \chi_r)+ O_{H,r}(p^{-2})\right)
 \prod_{p' | \widetilde W(N)} 
\beta_{p'}(\chi_1, \dots, \chi_r)
\\
&= (1+o(1))
\prod_{i=1}^r 
S_{|h_i|}(N) 
\sum_{\substack{\chi_1, \dots, \chi_r \\ \chi_j \in \mathcal{E}_j^+}} 
\prod_{\substack{p\leq N }}
 \beta_{p}(\chi_1, \dots, \chi_r),
\end{align*}
proving the first asymptotic expansion \eqref{eq:expansion-1} from the theorem.

Concerning the second asymptotic expansion, observe that
\begin{align*}
\nonumber
&
\prod_{\substack{p\leq N }}
 \beta_{p}(\chi_1, \dots, \chi_r)
&=
\beta_{Q}(\chi_1, \dots, \chi_r)
\prod_{p \leq N, p \nmid Q} \left(1 + \frac{h_i(p) \chi_i(p) - |h_i(p)|}{p} \right) 
\left(1 + O_{H,r} \left(p^{-2} \right)\right) ,
\end{align*}
and that, taking into account that $W(B)|Q$, we have 
\begin{align*}
\prod_{p\nmid Q} 
\left(1 + O_{H,r} \left(p^{-2} \right)\right)
&= \exp \bigg(\sum_{p \nmid Q} O_{H,r} \left(p^{-2} \right) (1 + O_{H,r}(p^{-4})) \bigg)
= \exp O_{H,r} \left(B^{-1} \right) \\
&= 1 + O_{H,r} \left(B^{-1} \right),
\end{align*}
provided $B$ is sufficiently large with respect to $H$ and $r$.
Hence, \eqref{eq:expansion-2} follows.
This completes the proof of Theorem \ref{t:main'}.
\end{proof}

\begin{proof}[Proof of Corollary \ref{cor:chi_0}]
 This is just a special case of Theorem \ref{t:main'}.
\end{proof}

\begin{proof}[Proof of Corollary \ref{c:pret}]
 Assuming that $\mathcal{E}_{j,N}^+ = \{\chi_j\}$ for $1 \leq j \leq r$ and all sufficiently large $N$,
 the asymptotic formula is, again, an immediate consequence of Theorem \ref{t:main'}. 
 To show that $\mathcal{E}_{j,x}^+ = \{\chi_j\}$ we will make use of the fact that characters are repulsive 
 in the sense of \cite[\S 3]{BGS}.
 More precisely, \cite[Lemma 3.1]{BGS} implies that 
 $$\sum_{p\leq x} \frac{1 - \Re(h_j(p) \chi(p) p^{it})}{p} 
 \geq \Big(1 - \frac{1}{\sqrt{2}}\Big) \log \log x + O(\sqrt{\log \log x})$$
 for all $t$, $|t| \leq (\log x)^2$ and for every primitive $\chi \not= \chi_j$, and hence
 $$\sum_{p\leq x} \frac{1 - \Re(h_j(p) \chi^*(p) p^{it})}{p} 
 \geq \Big(1 - \frac{1}{\sqrt{2}}\Big) \log \log x + 
 O(\log \log (\log x)^C ) + 
 O(\sqrt{\log \log x})$$
 whenever $\chi^* \Mod{q}$ is induced from $\chi$ and $q \leq (\log x)^C$.
 Hence, it follows from \cite[Corollaire 2.1]{tenen} with $r=\1$ and
 $f(n) = h_j(n) \chi^*(n) n^{it}$ that
 $$
 \left|\frac{1}{x} \sum_{n \leq x} h_j(n) \chi^*(n) n^{it} \right|
 \ll_{\eps} (\log x)^{-1 + \frac{1}{\sqrt{2}} + \eps},
 $$
 and, thus, \eqref{eq:tenen-assumpt-2} holds for $\chi \in \mathcal{E}_j \setminus \{ \chi_j \}$.
 (Alternatively, we could have directly applied \cite[Theorem 2]{BGS}.)
 Concerning \eqref{eq:tenen-assumpt-1}, we apply \cite[Theorem 1.3]{tenen}, this time with $r=|h_j|$.
 Since 
 \begin{align} \label{eq:pret-convergence}
   \sum_{p}\frac{|h_j(p)|-\Re(h_j(p)\chi_j(p)p^{it_j})}{p} 
 \leq \sum_{p} \frac{1-\Re(h_j(p)\chi_j(p)p^{it_j})}{p} < \infty,
 \end{align}
 the condition \cite[eq.\ (1.15)]{tenen} holds for any valid choice of constants 
 $\beta$ and $\mathfrak{b}$, provided $x \mapsto \eps_x \in ((\log x)^{-1/2},1/2]$ 
 is a function that tends to $0$ as $x \to \infty$ and provided $x$ is sufficiently large. 
 Thus, let $\mathfrak{a}=\mathfrak{b}=1/c$ for some fixed $c \in \NN$ with $c \geq 4$. 
 Then we need to check \cite[eq.\ (1.16)]{tenen} with 
 $\mathfrak{h}=(1-\mathfrak{b})/\mathfrak{b} = c-1$.
 Note that, since $|h_j(p)| \leq 1$ for all $p$, 
 $$
 \sum_{x^{\eps} < p \leq y} \frac{(|h_j(p)|-\Re(h_j(p)\chi_j(p)p^{it_j}))^{c-1} \log p}{p}
 \leq 2^{c-2} \log y \sum_{x^{\eps} < p \leq y} \frac{|h_j(p)|-\Re(h_j(p)\chi_j(p)p^{it_j})}{p}.
 $$
 From \eqref{eq:pret-convergence} it follows that
 $$\sum_{x^{1/\sqrt{\log x}} < p \leq x} \frac{|h_j(p)|-\Re(h_j(p)\chi_j(p)p^{it_j})}{p} \to 0$$
 as $x \to \infty$, allowing us to choose a function $\eps_x$ for which \cite[eq.\ (1.16)]{tenen} holds
 for $\mathfrak{h}=c-1$ and for which \cite[eq.\ (1.19)]{tenen} holds for $\mathfrak{h}=1$.
 Finally, since $\sum_p (1-|h_j(p)|)/p < \infty$ we deduce
 that $1 - |h_j(p)| < 1/2$ for all but $o(x/\log x)$ primes $p \in [x,2x)$ as $x \to \infty$.
 Since $1/\sqrt{\eps_x} \to \infty$ as $x \to \infty$,
 condition \cite[eq.\ (1.24)]{tenen} clearly holds for all sufficiently large $x$, provided 
 $c = 1/ \mathfrak{b}>8$.
 Thus, \cite[Theorem 1.3]{tenen} applies and yields \eqref{eq:tenen-assumpt-1} for $\chi=\chi_j$.
 It follows that $\mathcal{E}_{j,x}^+ = \{\chi_j\}$, as required.
\end{proof}

\section{Application to eigenvalues of cusp forms}\label{s:applications}
In this section we show that both Theorem \ref{t:main} and Corollary \ref{cor:chi_0} apply 
to the normalised eigenvalues of holomorphic cusp forms.
Given a primitive holomorphic cusp form $f$ of weight $k\in 2\NN$ and level $N \in \NN$ 
(see \cite[\S14.1 and \S14.7]{IK} for definitions), let
$$
f(z)=
\sum_{n=1}^{\infty}
\lambda_f(n) n^{(k-1)/2}e(nz)
$$
be its Fourier expansion, where the $\lambda_f(n)$ are the normalised Fourier coefficients.
As proved in \cite[Lemma 4.15]{lmm}, the function $h: n \mapsto |\lambda_f(n)|$ belongs to $\mathcal{F}$; 
the additional property~(ii) in Definition \ref{d:M} follows from the divisor function bound
$|\lambda_f(n)| \leq d(n)$ due to Deligne.
Thus, we may apply Theorem \ref{t:main} with $h_j = |\lambda_{f_j}|$ for cusp forms $f_j$ as above.
Moreover, the following holds:
\begin{lemma}
Let $f_1, \dots, f_r$ be primitive holomorphic cusp forms of even integral weight and denote by 
$h_j(n)=\lambda_{f_j}(n)$ their respective normalised Fourier coefficients.
Let $\mathfrak{K} \subset [-1,1]^s$ and $\vphi_1, \dots \vphi_r \in \ZZ[X_1, \dots, X_s]$ be as in Theorem \ref{t:main}.
Then, as $T \to \infty$, we have
\begin{align*}
\frac{1}{\vol T\fK}
&\sum_{\n \in \ZZ^s \cap T\fK} 
\prod_{j=1}^r 
|\lambda_{f_j}(\vphi_j(\n))| \\
&=
\Bigg(
\prod_{i=1}^r 
S_{|h_i|}(T)\Bigg)
\prod_{p} \beta_p
+ o\Bigg(
\frac{1}{(\log T)^r} 
\prod_{j=1}^r  \prod_{p \leq T} 
\left(1 + \frac{|\lambda_{f_j}(p)|}{p} \right)
\Bigg)
~, 
\end{align*}
where $\beta_p=\beta_p(\chi_0, \dots, \chi_0)$ is given by \eqref{eq:beta_P(B)-cor} and satisfies
$
\beta_p = 1 + O_{r}(p^{-2}),
$
and where
$$S_{|h_j|}(T)
=\frac{1}{T} \sum_{n\leq N} |\lambda_{f_j}(n)|
\asymp (\log T)^{-1} \prod_{p \leq T} (1 + |\lambda_{f_j}(p)|p^{-1} ).
$$
\end{lemma}

\begin{proof}
The proof of \cite[Lemma 4.15]{lmm} does not only show that $|\lambda_f(n)| \in \mathcal{F}$, but it yields 
the following:
if $\tilde h$ denotes the multiplicative function defined by
$$
\tilde h(p^k) = 
\begin{cases}
 |\lambda_f(p)|/2 &\text{if } k=1, \\
 0 &\text{if } k \geq 2,
\end{cases}
$$
then there exists a constant $c>0$ such that, given any constant $C>1$, we have 
\begin{equation*}
S_{\tilde h\chi}(T) 
 \ll \frac{1}{(\log T)^{1+c \alpha_h}} \prod_{p \leq T, p \nmid q} \Big(1 + \frac{\tilde h(p)}{p} \Big),
\end{equation*}
uniformly for all sufficiently large $T$ and all non-trivial characters $\chi \Mod{q}$ with $1 < q \leq (\log T)^C$
and $W(T)|q$.
In particular this implies that, given any $C>1$, we have
\begin{equation} \label{eq:cusps-1}
 S_{\tilde h\chi}(T) 
 \ll \frac{1}{(\log N)^{1+c \alpha_h}} \prod_{p \leq N, p \nmid q} \Big(1 + \frac{\tilde h(p)}{p} \Big),
\end{equation}
uniformly for all sufficiently large $T$, all $T \leq N \leq T^{8}$ and all non-trivial characters 
$\chi \Mod{q}$ with $1 < q \leq (\log N)^{C}$ and $W(N)|q$.
The estimate \eqref{eq:cusps-1} can be related to a bound on $S_{|\lambda_f| \chi}(T)$ via \cite[Lemma 4.8]{lmm}
\footnote{$y$ should be $x$ in the first equation display of that lemma.}, 
which shows that 
$S_{|\lambda_f| \chi}(T) = o(E_{|\lambda_f|}(N;q))$, uniformly for all $q$ as above and all $T\in [N^{1/2},N]$.
Finally, we apply \cite[Corollary 4.2]{lmm} to deduce that 
$$
S_{|\lambda_f|}(T; q, A) =  
\frac{q}{\phi(q)}
  \overline{\chi_0}(A) \frac{1}{T} \sum_{n\leq T} |\lambda_f(n)| \chi_0(n)
 + o(E_{|\lambda_f|}(N;q)),
$$
uniformly all $q$ as above and $T\in [N^{1/2},N]$, and where $\chi_0 \Mod{q}$ is the trivial character.
Thus, setting $q = \widetilde W(N) \leq (\log N)^{B_1}$ with $\widetilde W $ as in the statement of 
Theorem~\ref{t:main} (in fact, we have $\widetilde W (N) = W(N)$ in this case), 
the conditions of Corollary \ref{cor:chi_0} are seen to be satisfied and the lemma follows.
\end{proof}

\begin{ack}
It is a pleasure to thank Stephen Lester for very helpful discussions and 
R{\'e}gis de la Bret{\`e}che and Yuri Tschinkel for their questions which led to 
this work.
I am very grateful to the referee for their detailed comments.
\end{ack}

\end{document}